\documentclass[10pt,reqno]{amsart}

\usepackage[a4paper, margin=1.5in]{geometry}

\usepackage {color}
\usepackage{amsmath,amsthm,amssymb}
\usepackage{epsfig}
\usepackage{graphicx}
\usepackage{amssymb}
\usepackage{latexsym}
\usepackage{pdfpages}
\usepackage{tikz}

\usepackage{graphicx}
\usepackage{caption,subcaption}
\usepackage{amssymb}
\usepackage{latexsym}
\usepackage{pdfpages}

\usepackage{pgfplots}
\pgfplotsset{compat=1.15}
\usepackage{mathrsfs}
\usetikzlibrary{arrows}

\usepackage{pgf}

\usepackage{ulem} 

\usepackage{ifpdf}
\newcommand{\res}{\!\!\mathop{\hbox{
                                \vrule height 7pt width .5pt depth 0pt
                                \vrule height .5pt width 6pt depth 0pt}}
                                \nolimits}

\def\z{{\bf z}}

\def\g{{\bf g}}

\ifpdf 
  \usepackage[hidelinks]{hyperref}
\else 
\fi 

 \usepackage[usenames,dvipsnames]{pstricks}
 \usepackage{epsfig}
 \usepackage{pst-grad} 
 \usepackage{pst-plot} 

\newtheorem{theorem}{Theorem}[section]
\newtheorem{lemma}[theorem]{Lemma}
\newtheorem{definition}[theorem]{Definition}
\newtheorem{proposition}[theorem]{Proposition}

\newtheorem{remark}[theorem]{Remark}
\newtheorem{example}[theorem]{Example}

\newtheorem*{theorem*}{\it Theorem}

\def\vint_#1{\mathchoice%
          {\mathop{\kern 0.2em\vrule width 0.6em height 0.69678ex depth -0.58065ex
                  \kern -0.8em \intop}\nolimits_{\kern -0.4em#1}}%
          {\mathop{\kern 0.1em\vrule width 0.5em height 0.69678ex depth -0.60387ex
                  \kern -0.6em \intop}\nolimits_{#1}}%
          {\mathop{\kern 0.1em\vrule width 0.5em height 0.69678ex
              depth -0.60387ex
                  \kern -0.6em \intop}\nolimits_{#1}}%
          {\mathop{\kern 0.1em\vrule width 0.5em height 0.69678ex depth -0.60387ex
                  \kern -0.6em \intop}\nolimits_{#1}}}
\def\vintslides_#1{\mathchoice%
          {\mathop{\kern 0.1em\vrule width 0.5em height 0.697ex depth -0.581ex
                  \kern -0.6em \intop}\nolimits_{\kern -0.4em#1}}%
          {\mathop{\kern 0.1em\vrule width 0.3em height 0.697ex depth -0.604ex
                  \kern -0.4em \intop}\nolimits_{#1}}%
          {\mathop{\kern 0.1em\vrule width 0.3em height 0.697ex depth -0.604ex
                  \kern -0.4em \intop}\nolimits_{#1}}%
          {\mathop{\kern 0.1em\vrule width 0.3em height 0.697ex depth -0.604ex
                  \kern -0.4em \intop}\nolimits_{#1}}}

\def\R{\mathbb R}

\def\g{\hbox{\bf g}}

\numberwithin{equation}{section}

\def\1{\raisebox{2pt}{\rm{$\chi$}}}

\def\g{{\bf g}}

\def\z{{\bf z}}

\definecolor{violet(ryb)}{rgb}{0.53, 0.0, 0.69}
\usepackage{mathtools}
\mathtoolsset{showonlyrefs}

\begin{document}

\title[Evolution problems with perturbed $1$-Laplacian type operators]
{\bf Evolution problems with perturbed $1$-Laplacian type operators  on random walk spaces}

\author{ W. G\'{o}rny, J. M. Maz\'{o}n and J. Toledo}

\address{Wojciech G\'{o}rny
\hfill\break\indent Faculty of Mathematics, Universit\"at Wien
\hfill\break\indent Oskar-Morgerstern-Platz 1, 1090 Vienna, Austria
\hfill\break\indent Faculty of Mathematics, Informatics and Mechanics, University of Warsaw
\hfill\break\indent Banacha 2, 02-097 Warsaw, Poland
\hfill\break\indent
{\tt  wojciech.gorny@univie.ac.at }
}

\address{ Jos\'{e} M. Maz\'{o}n
\hfill\break\indent Departamento de An\'{a}lisis Matem\'{a}tico,
Universitat de Val\`encia \hfill\break\indent Valencia, Spain.}
\email{{\tt  mazon@uv.es}}

\address{ Juli\'{a}n Toledo
\hfill\break\indent Departamento de An\'{a}lisis Matem\'{a}tico,
Universitat de Val\`encia \hfill\break\indent Valencia, Spain.}
\email{{\tt toledojj@uv.es}}

\date{}

\keywords{}

\begin{abstract}
Random walk spaces are a general framework for the study of PDEs.
They include as particular cases locally finite weighted connected graphs and nonlocal
settings involving symmetric integrable kernels on $\R^N$. We are interested in the study of evolution
problems involving two   random walk structures so that the associated
functionals have different growth on each structure. We also deal with the case of a functional with different growth on a partition of the random walk.

\end{abstract}

\maketitle

\section{Introduction}

The main focus of this paper is to describe several phenomena in nonlocal partial differential equations with inhomogeneous growth. We work in a general setting of random walk spaces (for a precise definition, we refer to Section \ref{preliminaries}), which includes as two particular cases problems on locally finite weighted graphs (e.g. lattices) and nonlocal problems given by nonnegative non-singular symmetric kernels on $\mathbb{R}^N$. We study PDEs which appear when two random walk structures overlap and the associated functional has different growth on the two structures; we also study the case of a partition of a random walk, where again we have different growth of the associated functional on the two pieces. In other words, we propose a framework to study evolution problems with inhomogeneous growth on random walk spaces.

Our first objective is to study the gradient flow of the functional defined below. Let $[X,\mathcal{B},m^i,\nu_i]$, $i=1,2$, be two random walk spaces. For $p>1$, consider the energy functional with $(1,p)$-growth given by
\begin{equation}\label{0945}\begin{array}{l}
\displaystyle  \frac12\int_{X}\left(\int_{X}|u(y)-u(x)| \, dm^1_x(y) \right) d\nu_1(x)
 + \frac{1}{2p}\int_{X}\left(\int_{X} |u(y)-u(x)|^p \, dm^2_x(y) \right)d\nu_2(x),
\end{array}
\end{equation}
with the value equal to $+\infty$ if either integral is infinite. Assuming that $\nu_2 \ll \nu_1$ and $\mu:=\frac{d\nu_2}{d\nu_1}\in L^\infty(X,\nu_1)$, our   main aim is to study the gradient flow in $L^2(X,\nu_1)$  associated to such functional, which will be a nonlocal evolution problem given by
$$u_t=\Delta^{m^1}_1u+\mu \Delta_p^{m^2}u.$$
This is the content of Section~\ref{section3}; we refer to the introduction of that section for more details. We also consider the functional corresponding to $(q,p)$-growth with $p,q > 1$,~i.e.,
\begin{equation}\label{0945a}
\begin{array}{l}
\displaystyle  \frac{1}{2q} \int_{X}\left(\int_{X}|u(y)-u(x)|^q \, dm^1_x(y) \right) d\nu_1(x)
 + \frac{1}{2p}\int_{X}\left(\int_{X} |u(y)-u(x)|^p \, dm^2_x(y) \right)d\nu_2(x)
\end{array}
\end{equation}
and its gradient flow in $L^2(X,\nu_1)$ given by
$$u_t=\Delta^{m^1}_q u+\mu \Delta_p^{m^2}u.$$
The main effort in the paper concerns the $(1,p)$-growth case, where some additional difficulties appear and the proofs are presented in full; the technique for the $(q,p)$-growth case is very similar and we only show how to adapt the proofs from the $(1,p)$-growth case. Furthermore, we also study the case $p = 1$, for which the aforementioned approach fails and in the proof we rely on a different technique based on homogeneity of the functional. Let us note that the key assumption in all of the above problems is that $\nu_2 \ll \nu_1$ with a bounded Radon-Nikodym derivative, which allows us to pose the evolution problem in~a~joint Hilbert space $L^2(X,\nu_1)$.

The second type of problems we consider is the following. Let $[X,\mathcal{B},m,\nu]$ be a random walk space. Consider nonempty measurable subsets $A_x,B_x\subset  \hbox{supp}(m_x)$ such that
$$\hbox{supp}(m_x)=A_x\cup B_x.$$
In general, the intersection $A_x\cap B_x$ is not assumed to be empty. For $p > 1$, take the energy functional (defined on $L^2(X,\nu)$)
\begin{equation}\label{eldelaintro01}
\mathcal{F}_m(u)=\int_{X}\left(\frac12\int_{A_x}|u(y)-u(x)| \, dm_x(y) + \frac{1}{2p}\int_{B_x}|u(y)-u(x)|^p \, dm_x(y)\right) d\nu(x),
\end{equation}
where we understand that $\mathcal{F}_m(u)=+\infty$ if the integral is not finite. The second of our main objectives is to study the gradient flow associated to the energy functional~\eqref{eldelaintro01} (this is done in Section~\ref{section4}). Here, the situation is a bit simpler since we work in a natural Hilbert space $L^2(X,\nu)$, but another difficulty is posed by the fact that the restrictions of $m_x$ to the sets $A_x$, $B_x$ are not necessarily random walks themselves. To illustrate the type of problems we are going to study, we postpone the discussion on the relation of our results to the literature, and first look at some examples of random walk spaces with two overlapping structures or such partitions of the random walk.

\begin{example}\label{z2ejm}\rm
Consider   the lattice  $X=\mathbb{Z}^2$ with  the random walk given by short-range interactions, with non-zero weights only with the four neighbours of a given point, and possibly different weights in the horizontal and vertical directions. More precisely, for any $n,m\in \mathbb{Z}$, we have a weight $a>0$ between  the vertices $(n,m)$ and $(n\pm 1,m)$, and a weight $b>0$ between  the vertices $(n,m)$ and $(n,m\pm 1)$; the weight is equal to $0$ in any other case. Then,  the random walk $m$ (for the definition see Section \ref{preliminaries})  is given by
$$m_{(n,m)}=\frac{a}{2a+2b}\delta_{(n-1,m)}+\frac{a}{2a+2b}\delta_{(n+1,m})+\frac{b}{2a+2b}\delta_{(n,m-1)}+\frac{b}{2a+2b}\delta_{(n,m+1)},$$
and   the corresponding invariant measure $\nu$ on $\mathbb{Z}^2$ is given by
$$\nu{(n,m)} = 2a+2b.$$
Let us present several possible choices of inhomogeneous problems in this setting.

{\flushleft (i)}  A case with interaction in all directions. In the setting of the functional \eqref{eldelaintro01}, we choose the sets
$$A_{(n,m)}=B_{(n,m)}=\{(n-1,m), (n+1,m), (n,m-1),(n,m+1)\},$$
  so the functional to study becomes
$$\begin{array}{l}
\displaystyle\mathcal{F}_m(u)=  \sum_{(m,n)\in \mathbb{Z}^2}\bigg(\frac12\Big(a|u(n-1,m)-u(n,m)| + a|u(n+1,m)-u(n,m)|
\\[16pt]
\displaystyle \qquad\qquad\qquad\qquad\qquad + b|u(n,m-1)-u(n,m)|  + b|u(n,m+1)-u(n,m)|\Big)
\\[16pt]
\displaystyle \qquad\qquad\qquad\qquad + \frac{1}{2p}\Big(a|u(n-1,m)-u(n,m)|^p+ a|u(n+1,m)-u(n,m)|^p
\\[16pt]
\displaystyle \qquad\qquad\qquad\qquad\qquad +  b|u(n,m-1)-u(n,m)|^p + b|u(n,m+1)-u(n,m)|^p\Big)\bigg).
\end{array}$$
  In this case, the operator in the gradient flow associated to such energy looks like a $1$-Laplacian plus a $p$-Laplacian on the whole lattice. Note that this example also describes the functional \eqref{0945} with the two random walk structures being equal.

{\flushleft (ii)} A case of empty intersection.   In the setting of the functional \eqref{eldelaintro01}, we choose
$$A_{(n,m)}=\{(n-1,m),(n+1,m)\},\ B_{(n,m)}=\{(n,m-1),(n,m+1)\},$$
and the functional to study is
$$\begin{array}{l}
\displaystyle \mathcal{F}_m(u)= \sum_{(m,n)\in \mathbb{Z}^2}\bigg(\frac12\Big(a|u(n-1,m)-u(n,m)| +a|u(n+1,m)-u(n,m)|\Big)
\\[16pt]
  \displaystyle \qquad\qquad\qquad\qquad + \frac{1}{2p}\Big(b|u(n,m-1)-u(n,m)|^p + b|u(n,m+1)-u(n,m)|^p\Big)\bigg).
\end{array}$$
In this case,   the operator in the gradient flow associated to such energy looks like a $1$-Laplacian in the horizontal direction plus a $p$-Laplacian in the vertical direction. Note that if $a = b$, this example may be also described in terms of the functional \eqref{0945}, if we modify the random walk $m$ on the space and consider a random walk $m^1$ given by
$$m^1_{(n,m)}=\frac{1}{2} \delta_{(n-1,m)} + \frac{1}{2} \delta_{(n+1,m)}$$
and a random walk $m^2$ given by
$$m^2_{(n,m)}=\frac{1}{2} \delta_{(n,m-1)} + \frac{1}{2} \delta_{(n,m+1)}.$$

{\flushleft (iii)} A case of nonempty intersection.   In the setting of the functional \eqref{eldelaintro01}, we choose
$$A_{(n,m)}=\{(n-1,m), (n+1,m), (n,m-1),(n,m+1)\},\ B_{(n,m)}=\{(n,m-1), (n,m+1)\},$$
so the functional to study now is
$$\begin{array}{l}
\displaystyle\mathcal{F}_m(u)=  \sum_{(m,n)\in \mathbb{Z}^2}\bigg(
   \displaystyle \frac{1}{2}\Big(a|u(n-1,m)-u(n,m)| + a|u(n+1,m)-u(n,m)|
   \\[16pt]\displaystyle \qquad\qquad\qquad\qquad\qquad +  a|u(n,m-1)-u(n,m)| + a|u(n,m+1)-u(n,m)|\Big)
\\[16pt]\displaystyle \qquad\qquad\qquad\qquad +
\frac1{2p}\Big(b|u(n,m-1)-u(n,m)|^p + b|u(n,m)-u(n,m+1)|^p\Big)
   \bigg).
\end{array}$$
In this case,  that the operator in the gradient flow associated to such energy looks like a $1$-Laplacian on the whole lattice plus a $p$-Laplacian in the vertical direction.
\hfill $\blacksquare$
\end{example}

The inhomogeneous-growth variational problems with $(q,p)$-type growth and corresponding PDEs, for which the studied functional can be estimated from above and below by two power-type functions with different exponents, have been rigorously studied since the works of Marcellini \cite{Mar1,Mar2}. Evolution problems of this type have also been studied, starting with the works of B\"ogelein, Duzaar and Marcellini \cite{BDM1,BDM2}. The literature on inhomogeneous-growth problems is now vast (for an overview see \cite{Chl} or \cite{MR}), but the usual underlying assumption is that the growth is superlinear.

Operators with inhomogeneous-type growth, which include some linear-growth terms, have only been studied in several specific cases. The first rigorous results in this setting are due to Marcellini and Miller \cite{MM}; since then, such problems were studied for instance for the anisotropic $p$-Laplacian \cite{Gor2023,MRST1}. However, the most relevant case for this paper are problems involving the $(1,p)$-Laplacian operator, i.e.,
\begin{equation}
\Delta_{(1,p)} u = \Delta_1 u + \Delta_p u
\end{equation}
with $p > 1$. Such problems appear in relation to fluid mechanics \cite{DL} (for $p=2$) and materials sciences \cite{S} (for $p=3$). Among them, the parabolic equation
\begin{equation}\label{1,pEq}
u_t - \Delta_1 u -  \Delta_p u = f
\end{equation}
for $p=2$ emerges when modeling the motion of a laminar Bingham fluid, the non-Newtonian fluid which has both plastic and viscous properties. In this model, the $1$-Laplacian $\Delta_1$ reflects the plasticity of a fluid and the Laplacian $\Delta$ its viscosity. Inspired by methods of machine learning, Nguyen \cite{VTN} has also studied the well-posedness of problem \eqref{1,pEq}, and the recent works of Tsubouchi \cite{T1,T2} address the regularity of solutions for the elliptic problem for the $(1,p)$-Laplacian.  Finally, a fourth-order problem of the type
\begin{equation}
u_t + \Delta(\Delta_1 u + \mu \Delta_p u) = 0
\end{equation}
has been discussed in \cite{GiGi}. This type of problems however requires a different approach, and the one used in that paper is to treat this equation as a gradient flow of a certain functional in the Hilbert space $H^{-1}$. However, in this paper we deal exclusively with second-order parabolic equations, and in our setting it is natural to study them in the Hilbert space $L^2$.

While evolution problems of $(q,p)$-Laplacian  type have been studied, as described above, in the classical setting, the literature on nonlocal problems is very thin. To the best of our knowledge, this type of problems for nonlocal operators in $\mathbb{R}^N$ defined via non-singular kernels has not been studied. Moreover, in the case of weighted graphs, we are only aware of a result given in \cite{HJ} concerning the elliptic problem for the $(q,p)$-Laplace operator (with $q,p > 1$) in the particular case of lattice graphs. Our aim is, in fact, to study these problems in the general framework of random walk spaces, which includes the two previous examples as particular cases.  In the present paper, we address at once the case of nonlocal problems involving a $(q,p)$-Laplacian type operator with $q > 1$ and the case of a $(1,p)$-Laplacian type operator.

The structure of this paper is as follows. In Section \ref{preliminaries}, we recall the basic information concerning nonlinear semigroup theory in Hilbert spaces, the definition of random walk spaces, and basic definitions and properties of nonlocal differential operators on random walk spaces. Section \ref{section3} is dedicated to the study of the evolution problem related to the functionals \eqref{0945} and \eqref{0945a},  and to the study of  the case of an evolution governed by two $1$-Laplacians. For these problems, we characterise the subdifferential of the respective functionals, which gives us a suitable notion of solution. We prove existence of solutions and study the asymptotic behaviour: in particular, we show that assuming a Poincar\'e-type inequality on the random walk space, a solution to the evolution problem for the functional \eqref{0945} has finite extinction time, and for the functional \eqref{0945a} it has finite extinction time if $q < 2$. Then, in Section \ref{section4}, we study the gradient flow of the functional \eqref{eldelaintro01}. We show how to adapt the technique from the previous section to this problem and also obtain a characterisation of the subdifferential, existence of solutions to the gradient flow, and we discuss the asymptotic behaviour. Finally, in Section \ref{sec:examples} we present several examples concerning the possible behaviour of solutions to gradient flows of the functionals \eqref{0945} and \eqref{eldelaintro01} in the case of finite weighted graphs.

\section{Preliminaries}\label{preliminaries}

\subsection{Convex analysis  and nonlinear semigroup theory}\label{ConAnal}

 Let us first recall some terminology and key results concerning the methods of convex analysis and nonlinear semigroup theory (see \cite{BCr2} and
\cite{Brezis}). If $H$ is a real Hilbert space with inner product $( \ , \ )$ and
$\mathcal{F}: H \rightarrow (- \infty, + \infty]$ is convex   and proper, i.e., such that its domain
$$\mathcal{D}(\mathcal{F}):= \{ u \in H \ : \ \mathcal{F}(u) < +\infty \}$$
is nonempty, then the subdifferential of  $\mathcal{F}$ is defined as the multivalued operator
$\partial \mathcal{F}$ given by
$$v \in \partial \mathcal{F}(u) \ \iff \ \mathcal{F}(w) - \mathcal{F}(u) \geq (v, w -u)
\ \ \ \forall \, w \in H,$$
  with domain $D(\partial \mathcal{F}) := \{ u \in H: \ \partial \mathcal{F}(u) \not= \emptyset \}$.
Consider the abstract Cauchy problem
\begin{equation}\label{ACP1}
\left\{ \begin{array}{ll} \frac{du}{dt} + \partial \mathcal{F} (u(t)) \ni f(t, \cdot), \, \quad &t \in (0, T),
  \\[10pt] u(0) = u_0, \quad &u_0 \in H. \end{array} \right.
\end{equation}

\begin{definition}\label{StronSol}{\rm We say that $u \in C([0,T]; H)$ is a {\it strong solution}\index{strong solution} of problem \eqref{ACP1}, if the following conditions hold: $u \in W_{\rm loc}^{1,2}(0,T;H)$; for almost all $t \in (0,T)$ we have $u(t) \in D(\partial \mathcal{F})$; and it satisfies \eqref{ACP1}.}
\end{definition}

We are now in position to state the celebrated {\it Brezis-K\={o}mura theorem} (see \cite{Brezis} and  \cite{ABS} for recent overviews).

\begin{theorem}\label{BKTheorem}\index{Brezis-K\={o}mura theorem} Let $\mathcal{F} : H \to (-\infty, \infty]$ be a proper, convex, and lower semicontinuous functional. Given $u_0 \in \overline{D(\partial \mathcal{F})}$ and $f \in L^2(0,T; H)$, there exists a unique strong solution $u(t)$ of the abstract Cauchy problem \eqref{ACP1}.
\end{theorem}

The following result given in \cite[Proposition 2.11]{Brezis} is useful to characterise the closure of the domain of the subdifferential.

\begin{proposition}\label{A.Domain}
Let $\mathcal{F} : H \to (-\infty, \infty]$ be a proper, convex, and lower semicontinuous functional. Then,
$$D(\partial \mathcal{F}) \subset D( \mathcal{F}) \subset \overline{ D( \mathcal{F})} = \overline{ D( \partial \mathcal{F})}.$$
\end{proposition}

 We will also need the following characterisation of the subdifferential of an operator which is positive homogeneous of degree $1$ (i.e., $\Psi(t u) = t \Psi(u)$ for all $t \geq 0$ and $u \in H$) given in \cite{ACMBook}. Given such functional $\Psi : H \rightarrow [0,\infty]$, we define $\tilde{\Psi} : H \rightarrow [0,\infty]$ by
\begin{equation}\label{phitilde}
\tilde{\Psi}(v):= \sup \left\{ \frac{(v,u)}{\Psi(u)}: \ u \in H \right\}
\end{equation}
with the convention that $\frac00 = 0$ and $\frac{0}{\infty} =0$. Obviously, if $\Psi_1 \leq \Psi_2$, then $\widetilde{\Psi_2} \leq \widetilde{\Psi_1}.$

\begin{proposition}\label{degree1}
If $\Psi$ is convex, lower semicontinuous and positive homogeneous of degree $1$, then $\tilde{\tilde{\Psi}} = \Psi$. Moreover,
$$v \in \partial \Psi (u) \iff \tilde{\Psi}(v) \leq 1 \ \hbox{and} \ (v, u) = \Psi(u).$$
\end{proposition}

Let us also collect some preliminaries and notations concerning
completely accretive operators that will be used afterwards (see
\cite{BCr2}). Let $(\Sigma, \mathcal{B}, \mu)$ be a $\sigma$-finite
measure space, and $M(\Sigma,\mu)$ the space of $\mu$-a.e. equivalent
classes of measurable functions $u : \Sigma\to \R$.
Let
  \begin{displaymath}
    J_0:= \Big\{ j : \R \rightarrow
    [0,+\infty] : \text{$j$ is convex, lower
      semicontinuous, }j(0) = 0 \Big\}.
  \end{displaymath}
For every $u$, $v\in M(\Sigma,\mu)$, we write
  \begin{displaymath}
  u\ll v \quad \text{if and only if}\quad \int_{\Sigma} j(u)
  \,d\mu \le \int_{\Sigma} j(v) \, d\mu\quad\text{for all $j\in J_{0}$.}
\end{displaymath}

\begin{definition}{\rm
  An operator $A$ on $M(\Sigma,\mu)$ is called {\it completely
    accretive} if for every $\lambda>0$ and
  for every $(u_1,v_1)$, $(u_{2},v_{2}) \in A$ and $\lambda >0$, one
  has that
  \begin{displaymath}
    u_1 - u_2 \ll u_1 - u_2 + \lambda (v_1 - v_2).
  \end{displaymath}
  If $X$ is a linear subspace of $M(\Sigma,\mu)$ and $A$ an operator
  on $X$, then $A$ is {\it $m$-completely accretive on $X$} if $A$ is
  completely accretive and satisfies the {\it range
  condition}
\begin{displaymath}
  \textrm{Ran}(I+\lambda A)=X\qquad\text{for some (or equivalently, for
    all) $\lambda>0$.}
\end{displaymath}
}
\end{definition}

We denote
\begin{displaymath}
L_0(\Sigma,\mu):= \left\{ u \in M(\Sigma,\mu): \ \int_{\Sigma} \big[\vert u \vert - k\big]^+\, d\mu < \infty \text{ for all $k > 0$} \right\}.
\end{displaymath}
The following result  was proved in \cite{BCr2}.

\begin{proposition}\label{prop:completely-accretive}
Let $P_{0}$ denote the set of all functions $q\in C^{\infty}(\R)$ satisfying $0\le q'\le 1$, $q'$  is compactly supported, and $0$ is not contained in the support ${\rm supp}(q)$ of $q$. Then, an operator $A \subseteq L_{0}(\Sigma,\mu)\times L_{0}(\Sigma,\mu)$ is  completely accretive if and only if
\begin{displaymath}
\int_{\Sigma}q(u-\hat{u})(v-\hat{v})\, d\mu \ge 0
\end{displaymath}
for every $q\in P_{0}$ and every $(u,v)$, $(\hat{u},\hat{v})\in A$.
\end{proposition}

\subsection{Random walk spaces}\label{RWS1}

We recall some concepts and results about random walk spaces  given in \cite{MST0,MST1} and the monograph \cite{MSTBook}. Let $(X,\mathcal{B})$ be a measurable space such that the $\sigma$-field $\mathcal{B}$ is countably generated. A random walk $m$
on $(X,\mathcal{B})$ is a family of probability measures $(m_x)_{x\in X}$
on $\mathcal{B}$ such that $x\mapsto m_x(B)$ is a measurable function on $X$ for each fixed $B\in\mathcal{B}$.

The notation and terminology chosen in this definition comes from Ollivier's paper \cite{O}. As noted in that paper, geometers may think of $m_x$ as a replacement for the notion of balls around $x$, while in probabilistic terms we can rather think of these probability measures as defining a Markov chain whose transition probability from $x$ to $y$ in $n$ steps is
\begin{equation}
\displaystyle
dm_x^{*n}(y):= \int_{z \in X}  dm_z(y) \, dm_x^{*(n-1)}(z), \ \ n\ge 1
\end{equation}
and $m_x^{*0} = \delta_x$, the Dirac measure at $x$.

\begin{definition}\label{convolutionofameasure}{\rm
Let $m$ be a random walk on $(X,\mathcal{B})$ and $\mu$ a $\sigma$-finite measure on $X$. The convolution of $\mu$ with $m$ on $X$ is the measure defined as follows:
$$\mu \ast m (A) := \int_X m_x(A) \, d\mu(x)\ \ \forall A\in\mathcal{B},$$
which is the image of $\mu$ by the random walk $m$.}
\end{definition}

\begin{definition}\label{def.invariant.measure} {\rm If $m$ is a random walk on $(X,\mathcal{B})$,
a $\sigma$-finite measure $\nu$ on $X$ is {\it invariant}
with respect to the random walk $m$ if
$$\nu\ast m = \nu.$$
The measure $\nu$ is said to be {\it reversible} if moreover the detailed balance condition $$dm_x(y) \, d\nu(x)  = dm_y(x) \, d\nu(y) $$ holds true.}
\end{definition}

\begin{definition}\label{DefMRWSf}{\rm
Let $(X,\mathcal{B})$ be a measurable space where the $\sigma$-field $\mathcal{B}$ is countably generated. Let $m$ be a random walk on $(X,\mathcal{B})$ and $\nu$ a $\sigma$-finite measure which is invariant  with respect to $m$. The measurable space together with $m$ and $\nu$ is then called a random walk space
and is denoted by $[X,\mathcal{B},m,\nu]$.}
\end{definition}

\begin{definition}\label{def.m.interaction}{\rm
Let  $[X,\mathcal{B},m,\nu]$ be a random walk space and let $A$, $B\in\mathcal{B}$. We define the {\it $m$-interaction} between $A$ and $B$ as
\begin{equation}\label{m.interaction} L_m(A,B):= \int_A \int_B dm_x(y) \, d\nu(x) = \int_A m_x(B) \, d\nu(x).
 \end{equation}
 }
\end{definition}

\begin{definition}\label{def.m.connected.random.walk.space}{\rm
Let $[X,\mathcal{B},m,\nu]$ be a random walk space. We say that $[X,\mathcal{B},m,\nu]$ is $m$-connected, if for $ A,B\in\mathcal{B}$ which satisfy $A\cup B=X$ and $L_m(A,B)= 0$, we have either $\nu(A)=0$ or $\nu(B)=0$.
}
\end{definition}

Some equivalent characterisations of $m$-connectedness can be found in \cite{MST0}. Now, let us present several examples of random walk spaces.

\begin{example}\label{example.nonlocalJ} \rm
Consider the metric measure space $(\R^N, d, \mathcal{L}^N)$, where $d$ is the Euclidean distance and $\mathcal{L}^N$ the Lebesgue measure on $\R^N$, and $\mathcal{B}$ be the Borel   $\sigma$-algebra. For simplicity, we will write $dx$ instead of $d\mathcal{L}^N(x)$.  Let  $J:\R^N\to[0,+\infty)$ be a measurable, nonnegative and radially symmetric
function  verifying  $\int_{\R^N}J(x) \, dx=1$. Let $m^J$ be the following random walk on $(\R^N, \mathcal{B})$: 
\index{m-J@$m^J$} 
$$m^J_x(A) :=  \int_A J(x - y) \, dy \quad \hbox{ for every $x\in \R^N$ and every Borel set } A \subset  \R^N  .$$
Then, applying Fubini's Theorem it is easy to see that the Lebesgue measure $\mathcal{L}^N$ is reversible with respect to $m^J$. Therefore,  $[\R^N,  \mathcal{B}, m^J, \mathcal{L}^N]$  is a  reversible random walk space.
\end{example}

\begin{example}\label{example.graphs}[Weighted discrete graphs] \rm Consider a locally finite  weighted discrete graph $$G = (V(G), E(G)),$$ where $V(G)$ is the vertex set, $E(G)$ is the edge set and each edge $(x,y) \in E(G)$ has an assigned positive weight $w_{xy} = w_{yx}$ (we will write $x\sim y$ if $(x,y) \in E(G)$). Suppose further that $w_{xy} = 0$ if $(x,y) \not\in E(G)$.  Note that there may be loops in the graph, that is, we may have $(x,x)\in E(G)$ for some $x\in V(G)$ and, therefore, $w_{xx}>0$. Recall that a graph is locally finite if every vertex is only contained in a finite number of edges.

 A finite sequence $\{ x_k \}_{k=0}^n$  of vertices of the graph is called a {\it  path} if $x_k \sim x_{k+1}$ for all $k = 0, 1, ..., n-1$. The {\it length} of a path $\{ x_k \}_{k=0}^n$ is defined as the number $n$ of edges in the path. With this terminology, $G = (V(G), E(G))$ is said to be {\it connected} if, for any two vertices $x, y \in V$, there is a path connecting $x$ and $y$, that is, a path $\{ x_k \}_{k=0}^n$ such that $x_0 = x$ and $x_n = y$.  Finally, if $G = (V(G), E(G))$ is connected, the {\it graph distance} $d_G(x,y)$ between any two distinct vertices $x, y$ is defined as the minimum of the lengths of the paths connecting $x$ and $y$. Note that this metric is independent of the weights.

For $x \in V(G)$ we define the weight at $x$ as
$$d_x:= \sum_{y\sim x} w_{xy} = \sum_{y\in V(G)} w_{xy},$$
and the neighbourhood of $x$ as $N_G(x) := \{ y \in V(G) \, : \, x\sim y\}$. Note that, by definition of locally finite graph, the sets $N_G(x)$ are finite. When all the non-null weights are $1$, $d_x$ coincides with the degree of the vertex $x$ in a graph, that is,  the number of edges containing $x$.

For each $x \in V(G)$  we define the following probability measure 
  \index{m-G@$m$} 
\begin{equation}\label{discRW}m_x:=  \frac{1}{d_x}\sum_{y \sim x} w_{xy}\,\delta_y.\\ \\
\end{equation}
It is not difficult to see that the measure $\nu$ defined as
 $$\nu(A):= \sum_{x \in A} d_x,  \quad A \subset V(G),$$
is a reversible measure with respect to this random walk. Therefore, $[V(G),\mathcal{B},m,\nu]$ is a reversible random walk space,  where $\mathcal{B}$ is  the $\sigma$-algebra of all subsets of $V(G)$.

\end{example}

\begin{example}\label{example.restriction.to.Omega} \rm { Given a random walk  space $[X,\mathcal{B},m,\nu]$ and $\Omega \in \mathcal{B}$ with $\nu(\Omega) > 0$, let
  $\mathcal{B}_\Omega$ be the   $\sigma$-algebra
 $$\mathcal{B}_\Omega:=\{B\in\mathcal{B} \, : \, B\subset \Omega\}.$$
 Let us now define}
$$m^{\Omega}_x(A):=\int_A d m_x(y)+\left(\int_{X\setminus \Omega}d m_x(y)\right)\delta_x(A) \quad \hbox{ for every } A\in\mathcal{B}_\Omega  \hbox{ and } x\in\Omega.
$$
Then, $m^{\Omega}$ is a random walk on $(\Omega,\mathcal{B}_\Omega)$ and it easy to see that $\nu \res \Omega$ is invariant with respect to $m^{\Omega}$. Therefore,  $[\Omega,\mathcal{B}_\Omega,m^{\Omega},\nu \res \Omega]$ is a random walk space. Moreover, if $\nu$ is reversible with respect to $m$ then $\nu \res \Omega$ is  reversible with respect to $m^{\Omega}$. Of course, if $\nu$ is a probability measure we may normalise $\nu \res \Omega$ to obtain the random walk space
$$\left[\Omega,\mathcal{B}_\Omega,m^{\Omega}, \frac{1}{\nu(\Omega)}\nu \res \Omega \right].$$
In particular, in the context of Example \ref{example.nonlocalJ}, if $\Omega$ is a closed and bounded subset of $\R^N$, we obtain the  random walk space $[\Omega,\mathcal{B}_\Omega, m^{J,\Omega},\mathcal{L}^N\res \Omega]$ where 
\index{m-JOmega@$m^{J,\Omega}$} 
$m^{J,\Omega} := (m^J)^{\Omega}$; that is,
$$m^{J,\Omega}_x(A):=\int_A J(x-y) \, dy + \left(\int_{\R^n\setminus \Omega} J(x-z) \, dz \right) \delta_x(A)$$ for every Borel set   $A \subset  \Omega$  and $x\in\Omega$.

\end{example}

Another important example of random walk spaces are weighted hypergraphs;  for their presentation in this framework, we refer to \cite[Example 1.44]{MSTBook}.

\subsection{Nonlocal differential operators}\label{nonlocal.notions.1.section}

Let us introduce the nonlocal counterparts of some classical concepts.

\begin{definition}\label{nonlocalgraddiv}{\rm
Let $[X,\mathcal{B},m,\nu]$ be a random walk space. Given a function $u : X \rightarrow \R$ we define its {\it nonlocal gradient}
$\nabla u: X \times X \rightarrow \R$ as
$$\nabla u (x,y):= u(y) - u(x) \quad \forall \, x,y \in X.$$
Moreover, given $\z : X \times X \rightarrow \R$, its {\it $m$-divergence} 
${\rm div}_m \z : X \rightarrow \R$ is defined as
 $$({\rm div}_m \z)(x):= \frac12 \int_{X} (\z(x,y) - \z(y,x)) \, dm_x(y).$$
 }
\end{definition}

Let us consider the {\it tensor product} of $\nu$ and $m$, denoted by $\nu\otimes m_x$, which is the measure in $X \times X$ defined by
$$(\nu\otimes m_x)(A \times B):= \int_A \left(\int_B dm_x(y) \right) d\nu(x).$$
We have the following {\it integration by parts formula}.

\begin{proposition}\label{FIPartes1} If $v \in L^p(X, \nu)$ and $\z \in   L^{p'}(X \times X, \nu\otimes m_x)$, then
\begin{equation}\label{IntbP0}
\int_X v(x) \, {\rm div}_m (\z) (x) \, d \nu(x) = - \frac12 \int_{X \times X} \z(x,y) \, \nabla v(x,y) \, d(\nu\otimes m_x)(x,y).
\end{equation}
\end{proposition}

For $p>1$, we consider the functional $\mathcal{F}_{p,m} :L^2(X, \nu) \rightarrow (-\infty, +\infty]$ defined by
$$ \mathcal{F}_{p,m}(u):= \left\{\begin{array}{ll}\displaystyle\frac{1}{2p}\int_{X \times X}  |u(y)-u(x)|^p \, d(\nu\otimes m_x)(x,y) \quad &\hbox{if $\nabla u \in L^p(X \times X,\nu\otimes m_x)$;} \\[10pt] +\infty \quad &\hbox{otherwise}. \end{array} \right.$$
Observe that
$$L^p(X, \nu) \cap L^2(X, \nu)\subset D(\mathcal{F}_{p,m}),$$
and, since $\mathcal{F}_{p,m}$ is convex and lower semicontinuous,  the subdifferential $\partial_{L^2(X,\nu)} \mathcal{F}_{p,m}$ is~a~maximal monotone operator with a dense domain.

\begin{definition}\label{dfn:plaplacian}{\rm
We define the {\it $m$-$p$-Laplacian operator} $\Delta_p^m$  in $[X,\mathcal{B},m,\nu]$ as $$(u,v) \in \Delta_p^m \iff u,v\in L^2(X,\nu),\     \nabla u \in L^{p}(X \times X, \nu\otimes m_x) \ \hbox{and}$$
$$v(x)= {\rm div}_m (\vert \nabla u \vert^{p-2} \nabla u)(x) = \int_X \vert \nabla u(x,y) \vert^{p-2} \, \nabla u(x,y) \, dm_x(y).$$
}
\end{definition}

Note that as a consequence of Proposition \ref{FIPartes1}, we have the following integration by~parts formula: for all $u, v \in L^p(X, \nu)$, it holds that
\begin{equation}\label{IntbP1}
\int_X v(x) \Delta_p^m u(x) \, d \nu(x) = - \frac12 \int_{X \times X} \vert \nabla u(x,y) \vert^{p-2} \, \nabla u(x,y) \, \nabla v(x,y) \, d(\nu\otimes m_x)(x,y).
\end{equation}
The following result was given in~\cite{MT1} for weighted graphs;  here, we prove it for general random walk spaces.

\begin{theorem}\label{plaplac01} The operator $-\Delta_p^m$ is completely accretive and
\begin{equation}\label{pSubd}
(u,v) \in \Delta_p^m \iff (u,-v) \in \partial_{L^2(X,\nu)} \mathcal{F}_{p,m}.
\end{equation}
\end{theorem}

\begin{proof}
For every $q\in P_{0}$ and every $(u,-v)$, $(\hat{u},-\hat{v})\in \Delta_p^m$, by the integration by parts formula, we have
\begin{align}
\int_{X}q(u-\hat{u})(v-\hat{v})\, d\nu &= -\int_{X} q(u-\hat{u}) (\Delta_p^m u  -\Delta_p^m \hat{u}) \, d \nu \\
& = - \int_{X} q(u-\hat{u}) \cdot \Delta_p^m u \, d\nu + \int_{X} q(u-\hat{u}) \cdot \Delta_p^m \hat{u} \, d \nu \\
&=  \frac12 \int_{X \times X} \vert \nabla u \vert^{p-2} \, \nabla u \, \nabla q(u-\hat{u}) \, d(\nu \otimes m_x) \\
&\qquad\qquad\qquad -  \frac12 \int_{X \times X} \vert \nabla \hat{u} \vert^{p-2} \, \nabla \hat{u} \, \nabla q(u-\hat{u}) \, d(\nu \otimes m_x) \\
&= \frac12 \int_{X \times X} \left(\vert \nabla u \vert^{p-2} \, \nabla u - \vert \nabla  \hat{u} \vert^{p-2} \, \nabla \hat{u} \right) \nabla q(u-\hat{u}) \, d(\nu \otimes m_x) \geq 0.
\end{align}
Then, by Proposition \ref{prop:completely-accretive}, the operator $-\Delta_p^m$ is completely accretive.

In order to   see   that $-\Delta_p^m = \partial\mathcal{F}_{p,m}$, since $-\Delta_p^m$ is completely accretive, it is enough to  show that $\partial\mathcal{F}_{p,m}\subset-\Delta_p^m$.
To this end, let $(u, v) \in \partial \mathcal{F}_{p,m}$;  hence, it holds that $u\in L^2(X,\nu)$ and $\nabla u \in L^p(X \times X,\nu\otimes m_x)$. Then, for every $w \in  L^1(X, \nu) \cap L^\infty(X, \nu) $ and $t>0$, we have
 $$\frac{\mathcal{F}_{p,m}(u + tw) - \mathcal{F}_{p,m}(u)}{t} \geq \int_X vw \, d\nu.$$
Then, taking limit as $t \to 0^+$, we obtain that
$$\frac12 \int_{X \times X} \vert \nabla u(x,y) \vert^{p-2} \nabla u(x,y) \nabla w(x,y) \, dm_x(y) \, d\nu(x) \geq \int_X vw \, d\nu.$$
Now, since this inequality is also true for $-w$, we have
$$\frac12 \int_{X \times X} \vert \nabla u(x,y) \vert^{p-2} \nabla u(x,y) \nabla w(x,y) \, dm_x(y) \, d\nu(x)= \int_X v w \, d\nu.$$
Then, applying again the integration by parts formula, we get
$$- \int_X \Delta_p^m u (x) \, w(x) \, d\nu(x) = \int_X v w \, d\nu \quad \forall \, w \in    L^1(X, \nu) \cap L^\infty(X, \nu).$$
From here, we deduce that $v = - \Delta_p^m u$, and consequently $(u,-v) \in \Delta_p^m$.
\end{proof}

We define the space of {\it functions of bounded variation} in $[X,\mathcal{B},m,\nu]$ as
$$BV_m(X, \nu):= \left\{ u: X \to \R \ \hbox{measurable}:  \int_{X \times X} \vert \nabla u(x,y) \vert \, dm_x(y) \, d\nu(x) < \infty \right\}.$$
In \cite{MST1} (see also~\cite{MSTBook}) the total variation functional  $\mathcal{F}_{1,m} :L^2(X, \nu) \rightarrow (-\infty, +\infty]$ is  defined by
$$ \mathcal{F}_{1,m}(u):= \left\{\begin{array}{ll}\displaystyle\frac12\int_{X\times X}|u(y)-u(x)| \, d(\nu\otimes m_x)(x,y)\quad &\hbox{if} \ u \in BV_{m}(X, \nu); \\[10pt] +\infty \quad &\hbox{if} \ u \in L^2(X, \nu) \setminus BV_{m}(X, \nu). \end{array} \right.$$
The functional $\mathcal{F}_{1,m}$ can be characterised in a similar way as the total variation in the Euclidean setting. Denote
$$X_m^2(X, \nu):= \bigg\{ \z \in L^\infty(X \times X, \nu \otimes m_x): \ {\rm div}_m \z \in L^2(X, \nu) \bigg\}.$$
Then, we have the following {\it Green's formula} (see \cite[Proposition 3.8]{MSTBook}).

\begin{proposition}\label{Green} Given $v \in  BV_{m}(X, \nu) \cap L^2(X, \nu)$ and $\z\in X_m^2(X, \nu)$, we have
\begin{equation}\label{GreenF}
\int_X v(x) \, {\rm div}_m (\z) (x) \, d \nu(x) = - \frac12 \int_{X \times X} \z(x,y) \, \nabla v(x,y) \, d(\nu\otimes m_x)(x,y).
\end{equation}
\end{proposition}

 As a consequence of Green's formula, we have the following characterisation of the total variation shown in \cite[Proposition 3.9]{MSTBook}.

\begin{proposition}\label{LSC}
Given $u \in  BV_{m}(X, \nu) \cap L^2(X, \nu)$, we have
$$ \mathcal{F}_{1,m}(u) = \sup \left\{ \int_X u(x) ( {\rm div}_m \z)(x) \, d\nu(x): \ \z \in X_m^2(X, \nu), \ \Vert \z \Vert_{L^\infty(X \times X, \nu \otimes m_x)} \leq 1 \right\}.$$
\end{proposition}
By the above result, the functional $\mathcal{F}_{1,m}$ is lower semicontinuous with respect to convergence in $L^2(X,\nu)$. Since it is also convex, the subdifferential $\partial_{L^2(X,\nu)} \mathcal{F}_{1,m}$ is a maximal monotone operator with a dense domain, for which the following characterisation holds (see \cite{MST1,MSTBook}).

\begin{proposition}\label{CHarat1} Let $u, v \in L^2(X, \nu)$. The following assertions are equivalent:
\begin{itemize}

\item[(1)] $(u,v) \in \partial_{L^2(X,\nu)} \mathcal{F}_{1,m}$;

\item[(2)] There exists $\z \in X_m^2(X, \nu)$ with $\Vert \z \Vert_{L^\infty(X \times X, \nu \otimes m_x)} \leq 1$ such that
$$v = -  {\rm div}_m \z$$ and
$$\mathcal{F}_{1,m}(u) = \int_X u(x) \, v(x) \, d\nu(x);$$
\item[(3)] There exists an antisymmetric function $\g \in L^\infty(X \times X, \nu \otimes m_x)$ such that
$$\Vert \g \Vert_{L^\infty(X \times X, \nu \otimes m_x)} \leq 1;$$
$$v(x) = - \int_X \g(x,y) \, dm^1_x(y) \quad \hbox{for} \ \ \nu\hbox{-a.e.} \ x \in X;$$
$$\g(x,y) \in {\rm sign}(u(y) - u(x)) \quad \hbox{for} \ \ (\nu \otimes m_x) \hbox{-a.e.} \ (x,y) \in X \times X.$$
\end{itemize}

\end{proposition}

\begin{definition}{\rm We define the {\it  $m$-$1$-Laplacian operator} $\Delta_1^m$  in $[X,\mathcal{B},m,\nu]$ as
\begin{equation}\label{1Subd}
 (u,v) \in \Delta_1^m \iff (u,-v) \in \partial_{L^2(X,\nu)} \mathcal{F}_{1,m}.
 \end{equation}
}
\end{definition}

\begin{remark}\label{dominio01}\rm
In the case that $\nu(X)<\infty$ and $1\le p\le 2$, we have that
$$D(\Delta_p^m)=L^2(X,\nu).$$
For $1<p\le 2$, notice that if $u\in L^2(X,\nu)$, then $|u|^{p-1}\in L^2(X,\nu)$; indeed,  $$\int_X |u|^{2(p-1)}d\nu\le \int_{\{|u|\le 1\}}|u|^{2(p-1)}d\nu+\int_{\{|u|> 1\}}|u|^{2(p-1)}d\nu\le \nu(X)+\int_{\{|u|> 1\}}|u|^2d\nu,$$ where the last inequality  holds since $2(p-1)\le 2$. Therefore,
$$|\nabla u|^{p-1}\in L^2( X \times X, \nu \otimes m_x),$$
and consequently $\Delta_p^m u\in L^2(X,\nu)$. For $p=1$, if $u\in L^2(X,\nu)$, we have that
$$\displaystyle x\mapsto \int_X \hbox{sign}^0(u(y)-u(x)) \, dm_x(y)\in \Delta_1^mu,$$
hence the claim holds for all $1\le p\le 2$. \hfill $\blacksquare$
\end{remark}

\begin{definition}{\rm
For  $\nu(X) < \infty$   and $1\leq r <\infty$,  we say that $\mathcal{F}_{1,m}$ satisfies a  {\it $r$-Poincar\'{e} inequality}  if  there exists a constant $\lambda_r(\mathcal{F}_{1,m}) >0$ such that
\begin{equation}\label{coercDfini1}
\lambda_r(\mathcal{F}_{1,m}) \, \Vert u \Vert_{L^r(X, \nu)} \leq  \mathcal{F}_{1,m}(u) \quad \forall \, u\in \left\{ u \in L^r(X, \nu) \setminus \{ 0 \} : \int_X u \, d\nu =0\ \right\},
\end{equation}
which is equivalent to
\begin{equation}\label{Poincare1p=1}
\lambda_r(\mathcal{F}_{1,m}) \, \Vert u - \overline{u} \Vert_{L^r(X, \nu)} \leq \mathcal{F}_{1,m}(u) \quad \forall \, u \in L^r(X, \nu),
\end{equation}
where
$$\overline{u}:= \frac{1}{\nu(X)} \int_X u \, d\nu.$$
}
\end{definition}

\section{Mixing random walks with different growth functionals}\label{section3}

Let $[X,\mathcal{B},m^i,\nu_i]$, where $i=1,2$, be two random walk spaces. Note that the two random walk structures are defined on the same measurable  space. In this Section, we consider the problem of an evolution which is governed by the sum of two nonlocal $p$-Laplacian type operators which are defined using two different random walks. Throughout this section, we assume that
$$\nu_2 \ll \nu_1$$
and
$$\mu:=\frac{d\nu_2}{d\nu_1}\in L^\infty(X,\nu_1),$$
where $\mu>0$ $\nu_1$-a.e.  Due to these assumptions, we may consider the evolution in a joint Hilbert space, denoted by
$$H:=L^2(X,\nu_1).$$
Note that the aforementioned assumptions are satisfied by finite graphs, as well as some infinite graphs, e.g. $\mathbb{Z}^2$ in the first examples and the random walks of the Examples \ref{example.nonlocalJ} and \ref{example.restriction.to.Omega}.

This Section is divided into three parts. In the first subsection, we prove one of the main results of this Section, which concerns the analysis of the case of the sum of the (nonlocal) $1$-Laplacian operator and a (nonlocal) $p$-Laplacian operator  with $p > 1$. In the second subsection,  we prove another main result concerning the case of two $1$-Laplacian type operators. Finally, in the third subsection we consider the case of a $p$-Laplacian and $q$-Laplacian type operators with $p,q > 1$.

\subsection{ The $(1,p)$-Laplace operator}\label{subsec:1plaplace}

\subsubsection{The gradient flow}

In this subsection, we consider the evolution governed by the sum of the $1$-Laplacian type operator defined on one random walk structure and a $p$-Laplacian type operator (with $p > 1$) defined on a second random walk structure. We consider the functional given in~\eqref{0945} that, with the same previous notation, can be written as
$$\mathcal{F}_{1,m^1}+\mathcal{F}_{p,m^2}$$
where $\mathcal{F}_{1,m^1}: L^2(X,\nu_1)\to (-\infty,+\infty]$ is given by
$$ \mathcal{F}_{1,m^1}(u):=  \left\{\begin{array}{ll}\displaystyle\frac12\int_{X \times X}|u(y)-u(x)| \,  \, d(\nu_1\otimes m^1_x)(x,y)   &\hbox{ if} \ u \in BV_{m^1}(X, \nu_1); \\[10pt] +\infty   &\hbox{ if} \ u \in L^2(X, \nu_1) \setminus BV_{m^1}(X, \nu_1), \end{array} \right.$$
and  $ \mathcal{F}_{p,m^2}:L^2(X,\nu_1)\to (-\infty,+\infty]$ is given by
$$ \mathcal{F}_{p,m^2}(u):= \left\{\begin{array}{ll}\displaystyle\frac{1}{2p}\int_{X\times X}|u(y)-u(x)|^p \,  \, d(\nu_2\otimes m^2_x)(x,y)  &\hbox{ if $|\nabla u|^p \in L^1(X \times X,\nu_2\otimes m_x^2)$;} \\[10pt] +\infty  &\hbox{ otherwise}. \end{array} \right.$$
Observe that
$$L^1(X,\nu_1)\cap L^2(X,\nu_1)\subset \hbox{Dom}(\mathcal{F}_{1,m^1})$$
and
$$L^p(X,\nu_2)\cap L^2(X,\nu_1)\subset \hbox{Dom}(\mathcal{F}_{p,m^2}).$$
  Moreover, both functionals are convex and lower semicontinuous with respect to $H$-convergence. Indeed, the functional $\mathcal{F}_{1,m^1}$ is the $m^1$-total variation; for the functional $\mathcal{F}_{p,m^2}$, notice that if the sequence $u_n$ converges to $u$ in $H$, then also
$$\lim_{n \rightarrow \infty} \int_X |u_n-u|^2 \, d\nu_2 = \lim_{n \rightarrow \infty} \int_X \mu \, |u_n-u|^2 \, d\nu_1 \to 0,$$
and the claim follows from the fact that $\mathcal{F}_{p,m^2}$ is $L^2(X,\nu_2)$-lower semicontinuous. Therefore, the subdifferentials $\partial_H \mathcal{F}_{1,m^1}$, $\partial_H \mathcal{F}_{p,m^2}$, and $\partial_H \left(\mathcal{F}_{1,m^1}+\mathcal{F}_{p,m^2}\right)$ are maximal monotone operators on~$H$.

\begin{theorem}\label{0942}
Assume that
$$\mu:=\frac{d\nu_2}{d\nu_1}\in L^\infty(X,\nu_1),$$
\begin{equation}\label{ASSUMP1}
\hbox{and there exists $c >0$ such that} \ \mu\ge c\quad \nu_1\hbox{-a.e.}
\end{equation}
Suppose that  one of the following conditions holds:
\begin{itemize}
\item[(a)] $\nu_1(X) < \infty$;

\item[(b)] $\nu_1(X) = +\infty$ and $p \geq 2$.
\end{itemize}
Then, we have
\begin{equation}\label{eq:subdifferentialofsum}
\partial_H \left(\mathcal{F}_{1,m^1}+\mathcal{F}_{p,m^2}\right) = -\Delta^{m^1}_1- \mu\Delta^{m^2}_p.
\end{equation}
Moreover, this operator is completely accretive and has a dense domain in $H$.
\end{theorem}

\begin{proof}
 Define the operator $A\subset L^2(X,\nu_1)\times L^2(X,\nu_1)$ as follows:
$$ (u,v)\in A \quad \Leftrightarrow \quad u,v\in L^2(X,\nu_1) \, \mbox{ and } \, v \in -\Delta^{m^1}_1u- \mu\Delta^{m^2}_pu.$$

{\bf \flushleft Step 1.} We first show that this operator is completely accretive in $L^2(X,\nu_1)$. For every $q\in P_{0}$ and every $(u,v)$, $(\hat{u},\hat{v})\in A$, we have that $v = v_1 + v_2$ and $\hat{v} = \hat{v}_1 + \hat{v}_2$, where $v_1 \in -\Delta_1^{m^1} u$, $\hat{v}_1 \in -\Delta_1^{m^1} \hat{u}$, $v_2 \in - \mu \Delta_p^{m^2} u$ and $\hat{v}_2 \in - \mu \Delta_p^{m^2} \hat{u}$. Then,
\begin{align}
\int_{X} q(u-\hat{u})&(v-\hat{v})\, d\nu_1 \\
&=\int_{X} q(u-\hat{u})(v_1 + v_2 -\hat{v}_1 - \hat{v}_2)\, d\nu_1 \\
&= -\int_{ X} q(u-\hat{u}) \bigg( \int_X \g \, dm_x^1 - \int_X \hat{\g} \, dm_x^1 + \mu \Delta_p^{m^2} u  - \mu \Delta_p^{m^2} \hat{u} \bigg) \, d \nu_1 \\
&= -\int_{ X} q(u-\hat{u}) \bigg( \int_X \g \, dm_x^1 - \int_X \hat{\g} \, dm_x^1 \bigg) \, d\nu_1 \\
&\qquad\qquad\qquad\qquad\qquad\qquad\qquad - \int_X q(u - \hat{u}) (\Delta_p^{m^2} u - \Delta_p^{m^2} \hat{u}) \, d \nu_2 \geq 0.
\end{align}
Here, $\g$ and $\hat{\g}$ are the antisymmetric functions associated to $v_1$ and $\hat{v}_1$ respectively through Proposition \ref{CHarat1}. The final inequality is a consequence of complete accretivity of the operators $-\Delta_1^{m^1}$ and $-\Delta_p^{m^2}$, or in other words, the first term can be estimated as in \cite[Proposition 3.13]{MSTBook} and the second term can be estimated as in Theorem \ref{plaplac01}. Thus, by Proposition \ref{prop:completely-accretive}, the operator  $-\Delta_1^{m^1} -\Delta_p^{m^2}$ is completely accretive.

Then, to get~\eqref{eq:subdifferentialofsum} we only need to prove
$$  \partial_H  (\mathcal{F}_{1,m^1} +  \mathcal{F}_{p,m^2})\subset A.$$
This is done in   {\bf Steps 2, 3} and {\bf 4}.

{\flushleft  \bf Step 2.} Let us see that if $v\in \partial_H \left( \mathcal{F}_{1,m^1}+\mathcal{F}_{p,m^2}\right)(u)$, then
$$   \eta := \,  \hbox{$v+\mu{\rm div}_{m^2} (\vert \nabla u \vert^{p-2} \nabla u) \in L^\infty(X,\nu_1)$  with $\| \eta \|_\infty \leq 1$.}$$
Fix $t \in \mathbb{R}$. Then, for any
$$w\in  \hbox{Dom}(\mathcal{F}_{1,m^1})\cap\hbox{Dom}(\mathcal{F}_{p,m^2})$$
by definition of the subdifferential it holds that
\begin{equation}\label{1221}
\mathcal{F}_{1,m^1}(u+tw)-\mathcal{F}_{1,m^1}(u) +    \mathcal{F}_{p,m^2}(u+tw)-\mathcal{F}_{p,m^2}(u)\ge t\int_X vw \, d\nu_1.
\end{equation}
Therefore, for $t>0$ we have
\begin{align}
\frac12\int_X\int_X &\frac{|\nabla (u+tw)(x,y)|-|\nabla u(x,y)|}{t} \, dm^1_x(y) \, d\nu_1(x) \\
&+ \frac{1}{2p}\int_X\int_X\frac{|\nabla (u+tw)(x,y)|^{p}-|\nabla u(x,y)|^p}{t} \, dm^2_x(y) \, d\nu_2(x) \geq \int_X v w \, d\nu_1.
\end{align}
Taking the limit in the above expression as $t\to 0^+$, using the dominated convergence theorem (with $\left|\frac{|\nabla (u+tw)|-|\nabla u|}{t}\right|\le |\nabla w|$) we obtain
\begin{align}
\frac12 &\iint_{\{(x,y): u(y)-u(x)\neq0\}} \hbox{sign}^0(u(y)-u(x)) \, \nabla w(x,y) \, dm^1_x(y) \, d\nu_1(x) \\
& \qquad\qquad +\frac12\iint_{\{(x,y): u(y)-u(x)=0\}}|\nabla w(x,y)| \, dm^1_x(y) \, d\nu_1(x)\label{1941} \\
& \qquad\qquad + \frac{1}{2}\int_X\int_X \vert \nabla u(x,y)\vert^{p-2} \, \nabla u(x,y) \, \nabla w(x,y) \, dm^2_x(y) \, d\nu_2(x) \geq \int_X vw \, d\nu_1.
\end{align}
In the same way, with $t<0$ in~\eqref{1221} (or taking $-w$ in the above inequality), we get
\begin{align}
\frac12 &\iint_{\{(x,y): u(y)-u(x)\neq0\}} \hbox{sign}^0(u(y)-u(x)) \, \nabla w(x,y) \, dm^1_x(y) \, d\nu_1(x) \\
& \qquad\qquad -\frac12\iint_{\{(x,y): u(y)-u(x)=0\}}|\nabla w(x,y)| \, dm^1_x(y) \, d\nu_1(x) \label{19412}\\
& \qquad\qquad + \frac{1}{2}\int_X\int_X \vert \nabla u(x,y)\vert^{p-2} \, \nabla u(x,y) \, \nabla w(x,y) \, dm^2_x(y) \, d\nu_2(x) \le \int_X vw \, d\nu_1.
\end{align}
After integration by parts, \eqref{1941} may be written as
\begin{align}
\frac12 &\iint_{\{(x,y): u(y)-u(x)\neq0\}} \hbox{sign}^0(u(y)-u(x)) \, \nabla w(x,y) \, dm^1_x(y) \, d\nu_1(x) \\
& \qquad\qquad +\frac12\iint_{\{(x,y): u(y)-u(x)=0\}}|\nabla w(x,y)| \, dm^1_x(y) \, d\nu_1(x) \\
& \qquad\qquad\qquad\qquad\qquad\qquad - \int_X  {\rm div}_{m^2} (\vert \nabla u \vert^{p-2} \nabla u) \, w \, d\nu_2 \ge \int_X vw \, d\nu_1.
\end{align}
 Observe that  ${\rm div}_{m^2} (\vert \nabla u \vert^{p-2} \nabla u)\in L^{p'}(X,\nu_2)$. Consequently
\begin{align}
\frac12 &\iint_{\{(x,y): u(y)-u(x)\neq0\}} \hbox{sign}^0(u(y)-u(x)) \, \nabla w(x,y) \, dm^1_x(y) \, d\nu_1(x) \\
& \qquad\qquad +\frac12\iint_{\{(x,y): u(y)-u(x)=0\}}|\nabla w(x,y)| \, dm^1_x(y) \, d\nu_1(x) \\
& \qquad\qquad\qquad\qquad\qquad\qquad - \int_X \, \mu(x) \,{\rm div}_{m^2} (\vert \nabla u \vert^{p-2} \nabla u) \, w \, d\nu_1 \ge \int_X vw \, d\nu_1.
\end{align}
 We reorganise the terms in the above equation to obtain
\begin{align}\label{1239}
\int_X &(v+\mu{\rm div}_{m^2} (\vert \nabla u \vert^{p-2} \nabla u)\, w \, d\nu_1 \\
&\le  \frac12 \iint_{\{(x,y): u(y)-u(x)\neq0\}} \hbox{sign}^0(u(y)-u(x)) \, \nabla w(x,y) \, dm^1_x(y) \, d\nu_1(x) \\
&\quad\quad\quad\quad\quad\quad + \frac12 \iint_{\{(x,y): u(y)-u(x)=0\}} |\nabla w(x,y)| \, dm^1_x(y) \, d\nu_1(x).
\end{align}
Similarly, from~\eqref{19412} we get
\begin{align}\label{12392}
\int_X &(v+\mu{\rm div}_{m^2} (\vert \nabla u \vert^{p-2} \nabla u) \, w \, d\nu_1 \\
&\ge  \frac12 \iint_{\{(x,y): u(y)-u(x)\neq0\}} \hbox{sign}^0(u(y)-u(x)) \, \nabla w(x,y) \, dm^1_x(y) \, d\nu_1(x) \\
&\quad\quad\quad\quad\quad\quad - \frac12 \iint_{\{(x,y): u(y)-u(x)=0\}} |\nabla w(x,y)| \, dm^1_x(y) \, d\nu_1(x).
\end{align}
Now we easily conclude the proof of the claim \eqref{eq:subdifferentialofsum}. By taking $w=u$ in~\eqref{1239} and~\eqref{12392} we get
\begin{equation}\label{12551}
\int_X \left(v + \mu{\rm div}_{m^2} (\vert \nabla u \vert^{p-2} \nabla u)\right) \, u \, d\nu_1 = \mathcal{F}_{1,m^1}(u).
\end{equation}
On the other hand,
it follows from~\eqref{1239} that
\begin{equation}\label{1239r}\begin{array}{c}\displaystyle
 \int_X \left(v + \mu{\rm div}_{m^2} (\vert \nabla u \vert^{p-2} \nabla u)\right) \, w \, d\nu_1 \le \mathcal{F}_{1,m^1}(w)
\end{array}
\end{equation}
for any
$w\in  \hbox{Dom}(\mathcal{F}_{1,m^1})\cap\hbox{Dom}(\mathcal{F}_{p,m^2})$.

Let us see that $\eta:=v+\mu{\rm div}_{m^2} (\vert \nabla u \vert^{p-2} \nabla u) \in L^\infty(X,\nu_1)$  with $\| \eta \|_\infty \leq 1$. Indeed, take any $M > 1$ and suppose that there exists a $\nu_1$-measurable set $A_\varepsilon$ with $\nu(A_\varepsilon) \geq \varepsilon > 0$ such that $\eta \geq M$ on $A_\varepsilon$ (the case $\eta \leq -M$ is handled similarly). Then, taking $w = \1_{A_\varepsilon}$,
\begin{equation}\label{eq:estimateforeta}
\int_X \eta w \, d\nu_1 = \int_{A_\varepsilon} \eta \, d\nu_1 \geq M \nu_1(A_\varepsilon),
\end{equation}
yet estimating $\mathcal{F}_{1,m^1}(w)$ from above yields
\begin{equation}
\mathcal{F}_{1,m^1}(w) \leq \| w \|_{L^1(X,\nu_1)} = \nu_1(A_\varepsilon),
\end{equation}
and consequently estimate \eqref{eq:estimateforeta} may hold only if $0 \leq M \leq 1$. Therefore, we have that $\eta\in L^\infty(X,\nu_1)$ with $\| \eta \|_\infty \leq 1$.

{\flushleft \bf  Step 3.} We now show that in the cases when $\nu_1(X) < \infty$ or  $p \geq 2$ we have that
$$\eta=v+\mu{\rm div}_{m^2} (\vert \nabla u \vert^{p-2} \nabla u)\in L^2(X,\nu_1); $$
from this, in the next Step we will deduce that the equality \eqref{eq:subdifferentialofsum} holds true.  We split the proof into two parts corresponding to the assumptions  (a)-(b) in the statement of the Theorem.

{\flushleft (a)} First, observe that if $\nu_1(X)<+\infty$, we immediately obtain that $\eta \in L^2(X,\nu_1)$ as a consequence of the $L^\infty$ estimate in the previous Step (moreover, in this case all the above is not necessary for $p\le 2$, see Remark~\ref{lo001}).

{\flushleft (b)} Suppose now that $p\geq 2$; in this case, since $|\nabla u|^{p-1} \leq |\nabla u|$ for $|\nabla u| \leq 1$, we have
\begin{align}
|\eta(x)| &\leq |v(x)| + \mu(x) \bigg| \int_{X} |\nabla u(x,y)|^{p-2} \nabla u(x,y) \, dm^2_x(y) \bigg|
\\
&\le |v(x)| + \mu(x)\int_{\{y:|\nabla u(x,y)|\le 1\}} |\nabla u(x,y)|^{p-1} \, dm^2_x(y) \\ &\qquad\qquad\qquad\qquad\qquad\qquad\qquad +\mu(x) \int_{\{y:|\nabla u(x,y)|\ge 1\}}|\nabla u(x,y)|^{p-1} \, dm^2_x(y)
\\
&\le |v(x)| + \mu(x)\int_{\{y:|\nabla u(x,y)|\le 1\}} |\nabla u(x,y)| \, dm^2_x(y) \\ &\qquad\qquad\qquad\qquad\qquad\qquad\qquad +\mu(x) \int_{\{y:|\nabla u(x,y)|\ge 1\}}|\nabla u(x,y)|^{p-1} \, dm^2_x(y)
\\
&\le |v(x)|+ \underbrace{\mu(x) \int_{X}|\nabla u(x,y)| \, dm^2_x(y)}_{\rm I(x)} +  \underbrace{ \mu(x) \int_{X} |\nabla u(x,y)|^p \, dm^2_x(y)}_{\rm II(x)}.
\end{align}
By assumption, $v \in L^2(X,\nu_1)$. We will now show that the term ${\rm I}$ also lies in $L^2(X,\nu_1)$. Indeed, by the Jensen inequality,
\begin{align}
\int_X {\rm I(x)}^2 \, d\nu_1(x) &= \int_X \bigg(\mu(x)\int_{X}|\nabla u(x,y)| \,dm^2_x(y) \bigg)^2 d\nu_1(x) \\
&\leq \Vert \mu \Vert_\infty \int_X \mu(x) \int_{X}|\nabla u(x,y)|^2 \, dm^2_x(y) \, d\nu_1(x) \\
&= \Vert \mu \Vert_\infty \int_X \int_{X} |\nabla u(x,y)|^2 \, dm^2_x(y) \, d\nu_2(x),
\end{align}
which is finite since $u \in L^2(X, \nu_2)$.

Then, we show that the term ${\rm II}$ lies in $L^1(X,\nu_1)$. Indeed,
\begin{align}
\int_X {\rm II(x)} \, d\nu_1(x) &= \int_X \mu(x) \int_{X} |\nabla u(x,y)|^p \, dm^2_x(y) \, d\nu_1(x) \\
&= \int_X \int_{X} |\nabla u(x,y)|^p \, dm^2_x(y) \, d\nu_2(x),
\end{align}
which is finite since $u$ lies in the domain of $\mathcal{F}_{1,m^1} + \mathcal{F}_{p,m^2}$. Then, we arrive to
$$\eta=v+\mu{\rm div}_{m^2} (\vert \nabla u \vert^{p-2} \nabla u)\in \left(L^2(X,\nu_1)+L^1(X,\nu_1)\right)\cap L^\infty(X,\nu_1).$$
It is easy to see that
$$\left(L^2(X,\nu_1)+L^1(X,\nu_1)\right)\cap L^\infty(X,\nu_1)\subset L^2(X,\nu_1),$$
so we conclude that
\begin{equation}\label{hard02}\eta=v+\mu{\rm div}_{m^2} (\vert \nabla u \vert^{p-2} \nabla u)\in L^2(X,\nu_1).
\end{equation}

{\bf \flushleft Step 4.} Now we finally prove~\eqref{eq:subdifferentialofsum}. As shown in the previous step, in both of the cases $\nu_1(X) < \infty$ or  $p \geq 2$, we have that $\eta \in L^2(X,\nu_1)$, and also
\begin{equation}\label{hard01}{\rm div}_{m^2} (\vert \nabla u \vert^{p-2} \nabla u)\in L^2(X,\nu_2),
\end{equation}
which imply that
\begin{equation}
{\rm div}_{m^2} (\vert \nabla u \vert^{p-2} \nabla u)=\Delta_p^{m^2}u .
\end{equation}
Now, by a density argument we can pass from~\eqref{1239r} to
\begin{equation}\label{1239rbbb}
 \int_X (v+\mu\Delta_p^{m^2}u) \, w \, d\nu_1\le \mathcal{F}_{1,m^1}(w) \quad \forall w\in  \hbox{Dom}(\mathcal{F}_{1,m^1}).
\end{equation}
Indeed,   we have that $v+\mu\Delta_p^{m^2}u\in L^2(X,\nu_1)$,   then, for $w\in L^2(X,\nu_1)$, take $w_n\in L^2(X,\nu_1)\cap L^1(X,\nu_1)$ such that $w_n\to w$ in  $L^2(X,\nu_1)$. Then, from~\eqref{1239r} 
$$\begin{array}{c}\displaystyle
 \int_X \left(v + \mu{\rm div}_{m^2} (\vert \nabla u \vert^{p-2} \nabla u)\right) \, w_n \, d\nu_1 \le \mathcal{F}_{1,m^1}(w_n),
\end{array}
$$
and we can pass to the limit to get~\eqref{1239rbbb}.

 Then, by~\eqref{12551} and~\eqref{1239rbbb} we have that
$$\int_X (v+\mu\Delta_p^{m^2}u) (w-u) \, d\nu_1\le \mathcal{F}_{1,m^1}(w)-\mathcal{F}_{1,m^1}(u),$$
for any
$w\in  \hbox{Dom}(\mathcal{F}_{1,m^1})$,
which means that
$$v+\mu\Delta_p^{m^2}u\in\partial_H \mathcal{F}_{1,m^1}(u).$$
Consequently,
$$v\in-\Delta_1^{m^1}u-\mu\Delta_p^{m^2}u,$$
so the equality \eqref{eq:subdifferentialofsum} holds.

{\flushleft \bf  Step 5.}   It remains to show that the domain of $\partial_H \left(\mathcal{F}_{1,m^1}+\mathcal{F}_{p,m^2}\right)$ is dense in $H$. By Proposition \ref{A.Domain}, we have
\begin{align}
L^1(X,\nu_1)\cap L^p(X,\nu_2) \cap H &\subset D(\partial_H \left(\mathcal{F}_{1,m^1}+\mathcal{F}_{p,m^2}\right)) \subset D(\mathcal{F}_{1,m^1}+\mathcal{F}_{p,m^2}) \\
&\subset \overline{ D( \mathcal{F}_{1,m^1}+\mathcal{F}_{p,m^2})}^{H}  = \overline{ D( \partial_H \left(\mathcal{F}_{1,m^1}+\mathcal{F}_{p,m^2}\right)) }^{H},\end{align}
from where it follows that $D(\partial_H \left(\mathcal{F}_{1,m^1}+\mathcal{F}_{p,m^2}\right)) $ is dense in $H$.
\end{proof}

By Brezis-K\={o}mura's Theorem (Theorem \ref{BKTheorem}) and having in mind that the operator is completely accretive, as a consequence of the above theorem, we have the following existence and uniqueness result.

\begin{theorem}\label{0942cor}
 Suppose that the assumptions of Theorem \ref{0942} hold. Let~$T > 0$. For any $u_0\in L^2(X,\nu_1)$ and $f\in L^2(0,T;L^2(X,\nu_1))$, the following problem has a unique strong solution $u=u_{u_0,f}$:
\begin{equation}\left\{\begin{array}{ll}\label{eq:Cauchy1p}
u_t-\Delta^{m^1}_1u- \mu\Delta^{m^2}_pu\ni f&\hbox{on } [0,T]
\\[6pt]
u(0)=u_0.
\end{array}
\right.
\end{equation}
Moreover, we have the following comparison and contraction principle. If $u_0,\widetilde u_0\in L^2(X,\nu_1)$ and $f,\widetilde f\in L^2(0,T;L^2(X,\nu_1))$, if $u=u_{u_0,f}$ and $\tilde u=u_{\widetilde u_0,\widetilde f}$, we have that for any $0\le t\le T$ and any $1 \leq  q \leq \infty$
\begin{equation}\label{1416}\Big\Vert \left(u(t)-\widetilde u(t)\right)^+\Big\Vert_{L^q(X,\nu_1)}\le \Big\Vert \left(u_0-\widetilde u_0\right)^+\Big\Vert_{L^q(X,\nu_1)}
+\int_0^t \Big\Vert \left(f(s)-\widetilde f(s)\right)^+\Big\Vert_{L^q(X,\nu_1)} ds.
\end{equation}
\end{theorem}

\subsubsection{Asymptotic behaviour}\label{firsab01}
 Our next goal is to study the asymptotic behaviour of the unique solution to the Cauchy problem
\begin{equation}\label{Cauchy1p}
\left\{\begin{array}{ll}
u_t-\Delta^{m^1}_1u- \mu\Delta^{m^2}_pu\ni 0&\hbox{on } [0,T],
\\[6pt]
u(0)=u_0
\end{array}
\right.
\end{equation}
 under the assumptions of Theorem \ref{0942}.
First, let us see that if the measure of the space is finite, the mass of the solution is conserved.

\begin{lemma}\label{ConsMass}
Assume that  $\nu_1(X) < \infty$. Then, if $u(t)$ is the solution of problem \eqref{Cauchy1p}, we have
$$\int_X  u(t)(x) \, d \nu_1(x) = \int_X u_0 \, d \nu_1(x)\quad \hbox{for all } \ t \geq 0.$$
\end{lemma}

\begin{proof}
 By the definition of solution, writing explicitly the term involving $\Delta^{m^1}_1$, there exists $\g_t \in L^\infty(X \times X, \nu_1 \otimes m^1_x)$ antisymmetric such that
$$\Vert \g_t \Vert_{L^\infty(X \times X, \nu_1 \otimes m^1_x)} \leq 1$$
 and
$$\frac{d}{dt}u(t)(x)  = \int_X \g_t(x,y) \, dm^1_x(y)+ \mu(x) \cdot \Delta^{m^2}_pu(t)(x).$$
Hence,
\begin{align}
\frac{d}{dt} \int_X u(t)(x) \, d \nu_1(x) &= \int_X  \int_X \g_t(x,y) \, dm^1_x(y) \,d \nu_1(x) \\
&\qquad\quad + \int_X  \mu(x) \, \int_X \vert \nabla u(t)(x,y) \vert^{p-2} \, \nabla u(t)(x,y) \,  dm^2_x(y) \, d \nu_1(x) \\
&= \int_X  \int_X \g_t(x,y) \, dm^1_x(y) \, d \nu_1(x) \\
&\qquad\quad + \int_X  \int_X \vert \nabla u(t)(x,y) \vert^{p-2} \, \nabla u(t)(x,y) \, dm^2_x(y) \, d \nu_2(x) = 0.
\end{align}
Indeed, since $\g_t$ is antisymmetric and $\nu_1$ is reversible with respect to $m^1_x$ the first integral equals zero;  similarly, since $\vert \nabla u(t) \vert^{p-2} \, \nabla u(t)$ is antisymmetric and $\nu_2$ is reversible with respect to $m^2_x$, the second integral is also equal to zero.
\end{proof}

The conservation of mass property required no geometric assumptions on the space; if we additionally assume that $\mathcal{F}_{1,m^1}$ satisfies a $1$-Poincar\'{e} inequality, we may show that the solution converges to the mean of the initial data. The proof is similar to the proof of Theorem 3.24 in~\cite{MSTBook} (we give it here for the sake of completeness).

\begin{theorem}\label{thm:asymptotics1poincare1p}
Assume that  $\nu_1(X) < \infty$ and $\mathcal{F}_{1,m^1}$ satisfies a  $1$-Poincar\'{e} inequality. For  $u_0 \in L^2(X,\nu_1)$, let $u(t)$ be the  solution of the Cauchy problem \eqref{Cauchy1p}. Then, we have
 $$|| u(t) - \overline{u_0}||_{L^1(X,\nu_1)}\le \frac{1}{2\lambda_1( \mathcal{F}_{1,m^1})} \frac{||u_0||_{L^2(X,\nu_1)}^2}{t}\quad\forall t> 0.$$
\end{theorem}

\begin{proof}
By complete accretivity of the operator $-\Delta_1^{m^1} - \mu \Delta_p^{m^2}$, we have the following contraction estimate
\begin{equation}
\| u(t) - v(t) \|_{L^1(X,\nu_1)} \leq \| u(s) - v(s) \|_{L^1(X,\nu_1)}
\end{equation}
for all solutions $u(t)$, $v(t)$ of the Cauchy problem \eqref{Cauchy1p} and all $t > s \geq 0$. Setting $v_0 := \overline{u_0}$, it is easy to see that $v(t) = \overline{u_0}$ for all $t > 0$, and consequently
\begin{equation}
\| u(t) - \overline{u_0} \|_{L^1(X,\nu_1)} \leq \| u(s) - \overline{u_0} \|_{L^1(X,\nu_1)}.
\end{equation}
Integrating the above inequality over $s \in (0,t)$ yields
\begin{equation}
t \, \| u(t) - \overline{u_0} \|_{L^1(X,\nu_1)} \leq \int_0^t \| u(s) - \overline{u_0} \|_{L^1(X,\nu_1)} \, ds,
\end{equation}
and since the mass is conserved along the solution, by the $1$-Poincar\'e inequality we have
\begin{align}\label{eq:asymptotics1p1poincare}
t \, \| u(t) - \overline{u_0} \|_{L^1(X,\nu_1)} &\leq \frac{1}{\lambda_1(\mathcal{F}_{1,m^1})} \int_0^t \mathcal{F}_{1,m^1}(u(s)) \, ds \\ &\leq \frac{1}{\lambda_1(\mathcal{F}_{1,m^1})} \int_0^t \bigg( \mathcal{F}_{1,m^1}(u(s)) + p \mathcal{F}_{p,m^2}(u(s)) \bigg) \, ds.
\end{align}
Using integration by parts, let us now calculate
\begin{align}
\frac{d}{dt} \| u(t) \|_{L^2(X,\nu_1)}^2 &= \frac{d}{dt} \int_X u(t)^2 \, d\nu_1 = 2 \int_X u(t) \cdot \frac{d}{dt} u(t) \, d\nu_1 \\
&= 2 \int_X u(t) \cdot \bigg( \int_X \g_t \, dm_x^1 + \mu \Delta_p^{m^2} u(t) \bigg) \, d\nu_1 \\
&= 2 \int_{X} u(t) \cdot \bigg( \int_X \g_t \, dm_x^1 \bigg) \, d\nu_1 + 2 \int_X u(t) \cdot \Delta_p^{m^2} u(t) \, d\nu_2 \\
&= 2 \int_{X} u(t) \cdot (-v(t)) \, d\nu_1 - \int_{X \times X} |\nabla u(t)|^p \, d(\nu_2 \otimes m_x^2) \\
&= - 2 \mathcal{F}_{1,m^1}(u(t)) - 2p \mathcal{F}_{p,m^2}(u(t)),
\end{align}
where $\g_t$ is the antisymmetric function from the definition of $\Delta_1^m$ and $v(t) \in \partial_H \mathcal{F}_{1,m^1}(u)$ is the function associated to $\g_t$ through Proposition \ref{CHarat1}. Consequently,
\begin{equation}
\frac{d}{dt} \| u(t) \|_{L^2(X,\nu_1)}^2 = - 2 \mathcal{F}_{1,m^1}(u(t)) - 2p \mathcal{F}_{p,m^2}(u(t)),
\end{equation}
and integrating this equation over $(0,t)$ we obtain
\begin{align}
2 \int_0^t \bigg( \mathcal{F}_{1,m^1}(u(s)) + p \mathcal{F}_{p,m^2}(u(s)) \bigg) \, ds &= \| u_0 \|_{L^2(X,\nu_1)}^2 - \| u(t) \|_{L^2(X,\nu_1)}^2 \leq \| u_0 \|_{L^2(X,\nu_1)}^2.
\end{align}
Combining this with estimate \eqref{eq:asymptotics1p1poincare}, we obtain
\begin{equation}
t \, \| u(t) - \overline{u_0} \|_{L^1(X,\nu_1)} \leq \frac{1}{\lambda_1(\mathcal{F}_{1,m^1})} \cdot \frac{1}{2} \| u_0 \|_{L^2(X,\nu_1)}^2,
\end{equation}
which yields the claim.
\end{proof}

Under the assumption that $\mathcal{F}_{1,m^1}$ satisfies a  $2$-Poincar\'{e} inequality we have finite extinction time (this means that the mean of the initial data is attained in finite time, this time is called the extinction time). The proof of the following result is similar to the one for Theorem 3.27 in~\cite{MSTBook} (we give it for the sake of completeness).

\begin{theorem}\label{thm:asymptotics2poincare1p}
Assume that $\nu_1(X) < \infty$ and $\mathcal{F}_{1,m^1}$ satisfies a  $2$-Poincar\'{e} inequality. For  $u_0 \in L^2(X,\nu_1)$, let $u(t)$ be the  solution
 of the Cauchy problem \eqref{Cauchy1p}. Then, we have
$$\Vert u(t)- \overline{u_0} \Vert_{L^2(X,\nu_1)} \leq \left( \Vert u_0 - \overline{u_0}\Vert_{L^2(X,\nu_1)}- \lambda_2(\mathcal{F}_{1,m^1}) \, t \right)^+ \quad \forall t>0.$$
In particular, if we denote by
$$T_{\rm ex}(u_0):= \inf \{ T > 0: \, u(t) = \overline{u_0}  \  \ \forall   t \geq T \}$$
the extinction time, we have the following bound:
$$T_{\rm ex}(u_0) \le  \frac{\Vert u_0 -\overline{u_0} \Vert_{L^2(X,\nu_1)}}{\lambda_2(\mathcal{F}_{1,m^1})}.$$
\end{theorem}

\begin{proof}
Let $v(t):= u(t) - \overline{u_0}$. Then, we have
\begin{equation}\label{eq:differentialinclusionasymptotics}
\frac{d}{dt}v(t)(x) \in \Delta^{m^1}_1v(t)(x) + \mu(x) \Delta^{m^2}_p v(t)(x), \quad \hbox{for} \ \ t \geq 0.
\end{equation}
 Hence, writing explicitly the term involving $\Delta^{m^1}_1$, there exists $\g_t \in L^\infty(X \times X, \nu_1 \otimes m^1_x)$ antisymmetric such that
$$\Vert \g_t \Vert_{L^\infty(X \times X, \nu_1 \otimes m^1_x)} \leq 1$$
and
\begin{equation}\label{perfect1}
\mathcal{F}_{1,m^1}(v(t)) = - \int_X \g_t(x,y) \, v(t)(x) \, d(\nu_1 \otimes m^1_x)(x,y),
\end{equation}
 and that inclusion \eqref{eq:differentialinclusionasymptotics} takes the form
\begin{equation}\label{perfect2}
\frac{d}{dt}v(t)(x)   = \int_X \g_t(x,y) \, dm^1_x(y)+ \mu(x) \Delta^{m^2}_pv(t)(x).
\end{equation}
Multiplying  equation \eqref{perfect2} by $v(t)$, integrating with respect to $\nu_1$ and having in mind \eqref{perfect1}, we get
\begin{align}
\frac12 \frac{d}{dt} &\int_X v(t)(x)^2 \, d\nu_1(x) \\
&= \int_X  \int_X \g_t(x,y) \, v(t)(x) \, dm^1_x(y) \, d \nu_1(x) + \int_X \mu(x) \cdot \Delta^{m^2}_p v(t)(x) \cdot v(t)(x) \, d \nu_1(x) \\
&= - \mathcal{F}_{1,m^1}(v(t)) + \int_X  \Delta^{m^2}_p v(t)(x) \cdot v(t)(x) \, d \nu_2(x) \\
&= - \mathcal{F}_{1,m^1}(v(t)) - \frac12  \int_X \int_X \vert \nabla v(t) (x,y) \vert^p \, dm^2_x(y) \, d\nu_2(x) \\
&\leq  - \mathcal{F}_{1,m^1}(v(t)).
\end{align}
Thus,
$$\frac12 \frac{d}{dt} \int_X v(t)(x)^2 \, d \nu_1(x) + \mathcal{F}_{1,m^1}(v(t)) \leq 0.$$
Now, by Lemma \ref{ConsMass}, we have $\overline{u_0} = \overline{u(t)}$ for all $t \geq 0$.  By assumption, the functional $\mathcal{F}_{1,m^1}$ satisfies a  $2$-Poincar\'{e} inequality, so there exists $\lambda_2(\mathcal{F}_{1,m^1}) >0$ such that
$$\lambda_2(\mathcal{F}_{1,m^1}) \Vert v(t) \Vert_{L^2(X,\nu_1)} \leq \mathcal{F}_{1,m^1}(v(t)) \quad \hbox{for all} \ t >0.$$
Therefore, we have
$$\frac12 \frac{d}{dt} \Vert v(t) \Vert^2_{L^2(X,\nu_1)} + \lambda_2(\mathcal{F}_{1,m^1})\, \Vert v(t) \Vert_{L^2(X,\nu_1)}\leq 0.$$
Now, integrating this ordinary differential  inequality, we get
$$\Vert v(t) \Vert_{L^2(X,\nu_1)} \leq \left( v(0) - \lambda_2(\mathcal{F}_{1,m^1}) \, t\right)^+ \quad \hbox{for all} \ t \geq 0,$$
that is
$$\Vert u(t)- \overline{u_0} \Vert_{L^2(X,\nu_1)} \leq \left( \Vert u_0 - \overline{u_0}\Vert_{L^2(X,\nu_1)}- \lambda_2(\mathcal{F}_{1,m^1}) \, t \right)^+ \quad \hbox{for all} \ t \geq 0.$$
The bound on the extinction time follows immediately.
\end{proof}
%
%
%
%
%

\begin{remark}\rm It is well known   that the gradient flow of energy functional of the form $\mathcal{F}_{p,m}$ has finite extinction time  if and only if $1 \leq p < 2$ (see \cite{BB} or \cite{MSTBook}).  The  above result shows that in the case of  the gradient flow of an energy functional of the form $\mathcal{F}_{1,m^1}+\mathcal{F}_{p,m^2}$ we  have   finite extinction time for all $p$  with at most linear rate of convergence, i.e., we have the same asymptotic behaviour  as for the functional $\mathcal{F}_{1,m^1}$.\hfill$\blacksquare$
\end{remark}

\subsection{ The $(1,1)$-Laplace operator}\label{11problem}

 In this subsection, we consider the evolution governed by the sum of two $1$-Laplacian type operators defined on different random walk structures. In other words, our goal is to solve the equation
$$ u_t(t,x) - \int_{X}\frac{u(y)-u(x)}{|u(y)-u(x)|} \, dm^1_x(y) -
\mu(x) \int_{X}\frac{u(y)-u(x)}{|u(y)-u(x)|} \, dm^2_x(y) = f(t,x).$$
We continue assuming that
$$\mu:=\frac{d\nu_2}{d\nu_1}\in L^\infty(X,\nu_1).$$
The techniques used here are different to the ones used for $p>1$. These techniques now are based on the homogeneity of the operator involved.

\subsubsection{The gradient flow}

Our first aim is to prove an equivalent characterisation of the functional $\mathcal{F}_{1,m^1} + \mathcal{F}_{1,m^2}$. Using a similar argument as in the proof of \cite[Proposition 3.9]{MSTBook}, we have the following result.

\begin{proposition}\label{LSC+2}
Given $u \in BV_{m^1}(X, \nu_1) \cap BV_{m^2}(X, \nu_2) \cap H$, we have
\begin{align}
&\mathcal{F}_{1,m^1}(u) + \mathcal{F}_{1,m^2}(u) \\
&\qquad\quad  = \sup\bigg\{\int_X u \, \mathrm{div}_{m^1} \, \z_1 \, d\nu_1 + \int_X u \, \mathrm{div}_{m^2} \, \z_2 \, d\nu_2:  \\
&\qquad\qquad\qquad\qquad\qquad\quad \   {\rm div}_{m^i} \z_i \in L^2(X, \nu_i), \,\, || \z_i||_{L^\infty(X \times X, \nu_i \otimes m^i_x)} \le 1, \,\, i=1,2 \bigg\}.
\end{align}
\end{proposition}

\begin{proof}
Let
$u \in BV_{m^1}(X, \nu_1) \cap BV_{m^2}(X, \nu_2) \cap H.$
  Given $\z_i \in L^\infty(X, \nu_i)$, where $i=1,2$, such that
${\rm div}_{m^i} \z_i \in L^2(X, \nu_i)$
and
$ || \z_i||_{L^\infty(X \times X, \nu_i \otimes m^i_x)} \leq 1$
for $i = 1,2$,  by the Green formula \eqref{GreenF} we have
\begin{align}
\int_X u \, \mathrm{div}_{m^1} \, &\z_1 \, d\nu_1 + \int_X u \, \mathrm{div}_{m^2} \, \z_2 \, d\nu_2  \\
& = - \frac12 \int_{X \times X} \z_1(x,y) \, \nabla u(x,y) \, d(\nu_1\otimes m^1_x)(x,y) \\
& \qquad\qquad\qquad\qquad\qquad - \frac12 \int_{X \times X} \z_2(x,y) \, \nabla u(x,y) \, d(\nu_2\otimes m^2_x)(x,y) \\
& \leq \frac12 \int_{X \times X} \vert \nabla u(x,y) \vert \, d(\nu_1\otimes m^1_x)(x,y) \\
& \qquad\qquad\qquad\qquad\qquad + \frac12 \int_{X \times X} \vert \nabla u(x,y) \vert \, d(\nu_2\otimes m^2_x)(x,y) \\
& = \mathcal{F}_{1,m^1}(u) + \mathcal{F}_{1,m^2}(u).
\end{align}
Therefore,
\begin{align*}
\sup\bigg\{ &\int_X u \, \mathrm{div}_{m^1} \, \z_1 \, d\nu_1 + \int_X u \, \mathrm{div}_{m^2} \, \z_2 \, d\nu_2:  \\
& \,\,\,\,\, {\rm div}_{m^i} \z_i \in L^2(X, \nu_i), \, || \z_i||_{L^\infty(X \times X, \nu_i \otimes m^i_x)} \le 1, \, i=1,2 \bigg\} \leq \mathcal{F}_{1,m^1}(u) + \mathcal{F}_{1,m^2}(u).
\end{align*}
On the other hand, since $\nu_1$ is $\sigma$-finite,
there exists a sequence of measurable sets
$K_1 \subset K_2 \subset \ldots \subset K_n \subset \ldots$
of $\nu_1$-finite measure (and consequently also $\nu_2$-finite measure) such that
$X= \bigcup_{n=1}^\infty K_n.$
Then, if we define
$$\z_n(x,y):= {\rm sign}_0 (u(y) - u(x)) \1_{K_n \times K_n}(x,y),$$
we have $\z_n \in X_{m^i}^2(X, \nu_i)$ with $\Vert \z_n \Vert_\infty \leq 1$, $i=1,2$, and
\begin{align}
\mathcal{F}_{1,m^1}&(u) + \mathcal{F}_{1,m^2}(u) \\
& = \frac12 \int_{X \times X} \vert \nabla u(x,y) \vert \, d(\nu_1\otimes m^1_x)(x,y) + \frac12 \int_{X \times X} \vert \nabla u(x,y) \vert \, d(\nu_2\otimes m^2_x)(x,y) \\
& = \lim_{n \to \infty} \bigg(\frac12 \int_{K_n \times K_n} \vert \nabla u(x,y) \vert \, d(\nu_1\otimes m^1_x)(x,y) \\
& \qquad\qquad\qquad\qquad\qquad\qquad\qquad + \frac12 \int_{K_n \times K_n} \vert \nabla u(x,y) \vert \, d(\nu_2\otimes m^2_x)(x,y) \bigg) \\
& = \lim_{n \to \infty} \bigg(\frac12 \int_{X \times X} \nabla u(x,y) \, \z_n(x,y) \, d(\nu_1\otimes m^1_x)(x,y) \\
& \qquad\qquad\qquad\qquad\qquad\qquad\qquad + \frac12 \int_{X \times X}  \nabla u(x,y) \z_n(x,y)\, d(\nu_2\otimes m^2_x)(x,y) \bigg) \\
& = \lim_{n \to \infty} \left(\int_X u(x) \, {\rm div}_{m^1} (-\z_n) \, d\nu_1(x) + \int_X u(x) \, {\rm div}_{m^2} (-\z_n) \, d\nu_2(x) \right) \\
&   \leq \sup\bigg\{ \int_X u \, \mathrm{div}_{m^1} \, \z_1 \, d\nu_1 + \int_X u \, \mathrm{div}_{m^2} \, \z_2 \, d\nu_2:   \\
&   \qquad\qquad\qquad\qquad\qquad {\rm div}_{m^i} \z_i \in L^2(X, \nu_i), \,\, || \z_i||_{L^\infty(X \times X, \nu_i \otimes m^i_x)} \le 1, \,\, i=1,2 \bigg\},
\end{align}
  which concludes the proof.
\end{proof}

As a consequence of the above result, we see that the convex functional
$$ \mathcal{F}_{1,m^1} + \mathcal{F}_{1,m^2}$$
is lower semicontinuous in the space $H = L^2(X,\nu_1)$. Therefore, the subdifferential $\partial_H(\mathcal{F}_{1,m^1} + \mathcal{F}_{1,m^2})$ is a maximal monotone operator on~$H$.
In the next result we characterise this subdifferential in terms of the subdifferentials of $\mathcal{F}_{1,m^1}$ and $\mathcal{F}_{1,m^2}$.

\begin{theorem}\label{0942NEW} We have
\begin{equation}\label{eq:subdifferentialofsumNEW}
\partial_H \left(\mathcal{F}_{1,m^1}+\mathcal{F}_{1,m^2}\right) = -\Delta^{m^1}_1- \mu\Delta^{m^2}_1.
\end{equation}
Furthermore, this operator is completely accretive and has a dense domain in $H$.
\end{theorem}

\begin{proof}
{\bf Step 1.} The complete accretivity of the operator and the density of its domain can be shown similarly to the proof of Theorem \ref{0942}. Therefore, we will only prove the equality \eqref{eq:subdifferentialofsumNEW}. To this end, we define the functionals $\Psi_i: H \rightarrow [0, \infty]$,   $i = 1,2$,  as
\begin{align}
\Psi_1 (v):= \inf \bigg\{ \max \{\Vert \z_1 \Vert_\infty, \Vert \z_2\Vert_\infty \}: & \,\, v = - {\rm div}_{m^1} \z_1- \mu {\rm div}_{m^2} \z_2, \\
& \ \z_i \in L^\infty(X\times X,\nu_i\otimes m^i_x), \ i=1,2 \bigg\}
\end{align}
and
\begin{align}
\Psi_2 (v):= \inf \bigg\{ & \max \{\Vert \z_1 \Vert_\infty, \Vert \z_2\Vert_\infty \}: \,\, v = - {\rm div}_{m^1} \z_1- \mu {\rm div}_{m^2} \z_2, \\
& \,\, \z_i \in L^\infty(X\times X,\nu_i\otimes m^i_x), \ {\rm div}_{m^i} \z_i \in L^2(X, \nu_i), \ i=1,2 \bigg\}.
\end{align}
For brevity, we set
$$\mathcal{F}_{m} := \mathcal{F}_{1,m^1}+\mathcal{F}_{1,m^2}.$$
Note that the functional $\mathcal{F}_m$ is proper, convex, lower semicontinuous and positive homogeneous of degree $1$. We claim that (see~\eqref{phitilde} for the definition of  $\widetilde{\mathcal{F}_{m}}$)
\begin{equation}\label{iqual1}
\widetilde{\mathcal{F}_{m}} =   \Psi_1 = \Psi_2.
\end{equation}
Indeed, let $v \in H$ be such that $  \Psi_1(v)  < \infty$.   Applying the direct method of the calculus of variations,  it is easy to see that there exist  $\z_i \in L^\infty(X\times X,\nu_i\otimes m^i_x)$,   $i = 1,2$,  with
$$v = - {\rm div}_{m^1} \z_1- {\mu} {\rm div}_{m^2} \z_2$$
and
$$\Psi_1(v) = \max \{\Vert \z_1 \Vert_\infty ,\Vert \z_2\Vert_\infty\}.$$
Fix any $w \in H$. Then, integrating by parts we have
\begin{align}
\int_X w v \, d \nu_1 &=   - \int_X w(x) \, {\rm div}_{m^1} \z_1 (x) \, d \nu_1(x)  \\
&\qquad\qquad\qquad\qquad\qquad - \int_X w(x) \,  {\mu(x)} \, {\rm div}_{m^2} \z_2 (x) \, d \nu_1(x) \\
&=\frac12 \int_{X \times X} \nabla w(x,y) \, \z_1(x,y) \, d(\nu_1\otimes m^1_x)(x,y) \\
&\qquad\qquad\qquad\qquad\qquad + \frac12 \int_{X \times X} \nabla w(x,y) \, \z_2(x,y) \, d(\nu_2\otimes m^2_x)(x,y) \\
&\leq \mathcal{F}_{1,m^1}(w) \Vert \z_1 \Vert_\infty +  \mathcal{F}_{1,m^2}(w) \Vert \z_2 \Vert_\infty \leq \mathcal{F}_{m}(w) \max \{\Vert \z_1 \Vert_\infty ,\Vert \z_2\Vert_\infty \}.
\end{align}
Taking a supremum in $w$, we obtain the inequality
$$ \widetilde{\mathcal{F}_{m}}(v) \leq \Psi_1(v);$$
since by definition $\Psi_1 \leq \Psi_2$, we have the following chain of inequalities
$$ \widetilde{\mathcal{F}_{m}}(v) \leq \Psi_1(v) \leq \Psi_2(v).$$
We now prove that they are, in fact, equalities. Set
$$D= \left\{ {\rm div}_{m^1} \z_1+ {\mu} {\rm div}_{m^2} \z_2: \ \z_i \in L^\infty(X\times X,\nu_i\otimes m^i_x), \ {\rm div}_{m^i} \z_i \in L^2(X, \nu_i) \right\}.$$
Then,
\begin{align}
\tilde{\Psi}_2(v) &= \sup_{w \in H} \frac{\displaystyle\int_X w v \, d \nu_1} {\Psi_2(w)} \geq \sup_{w \in D, \Psi_2(w) < \infty} \frac{\displaystyle\int_X w v \, d \nu_1} {\Psi_2(w)} \\
&= \sup_{\substack{ \z_i \in L^\infty(X\times X,\nu_i\otimes m^i_x) \\ {\rm div}_{m^i} \z_i \in L^2(X, \nu_i)}} \frac{ \displaystyle \int_X  v \bigg({\rm div}_{m^1} \z_1+ {\mu} {\rm div}_{m^2} \z_2 \bigg) \, d\nu_1}{\max \{\Vert \z_1 \Vert_\infty ,\Vert \z_2\Vert_\infty\} } \\
&= \sup_{\substack{\z_i \in L^\infty(X\times X,\nu_i\otimes m^i_x) \\ {\rm div}_{m^i} \z_i \in L^2(X, \nu_i)}}  \frac{\displaystyle\int_X v \, {\rm div}_{m^1} \z_1 \, d \nu_1 +  \int_X v \, {\rm div}_{m^2} \z_2 \, d\nu_2 }{\max \{\Vert \z_1 \Vert_\infty ,\Vert \z_2\Vert_\infty\} } \geq \mathcal{F}_{m}(v),
\end{align}
where the last inequality is a consequence of Proposition \ref{LSC+2}. Hence, by Proposition \ref{degree1}, we have
$$\Psi_2 (v) = \widetilde{\widetilde{\Psi_2}}(v)\leq \widetilde{\mathcal{F}_{m}}(v).$$
Hence, we have
$$   \Psi_2 (v)  \leq \widetilde{\mathcal{F}_{m}}(v) \leq \Psi_1 (v) \leq \Psi_2 (v),$$
and the claim \eqref{iqual1} holds.

{\flushleft \bf Step 2.} We are now in position to apply Proposition \ref{degree1} to compute the subdifferential of $\mathcal{F}_m$. Since $\widetilde{\mathcal{F}_{m}} = \Psi_2$, it follows that
$$v \in \partial_H \mathcal{F}_{m} (u) \iff \Psi_2 (v) \leq 1 \ \hbox{and} \ \int_X uv \, d\nu_1 = \mathcal{F}_{m} (u).$$
Thus, for $v \in \partial_H \mathcal{F}_{m} (u)$, there exist $\z_i \in L^\infty(X\times X,\nu_i\otimes m^i_x)$, $i = 1,2$,  with
$$\Vert \z_i \Vert_\infty \leq 1$$ and
$${\rm div}_{m^i} \z_i \in L^2(X, \nu_i),$$
such that
$$v = - {\rm div}_{m^1} \z_1- {\mu} {\rm div}_{m^2} \z_2$$
and we have
\begin{align}
\mathcal{F}_{1,m^1}(u) + \mathcal{F}_{1,m^2}(u) &=  \int_X uv \, d\nu_1 =  - \int_X u \,{\rm div}_{m^1} \z_1 \, d\nu_1 - \int_X u \, {\mu} \, {\rm div}_{m^2} \z_2 \, d\nu_1 \\
&= - \int_X u \, {\rm div}_{m^1} \z_1 \, d\nu_1  - \int_X u \, {\rm div}_{m^2} \z_2 \, d\nu_2.
\end{align}
Hence, since $\Vert \z_i \Vert_\infty \leq 1$ for $i=1,2$, we obtain that
$$\mathcal{F}_{1,m^1}(u) = - \int_X u \, {\rm div}_{m^i} \z_i \, d\nu_i \quad \hbox{for $i = 1,2$.}$$
Now, given $w \in BV_{m^1}(X,\nu_1)$, by integration by parts we have
$$- \int_X w(x) \, {\rm div}_{m^1} \z_1 (x) \, d \nu_1(x) = \frac12 \int_{X \times X} \nabla w(x,y) \, \z_1(x,y) \, d(\nu_1\otimes m^1_x)(x,y) \leq \mathcal{F}_{1,m^1}(w).$$
Hence,
$$\mathcal{F}_{1,m^1}(w) - \mathcal{F}_{1,m^1}(u) \geq \int_X (- {\rm div}_{m^1} \z_1 (x)) (w(x) - u(x)) \, d\nu_1(x).$$
Similarly, for $w \in BV_{m^2}(X,\nu_2)$ we obtain that
$$ \mathcal{F}_{1,m^2}(w) - \mathcal{F}_{1,m^2}(u) \geq \int_X (- {\rm div}_{m^2} \z_2 (x)) (w(x) - u(x)) \, d\nu_2(x).$$
Therefore,
$$- {\rm div}_{m^1} \z_1 \in \partial_{L^2(X, \nu_1)}  \mathcal{F}_{1,m^1}(u) \quad \hbox{and} \quad - {\rm div}_{m^2} \z_2 \in \partial_{L^2(X, \nu_2)}  \mathcal{F}_{1,m^2}(u).$$
Consequently,
$$v \in -\Delta^{m^1}_1 u- \mu\Delta^{m^2}_1 u,$$
which concludes the proof.
\end{proof}

As a consequence of the above theorem, by the Brezis-K\={o}mura Theorem, we have the following existence and uniqueness result.

\begin{theorem}\label{0942corI}  Let $T>0$. For any $u_0\in L^2(X,\nu_1)$ and $f\in L^2(0,T;L^2(X,\nu_1))$, the following problem has a unique strong solution $u=u_{u_0,f}$:
$$\left\{\begin{array}{ll}
u_t-\Delta^{m^1}_1u- \mu\Delta^{m^2}_1u\ni f&\hbox{on } [0,T]
\\[6pt]
u(0)=u_0.
\end{array}
\right.
$$
Moreover, we have the following comparison and contraction principle. For initial data $u_0,\widetilde u_0\in L^2(X,\nu_1)$ and for $f,\widetilde f\in L^2(0,T;L^2(X,\nu_1))$, if $u=u_{u_0,f}$ and $\tilde u=u_{\widetilde u_0,\widetilde f}$, we have that for any $0\le t\le T$ and any $1 \leq  q \leq \infty$
\begin{equation}\label{1416I}\Big\Vert \left(u(t)-\widetilde u(t)\right)^+\Big\Vert_{L^q(X,\nu_1)}\le \Big\Vert \left(u_0-\widetilde u_0\right)^+\Big\Vert_{L^q(X,\nu_1)}
+\int_0^t \Big\Vert \left(f(s)-\widetilde f(s)\right)^+\Big\Vert_{L^q(X,\nu_1)} ds.
\end{equation}
\end{theorem}

\subsubsection{Asymptotic behaviour} The  study of
the asymptotic behaviour of the
solutions of problem
\begin{equation}\label{ACP1-1}\left\{\begin{array}{ll}
u_t-\Delta^{m^1}_1u- \mu\Delta^{m^2}_1u\ni 0&\hbox{on } [0,T],
\\[6pt]
u(0)=u_0,
\end{array}
\right.
\end{equation}
is similar to the that of Subsection~\ref{firsab01}.  It is easy to check that the mass is conserved if $\nu_1(X) < \infty$. Then, working similarly as in the proof of Theorem \ref{thm:asymptotics1poincare1p}, we get the following result.

\begin{theorem}
Assume that $\nu_1(X) < \infty$ and $\mathcal{F}_{1,m^1}$ satisfies a $1$-Poincar\'{e} inequality. For $u_0 \in L^2(X,\nu_1)$, let $u(t)$ be the solution of the Cauchy problem \eqref{ACP1-1}. Then, we have
 $$||  u(t) - \overline{u_0}||_{L^1(X,\nu_1)}\le \frac{1}{2\lambda_1( \mathcal{F}_{1,m^1})}\frac{||u_0||_{L^2(X,\nu_1)}^2}{t}\quad\forall t> 0.$$
\end{theorem}

Since the operator $\mathcal{F}_{m}$ is $1$-homogeneous, in the case when a $2$-Poincar\'e inequality is satisfied, we can (to the same effect) either follow the strategy of the proof of Theorem \ref{thm:asymptotics2poincare1p} or apply the general results of Bungert and Burger \cite{BB} to get the following result.

\begin{theorem}
Assume that  $\nu_1(X) < \infty$ and $\mathcal{F}_{1,m^1}$ satisfies a  $2$-Poincar\'{e} inequality. For $u_0 \in L^2(X,\nu_1)$, let $u(t)$ be the  solution
 of the Cauchy problem \eqref{ACP1-1}. Then, we have
$$\Vert u(t) -\overline{u_0} \Vert_{L^2(X,\nu_1)} \leq (\Vert u_0 -\overline{u_0} \Vert_{L^2(X,\nu_1)} - \lambda_2(\mathcal{F}_{1,m^1}) \, t)^+, \quad \forall t>0.$$
\end{theorem}
%

\subsection{The $(q,p)$-Laplace operator.}\label{subsec:qplaplace}

One may also apply the technique developed in this paper to study the problem of mixing a $m$-$q$-Laplacian  ($1 < q \leq 2$), instead of the $m$-$1$-Laplacian, with a $m$-$p$-Laplacian ($p \geq q$). The overall approach is very similar to the one in Section \ref{subsec:1plaplace}, therefore we  restrict ourselves to stating the results and highlighting these parts of the proofs which require a different argument. To be precise, in the notation  similar to the one of Theorem \ref{0942}, consider the functionals $\mathcal{F}_{q,m^1}: L^2(X,\nu_1)\to (-\infty,+\infty]$ given by
$$ \mathcal{F}_{q,m^1}(u):=  \left\{\begin{array}{ll}\displaystyle\frac{1}{2q} \int_{X \times X}|u(y)-u(x)|^q \,  \, d(\nu_1\otimes m^1_x)(x,y)   &\hbox{ if} \ |\nabla u|^q \in L^1(X \times X,\nu_1\otimes   m_x^1), \\[10pt] +\infty   &\hbox{ otherwise}, \end{array} \right.$$
and  $ \mathcal{F}_{p,m^2}:L^2(X,\nu_1)\to (-\infty,+\infty]$ given by
$$ \mathcal{F}_{p,m^2}(u):= \left\{\begin{array}{ll}\displaystyle\frac{1}{2p}\int_{X\times X}|u(y)-u(x)|^p \,  \, d(\nu_2\otimes m^2_x)(x,y)  &\hbox{ if $|\nabla u|^p \in L^1(X \times X,\nu_2\otimes  m_x^2)$,} \\[10pt] +\infty  &\hbox{ otherwise}. \end{array} \right.$$
Clearly, both functionals are convex and lower semicontinuous with respect to $H$-convergence. Therefore, the subdifferentials $\partial_H \mathcal{F}_{q,m^1}$, $\partial_H \mathcal{F}_{p,m^2}$, and $\partial_H \left(\mathcal{F}_{q,m^1}+\mathcal{F}_{p,m^2}\right)$ are maximal monotone operators on~$H$,  and with a proof similar to Theorem \ref{0942} we obtain the following result concerning the relationship between these three operators.

\begin{theorem}\label{thm:qplaplace}
Assume that
$$\mu:=\frac{d\nu_2}{d\nu_1}\in L^\infty(X,\nu_1),$$
\begin{equation}
\hbox{and there exists $c >0$ such that} \ \mu\ge c\quad \nu_1\hbox{-a.e.}
\end{equation}
Suppose that one of the following conditions holds:
\begin{itemize}
\item[(a)] $\nu_1(X) < \infty$ and $q \leq 2$;

\item[(b)] $\nu_1(X) = +\infty$ and $q \leq \frac{p}{p-1} \leq 2 \leq p$.
\end{itemize}
Then, we have
\begin{equation}
\partial_H \left(\mathcal{F}_{q,m^1}+\mathcal{F}_{p,m^2}\right) = -\Delta^{m^1}_q - \mu\Delta^{m^2}_p.
\end{equation}
Moreover, this operator is completely accretive and has a dense domain in $H$.
\end{theorem}

\begin{proof}
Following the strategy of the proof of Theorem~\ref{0942}, Steps 1 and 2 are nearly identical up to the point where we get an inequality analogous to \eqref{1239r}, i.e.,
\begin{equation}\begin{array}{c}\displaystyle
\int_X \left(v + \mu{\rm div}_{m^2} (\vert \nabla u \vert^{p-2} \nabla u)\right) \, w \, d\nu_1 \le \mathcal{F}_{q,m^1}(w)
\end{array}
\end{equation}
for any $w \in \hbox{Dom}(\mathcal{F}_{q,m^1}) \cap \hbox{Dom}(\mathcal{F}_{p,m^2})$, from which we conclude that
\begin{equation}
\eta := v + \mu{\rm div}_{m^2} (\vert \nabla u \vert^{p-2} \nabla u) \in L^{q'}(X,\nu_1).
\end{equation}
Let us see now that $\eta\in L^2(X,\nu_1)$. We now follow the lines of Step 3 in the proof of Theorem~\ref{0942}. In case (a), since $q \leq 2$, it holds that $q' \geq 2$ and consequently $\eta \in L^{q'}(X,\nu_1)$ implies that $\eta \in L^{2}(X,\nu_1)$. In case (b), by the assumption that $p \geq 2$, we have $|\nabla u|^{p-1} \leq |\nabla u|$ for $|\nabla u| \leq 1$; and by the assumption that $q \leq \frac{p}{p-1}$, we have $|\nabla u|^{p-1} \leq |\nabla u|^{\frac{p}{q}}$ for $|\nabla u| \geq 1$. Therefore,
\begin{align}
|\eta(x)| &\leq |v(x)| + \mu(x) \bigg| \int_{X} |\nabla u(x,y)|^{p-2} \nabla u(x,y) \, dm^2_x(y) \bigg|
\\
&\le |v(x)| + \mu(x)\int_{\{y:|\nabla u(x,y)|\le 1\}} |\nabla u(x,y)|^{p-1} \, dm^2_x(y) \\ &\qquad\qquad\qquad\qquad\qquad\qquad\qquad +\mu(x) \int_{\{y:|\nabla u(x,y)|\ge 1\}}|\nabla u(x,y)|^{p-1} \, dm^2_x(y)
\\
&\le |v(x)| + \mu(x)\int_{\{y:|\nabla u(x,y)|\le 1\}} |\nabla u(x,y)| \, dm^2_x(y) \\ &\qquad\qquad\qquad\qquad\qquad\qquad\qquad +\mu(x) \int_{\{y:|\nabla u(x,y)|\ge 1\}}|\nabla u(x,y)|^{\frac{p}{q}} \, dm^2_x(y)
\\
&\le |v(x)|+ \underbrace{\mu(x) \int_{X}|\nabla u(x,y)| \, dm^2_x(y)}_{\rm I(x)} +  \underbrace{ \mu(x) \int_{X} |\nabla u(x,y)|^{\frac{p}{q}} \, dm^2_x(y)}_{\rm II(x)}.
\end{align}
By assumption, $v \in L^2(X,\nu_1)$. The term ${\rm I}$ also lies in $L^2(X,\nu_1)$ in the same way as in the proof of Theorem \ref{0942}. Concerning the term ${\rm II}$, by the Jensen inequality,
\begin{align}
\int_X {\rm II(x)}^q \, d\nu_1(x) &= \int_X \bigg(\mu(x)\int_{X}|\nabla u(x,y)|^{\frac{p}{q}} \,dm^2_x(y) \bigg)^q d\nu_1(x) \\
&\leq \Vert \mu \Vert^{q-1}_\infty \int_X \mu(x) \int_{X}|\nabla u(x,y)|^p \, dm^2_x(y) \, d\nu_1(x) \\
&= \Vert \mu \Vert^{q-1}_\infty \int_X \int_{X} |\nabla u(x,y)|^p \, dm^2_x(y) \, d\nu_2(x),
\end{align}
which is finite since $u \in \hbox{Dom}(\mathcal{F}_{p,m^2})$. Therefore,
$$\eta=v+\mu{\rm div}_{m^2} (\vert \nabla u \vert^{p-2} \nabla u) \in \left(L^2(X,\nu_1) + L^q(X,\nu_1)\right)\cap L^{q'}(X,\nu_1).$$
It is easy to see that
$$\left(L^2(X,\nu_1) + L^q(X,\nu_1)\right) \cap L^{q'}(X,\nu_1) \subset L^2(X,\nu_1),$$
so we conclude that
\begin{equation}
\eta=v+\mu{\rm div}_{m^2} (\vert \nabla u \vert^{p-2} \nabla u)\in L^2(X,\nu_1).
\end{equation}
We then proceed without further changes as in Steps 4 and 5 of Theorem \ref{0942}.
\end{proof}

Consequently, we get the corresponding existence results and contraction estimates for the $(q,p)$-Laplace evolution equation.  Until the end of this Section, we assume that the assumptions of Theorem \ref{thm:qplaplace} hold.

\begin{theorem}\label{thm:qplaplacegradientflow}
Let $T>0$. For any $u_0\in L^2(X,\nu_1)$ and $f\in L^2(0,T;L^2(X,\nu_1))$, the following problem has a unique strong solution:
\begin{equation}\label{Cauchyqp}
\left\{\begin{array}{ll}
u_t - \Delta^{m^1}_q u- \mu\Delta^{m^2}_p u \ni f&\hbox{on } [0,T]
\\[6pt]
u(0)=u_0.
\end{array}
\right.
\end{equation}
Moreover, a comparison and contraction principle analogous to Theorem \ref{0942cor} holds.
\end{theorem}

\begin{theorem}\label{thm:asymptotics2poincareqp}
Assume that $\nu_1(X) < \infty$ and $\mathcal{F}_{q,m^1}$ satisfies a $(q,2)$-Poincar\'{e} inequality, i.e., there exists a constant $\lambda_2(\mathcal{F}_{q,m^1}) >0$ such that
\begin{equation}\label{Poincareq2}
\lambda_2(\mathcal{F}_{q,m^1}) \, \Vert u - \overline{u} \Vert_{L^2(X, \nu_1)}^q \leq \mathcal{F}_{q,m^1}(u) \quad \forall \, u \in L^2(X, \nu_1),
\end{equation}
where
$$\overline{u}:= \frac{1}{\nu_1(X)} \int_X u \, d\nu_1.$$
For $u_0 \in L^2(X,\nu_1)$, let $u(t)$ be the solution of the Cauchy problem \eqref{Cauchyqp}. Then,
$$\Vert u(t)- \overline{u_0} \Vert_{L^2(X,\nu_1)}^{2-q} \leq \left( \Vert u_0 - \overline{u_0}\Vert_{L^2(X,\nu_1)}^{2-q} - \lambda_2(\mathcal{F}_{q,m^1}) \, t \right)^+ \quad \forall t>0$$
if $q < 2$ and
$$\Vert u(t)- \overline{u_0} \Vert_{L^2(X,\nu_1)}^{2} \leq \Vert u_0 - \overline{u_0}\Vert_{L^2(X,\nu_1)}^{2} \exp(-2 \lambda_2(\mathcal{F}_{q,m^1}) \, t) \quad \forall t>0$$
if $q = 2$. In particular, if we denote by
$$T_{\rm ex}(u_0):= \inf \{ T > 0: \, u(t) = \overline{u_0}  \  \ \forall   t \geq T \}$$
the extinction time, for $q < 2$ it is finite and we have the following bound
$$T_{\rm ex}(u_0) \le  \frac{\Vert u_0 -\overline{u_0} \Vert_{L^2(X,\nu_1)}^{2-q}}{(2-q) \lambda_2(\mathcal{F}_{q,m^1})}.$$
\end{theorem}

\begin{proof}
Let $v(t):= u(t) - \overline{u_0}$. We proceed as in the proof of Theorem \ref{thm:asymptotics2poincare1p} until we arrive at the estimate for $\frac12 \frac{d}{dt} \int_X v(t)(x)^2 \, d\nu_1(x)$, where integrating by parts we get
\begin{align}
\frac12 \frac{d}{dt} &\int_X v(t)(x)^2 \, d\nu_1(x) \\
&= \int_X  \Delta^{m^1}_q v(t)(x) \cdot v(t)(x) \, d \nu_1(x) + \int_X \mu(x) \cdot \Delta^{m^2}_p v(t)(x) \cdot v(t)(x) \, d \nu_1(x) \\
&= \int_X  \Delta^{m^1}_q v(t)(x) \cdot v(t)(x) \, d \nu_1(x) + \int_X  \Delta^{m^2}_p v(t)(x) \cdot v(t)(x) \, d \nu_2(x) \\
&= - \frac12  \int_{X \times X} \vert \nabla v(t)(x,y) \vert^q \, dm^1_x(y) \, d\nu_1(x) - \frac12  \int_{X \times X} \vert \nabla v(t)(x,y) \vert^p \, dm^2_x(y) \, d\nu_2(x) \\
&= - q \mathcal{F}_{q,m^1}(v(t)) - p \mathcal{F}_{p,m^2}(v(t)) \leq - q \mathcal{F}_{q,m^1}(v(t)).
\end{align}
Thus,
$$\frac12 \frac{d}{dt} \int_X v(t)(x)^2 \, d \nu_1(x) + q \mathcal{F}_{q,m^1}(v(t)) \leq 0.$$
It is easy to see that the mass is conserved. Since we assumed that the functional $\mathcal{F}_{q,m^1}$ satisfies a  $(q,2)$-Poincar\'{e} inequality, for some $\lambda_2(\mathcal{F}_{q,m^1}) >0$ we have
$$\frac12 \frac{d}{dt} \Vert v(t) \Vert^2_{L^2(X,\nu_1)} + q \lambda_2(\mathcal{F}_{q,m^1})\, \Vert v(t) \Vert_{L^2(X,\nu_1)}^q \leq 0$$
and integrating this ordinary differential inequality yields the claim.
\end{proof}
%
%

Finally, let us note that in the case when $\nu_1(X) < \infty$ and for $p,q \leq 2$, some of the results presented in this Section can be shown using a much simpler argument.

\begin{remark}\label{lo001}{\rm
If $\nu_1(X) < +\infty$ and $1 \leq q,p \leq 2$, the result    given in Theorem \ref{thm:qplaplace} is trivial.  Indeed, thanks to Remark~\ref{dominio01} we are in the conditions to apply ~\cite[Corollaire~2.7]{Brezis}, so in this case we obtain directly that
$$\partial_H  \mathcal{F}_{q,m^1} + \partial_H  \mathcal{F}_{p,m^2} = \partial_H \left(\mathcal{F}_{q,m^1}+\mathcal{F}_{p,m^2}\right),$$
and it is easy to see that $\partial_H  \mathcal{F}_{p,m^2}=\mu\Delta_p^{m^2}$.
\hfill $\blacksquare$
}\end{remark}

\section{Different growth functionals on a partition of the random walk domain}\label{section4}

Let $[X,\mathcal{B},m,\nu]$ be a random walk space. In this Section, our main object of interest is a nonlocal functional which behaves dissimilarly on different subsets of the set $X \times X$. We consider non-null measurable subsets $A_x,B_x\subset  \hbox{supp}(m_x)$  such that
$$\hbox{supp}(m_x)=A_x\cup B_x,$$
 Note that the sets $A_x$ and $B_x$ may overlap. Let us consider now the energy functional~\eqref{eldelaintro01} stated  in the introduction, i.e.,
\begin{equation}
\mathcal{F}_m(u)=\int_{X}\left(\frac12\int_{A_x}|u(y)-u(x)| \, dm_x(y) + \frac{1}{2p}\int_{B_x}|u(y)-u(x)|^p \, dm_x(y)\right) d\nu(x),
\end{equation}
where $\mathcal{F}_m(u)=+\infty$ if the integral is not finite. By reversibility of $\nu$ with respect to~$m$, we have that
\begin{align}\label{elfunpar}
\mathcal{F}_m(u) &=  \frac12\int_{X} \int_{X}|u(y)-u(x)|\frac{\1_{A_x}(y)+\1_{A_y}(x)}{2} \, dm_x(y) \, d\nu(x) \\
&\qquad\qquad +\frac{1}{2p}\int_{X} \int_{X} |u(y)-u(x)|^p\frac{\1_{B_x}(y)+\1_{B_y}(x)}{2} \, dm_x(y) \, d\nu(x).
\end{align}
Consider the symmetric functions $K_A, K_B: X \rightarrow \R$ defined by
$$K_A(x,y):= \frac{\1_{A_x}(y)+\1_{A_y}(x)}{2} \quad \mbox{and} \quad K_B(x,y):= \frac{\1_{B_x}(y)+\1_{B_y}(x)}{2}.$$
Then, we define the energy functionals $ \mathcal{F}_{A,1,m}, \mathcal{F}_{B,p,m}: L^2(X, \nu) \rightarrow (-\infty, + \infty]$ as
$$ \mathcal{F}_{A,1,m}(u):= \left\{\begin{array}{ll}\displaystyle\frac12\int_{X\times X}|u(y)-u(x)| \, K_A(x,y) \, d(\nu\otimes m_x)(x,y)\quad &\hbox{if the integral is finite}; \\[10pt] +\infty \quad &\hbox{otherwise}, \end{array} \right.$$
and
$$\mathcal{F}_{B,p,m}(u):= \left\{\begin{array}{ll}\displaystyle\frac{1}{2p}\int_{X \times X}  |u(y)-u(x)|^p \, K_B(x,y) \ d(\nu\otimes m_x)(x,y) \quad &\hbox{if the integral is finite}; \\[10pt] +\infty \quad &\hbox{otherwise}. \end{array} \right.$$
Observe that
$$L^p(X, \nu) \cap L^2(X, \nu)\subset D(\mathcal{F}_{B,p,m}).$$
Moreover, since $\mathcal{F}_{B,p,m}$ is convex and lower semicontinuous with respect to convergence in $L^2(X,\nu)$, the subdifferential $\partial_{L^2(X,\nu)} \mathcal{F}_{B,p,m}$ is a maximal monotone operator with a dense domain.

\begin{definition}{\rm
We define the {\it $m$-$p$-$B$-Laplacian operator} $\Delta_{p,B}^m$  in $[X,\mathcal{B},m,\nu]$ as $$(u,v) \in \Delta_{p,B}^m \iff u,v\in L^2(X,\nu),\   |\nabla u|^{p-1} \in L^{1}(X \times X, \nu\otimes m_x) \ \hbox{and}$$
$$v(x)= {\rm div}_m (K_B \vert \nabla u \vert^{p-2} \nabla u)(x) = \int_X K_B(x,y)\vert \nabla u(x,y) \vert^{p-2} \, \nabla u(x,y) \, dm_x(y).$$
}
\end{definition}

Working as in the proof of Theorem \ref{plaplac01}, we have the following result.

\begin{theorem}\label{plaplac01B} We have
\begin{equation}
(u,v) \in \Delta_{p,B}^m \iff (u,-v) \in \partial_{L^2(X,\nu)} \mathcal{F}_{B,p,m}.
\end{equation}
\end{theorem}

Similarly, since $\mathcal{F}_{A,1,m}$ is convex and lower semicontinuous with respect to convergence in $L^2(X,\nu)$, the subdifferential $\partial_{L^2(X,\nu)} \mathcal{F}_{A,1,m}$ is a maximal monotone operator with a dense domain. Additionally, working as in \cite{MST1} (see also \cite{MSTBook}) one easily  obtains the following characterisation.

\begin{theorem}\label{1051}
$$(u,v) \in \partial_{L^2(X,\nu)} \mathcal{F}_{A,1,m} \iff u,v\in L^2(X,\nu) \hbox{ and }$$
$$\hbox{there exists} \ \g \in L^\infty(X \times X, \nu \otimes m_x) \ \hbox{antisymmetric such that $\Vert \g \Vert_{L^\infty(X \times X, \nu \otimes m_x)} \leq 1$;}$$
$$v(x) = - \int_X \g(x,y) K_A(x,y) \, dm_x(y) \quad \hbox{for} \ \ \nu\hbox{-a.e.} \ x \in X;$$
and
$$\g(x,y) K_A(x,y) \in {\rm sign}(u(y) - u(x))K_A(x,y) \quad \hbox{for} \ \ (\nu \otimes m_x) \hbox{-a.e.} \ (x,y) \in X \times X.$$
\end{theorem}

\begin{definition}{\rm We define the {\it  $m$-$1$-$A$-Laplacian operator} $\Delta_{1,A}^m$  in $[X,\mathcal{B},m,\nu]$ as
\begin{equation}\label{1SubdA}
 (u,v) \in \Delta_{1,A}^m \iff (u,-v) \in \partial_{L^2(X,\nu)} \mathcal{F}_{A,1,m}.
 \end{equation}
}
\end{definition}

We have the following characterisation of the subdifferential of $\mathcal{F}_m$.

\begin{theorem}\label{thm:abcharacterisation}
Suppose that one of the following conditions holds:
\begin{itemize}
\item[(a)] $\nu(X) < \infty$;

\item[(b)] $\nu(X) = +\infty$ and $p \geq 2$.
\end{itemize}
Then, we have
$$\partial_{L^2(X,\nu)} \left(\mathcal{F}_{A,1,m} +  \mathcal{F}_{B,p,m} \right) = \partial_{L^2(X,\nu)} \mathcal{F}_{A,1,m} + \partial_{L^2(X,\nu)} \mathcal{F}_{B,p,m} = - \Delta_{1,A}^m - \Delta_{p,B}^m.$$
Furthermore, this operator is completely accretive and has a dense domain in $L^2(X, \nu)$.
\end{theorem}

\begin{proof}
{\bf Step 1.} By the integration by parts formula, it is easy to see that the operator $- \Delta_{1,A}^m - \Delta_{p,B}^m$ is completely accretive, and the density of the domain is proved as in   Step~5  of the proof of Theorem \ref{0942}. Therefore, since as a consequence of complete accretivity we have that
$$\partial_{L^2(X,\nu)} \mathcal{F}_{A,1,m} + \partial_{L^2(X,\nu)} \mathcal{F}_{B,p,m} \subset \partial_{L^2(X,\nu)} \left(\mathcal{F}_{A,1,m} +  \mathcal{F}_{B,p,m} \right),$$
we only need to prove the reverse inclusion.

{\flushleft \bf Step 2.} Fix $v\in \partial_{L^2(X,\nu)} \left(\mathcal{F}_{A,1,m} +  \mathcal{F}_{B,p,m} \right)(u)$. Then, for any
$$w\in  \hbox{Dom}(\mathcal{F}_{A,1,m})\cap\hbox{Dom}(\mathcal{F}_{B,p,m}),$$
working as in the proof of Step 2 of Theorem \ref{0942}, we get
\begin{align}
\displaystyle\frac12 &\iint_{\{(x,y): u(y)-u(x)\neq0\}} \hbox{sign}^0(u(y)-u(x)) \, \nabla w(x,y) \, K_A(x,y) \, dm_x(y) \, d\nu(x) \\
&\qquad -\frac12\iint_{\{(x,y): u(y)-u(x)=0\}}|\nabla w(x,y)| \, K_A(x,y) \, dm_x(y) \, d\nu(x) \\
&\qquad + \frac{1}{2}\int_X\int_X \vert \nabla u(x,y)\vert^{p-2} \, \nabla u(x,y) \, \nabla w(x,y) \, K_B(x,y)\, dm_x(y) \, d\nu(x) = \int_X vw \, d\nu.
\end{align}
Now, integrating by parts, we obtain that
\begin{align}\label{19412NewI}
\displaystyle\frac12 &\iint_{\{(x,y): u(y)-u(x)\neq0\}} \hbox{sign}^0(u(y)-u(x)) \, \nabla w(x,y) \, K_A(x,y) \, dm_x(y) \, d\nu(x) \\
&\qquad -\frac12\iint_{\{(x,y): u(y)-u(x)=0\}}|\nabla w(x,y)| \, K_A(x,y) \, dm_x(y) \, d\nu(x) \\
&\qquad\qquad = \int_X \left(v + {\rm div}_m( K_B(x,y)\vert \nabla u(x,y)\vert^{p-2} \, \nabla u(x,y))\right) \,w \, d\nu.
\end{align}
Taking $w = u$ in \eqref{19412NewI}, we have
\begin{equation}\label{19412NewII}
\int_X \left(v + {\rm div}_m( K_B(x,y)\vert \nabla u(x,y)\vert^{p-2} \, \nabla u(x,y))\right) \,u \, d\nu = \mathcal{F}_{A,1,m}(u).
\end{equation}
On the other hand, by \eqref{19412NewI}, we also get
\begin{equation}\label{19412NewIII}
\int_X \left(v + {\rm div}_m( K_B(x,y)\vert \nabla u(x,y)\vert^{p-2} \, \nabla u(x,y))\right) \,w \, d\nu \leq \mathcal{F}_{A,1,m}(w).
\end{equation}
Then, we may conclude as in the proof of Theorem \ref{0942}. Indeed, from \eqref{19412NewII}, \eqref{19412NewIII} and a density argument, for any $ w\in  \hbox{Dom}(\mathcal{F}_{A,1,m})$ we have
$$ \mathcal{F}_{A,1,m}(w) -  \mathcal{F}_{A,1,m}(u) \geq \int_X \left(v + {\rm div}_m( K_B(x,y)\vert \nabla u(x,y)\vert^{p-2} \, \nabla u(x,y))\right) \,(w -u) \, d\nu.$$
Under the assumptions (a) or (b), working similarly as in Steps 3 and 4 of Theorem \ref{0942}, we have
$$ v + {\rm div}_m( K_B(x,y)\vert \nabla u(x,y)\vert^{p-2} \, \nabla u(x,y)) \in L^2(X,\nu) $$
and subsequently
$$v + {\rm div}_m( K_B(x,y)\vert \nabla u(x,y)\vert^{p-2} \, \nabla u(x,y)) \in \partial_{L^2(X,\nu)} \mathcal{F}_{A,1,m}(u),$$
so
$$v \in - \Delta_{1,A}^m - \Delta_{p,B}^m,$$
which concludes the proof.
\end{proof}

By the Brezis-K\={o}mura Theorem and the complete accretivity of the operator, as a consequence of the above theorem we have the following existence and uniqueness result.

\begin{theorem}\label{0942corNew}  Let $T>0$. For any $u_0\in L^2(X,\nu)$ and $f\in L^2(0,T;L^2(X,\nu))$, the following problem
\begin{equation}\label{ProblemAB}\left\{\begin{array}{ll}
u_t- \Delta_{1,A}^m(u) - \Delta_{p,B}^m(u) \ni f&\hbox{on } [0,T];
\\[6pt]
u(0)=u_0
\end{array}
\right.
\end{equation}
has a unique strong solution $u=u_{u_0,f}$. Moreover, we have the following comparison and contraction principle: if $u_0,\widetilde u_0\in L^2(X,\nu)$ and $f,\widetilde f\in L^2(0,T;L^2(X,\nu))$, denoting $u=u_{u_0,f}$ and $\tilde u=u_{\widetilde u_0,\widetilde f}$, we have that for any $0\le t\le T$ and any $1 \leq  q \leq \infty$
\begin{equation}
\Big\Vert \left(u(t)-\widetilde u(t)\right)^+\Big\Vert_{L^q(X,\nu)}\le \Big\Vert \left(u_0-\widetilde u_0\right)^+\Big\Vert_{L^q(X,\nu)}
+\int_0^t \Big\Vert \left(f(s)-\widetilde f(s)\right)^+\Big\Vert_{L^q(X,\nu)} ds.
\end{equation}
\end{theorem}

Regarding the asymptotic behaviour of solutions to problem \eqref{ProblemAB}, we observe that if we assume a Poincar\'e inequality of the type
\begin{equation}
c \| u \|_{L^2(X,\nu)} \leq \mathcal{F}_{A,1,m}(u) = \frac12\int_{X\times X} |u(y) - u(x)| \, K_A(x,y) \, d(\nu\otimes m_x)(x,y)
\end{equation}
for some $c > 0$, with a proof very similar to the one of Theorem \ref{thm:asymptotics2poincare1p} we recover the finite extinction time. This assumption is natural in the sense that it corresponds to coercivity of the operator $\mathcal{F}_{A,1,m}$ in the sense of \cite{BB}. However, it depends very strongly on the choice of $A$ and may be difficult to verify.   If such a condition is violated, we will see in Section \ref{sec:examples} that depending on the choice of initial data, the extinction time of solutions may be finite or infinite.  However, let us note that it is satisfied in the case when $A$ has full support, i.e., $A_x = \mbox{supp}(m_x)$, and $\mathcal{F}_{1,m^1}$ satisfies a suitable Poincar\'e inequality. 

Let us also note that different choices of the sets $A$ and $B$ in problem \eqref{ProblemAB} may lead to essentially the same evolution equation, as shown in the next example.

\begin{example}\label{z2rem}\rm

 Recall the setting of Example \ref{z2ejm} on the lattice $\mathbb{Z}^2$ and consider the following partitions of the random walk $m$.

{\flushleft (i)}   Setting
$$A_{(n,m)}=\{(n-1,m),(n+1,m), (n,m-1)\},\ B_{(n,m)}=\{(n,m+1)\},$$
the functional to study becomes
$$\begin{array}{l}
\displaystyle\mathcal{F}_m(u)=  \sum_{(m,n)\in \mathbb{Z}^2}\bigg(\frac12\Big(a|u(n-1,m)-u(n,m)| +a |u(n+1,m)-u(n,m)|
\\[16pt]
  \displaystyle \qquad\qquad\qquad\qquad\qquad
+b|u(n,m-1)-u(n,m)|\Big)
  + \frac{1}{2p} b|u(n,m+1)-u(n,m)|^p\bigg),
\end{array}$$
but we can rewrite it as
$$\begin{array}{l}
\displaystyle\mathcal{F}_m(u)=  \sum_{(m,n)\in \mathbb{Z}^2}\bigg(\frac12\Big(a|u(n-1,m)-u(n,m)| +a|u(n+1,m)-u(n,m)|
\\[16pt]
  \displaystyle \qquad\qquad\qquad\qquad\qquad
+ \frac{b}{2}|u(n,m-1)-u(n,m)|+\frac{b}{2}|u(n,m+1)-u(n,m)|
\Big)
\\[16pt]
  \displaystyle \qquad\qquad\qquad\qquad
+\frac{1}{2p}\Big( \frac{b}{2}|u(n,m-1)-u(n,m)|^p + \frac{b}{2}|u(n,m+1)-u(n,m)|^p\Big)\bigg),
\end{array}$$
for which   the gradient flow associated to it looks like a $1$-Laplacian over the whole space plus a $p$-Laplacian in the vertical direction (with slightly different coefficients).

{\flushleft (ii)}   For the choice
$$A_{(n,m)}=\{(n-1,m), (n+1,m),(n,m-1)\},\ B_{(n,m)}=\{ (n,m-1),(n,m+1)\},$$
the functional to study is
$$\begin{array}{l}
\displaystyle\mathcal{F}_m(u)=  \sum_{(m,n)\in \mathbb{Z}^2}\bigg(\frac12\Big(a|u(n-1,m)-u(n,m)| + a|u(n+1,m)-u(n,m)|
\\[16pt]
   \displaystyle \qquad\qquad\qquad\qquad\qquad + b|u(n,m-1)-u(n,m)|(n,m-1)\Big)
\\[16pt]
   \displaystyle \qquad\qquad\qquad\qquad + \frac{1}{2p}\Big( b|u(n,m-1)-u(n,m)|^p  + b|u(n,m+1)-u(n,m)|^p\Big)\bigg).
\end{array}$$
We can rewrite it as follows:
$$\begin{array}{l}
\displaystyle\mathcal{F}_m(u)=  \sum_{(m,n)\in \mathbb{Z}^2}\bigg(\frac12\Big(a|u(n-1,m)-u(n,m)| + a|u(n+1,m)-u(n,m)|
\\[16pt]
   \displaystyle \qquad\qquad\qquad\qquad\qquad +\frac{b}{2}|u(n,m-1)-u(n,m)| + \frac{b}{2}|u(n,m+1)-u(n,m)|\Big)
\\[16pt]
   \displaystyle \qquad\qquad\qquad\qquad +\frac{1}{2p}\Big( b|u(n,m-1)-u(n,m)|^p  + b|u(n,m+1)-u(n,m)|^p\Big)\bigg),
\end{array}$$
for which again   the associated gradient flow looks like a $1$-Laplacian over the whole space plus a $p$-Laplacian in the vertical direction (with slightly different coefficients).   \hfill$\blacksquare$
\end{example}

In a similar manner, we obtain the counterpart of this result in the case when the growth on the two subsets $A,B$ is of the $(q,p)$-Laplace type with $q > 1$.

\begin{theorem}
Suppose that one of the following conditions holds:
\begin{itemize}
\item[(a)] $\nu(X) < \infty$, $q \leq 2$;

\item[(b)] $\nu(X) = +\infty$ and $q \leq \frac{p}{p-1} \leq 2 \leq p$.
\end{itemize}
Then, we have
$$\partial_{L^2(X,\nu)} \left(\mathcal{F}_{A,q,m} +  \mathcal{F}_{B,p,m} \right) = \partial_{L^2(X,\nu)} \mathcal{F}_{A,q,m} + \partial_{L^2(X,\nu)} \mathcal{F}_{B,p,m} = - \Delta_{q,A}^m - \Delta_{p,B}^m.$$
Furthermore, this operator is completely accretive and has a dense domain in $L^2(X, \nu)$.
\end{theorem}

We immediately obtain the corresponding existence and uniqueness result.

\section{Examples}\label{sec:examples}

\subsection{Weighted Graphs}

We now give several examples of explicit solutions to particular cases of problems \eqref{eq:Cauchy1p} and \eqref{ProblemAB} in  finite weighted graphs in order to illustrate the results obtained. We will study cases of the parabolic equations
\begin{equation}\label{eq:Cauchy1penex}
u_t = \Delta_1^{m} (u) + \Delta_2^{m} (u),
\end{equation}
\begin{equation}\label{eq:Cauchy1penex2rw}
u_t = \Delta_1^{m_1} (u) + \mu\Delta_2^{m_2} (u),
\end{equation}
where $\mu=\frac{d\nu_2}{d\nu_1}$,
and
\begin{equation}\label{ProblemABenex} u_t = \Delta_{1,A}^m(u) + \Delta_{2,B}^m(u),
\end{equation}
which will become systems of linear ordinary differential equations after dealing with the  particular difficulty due to the presence of the $1$-Laplacian term; the coefficients of the system may change with time depending on the sign of the nonlocal gradient of $u$ between two given points. In the first example, we consider the problem associated to~\eqref{eq:Cauchy1penex} on a linear graph with three points.

\begin{example}\label{ex3points}{\rm
 Consider a linear graph $G=(V,E)$ with three vertices $V = \{1,2,3\}$, two edges $E = \{(1,2), (2,3)\}$, and with positive weights
$$w_{1,2} =a,\quad  w_{2,3}=b. $$ The graph is shown in Figure~\ref{fig02_3points}.
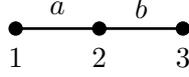
\begin{figure}[ht]
\centering
\begin{tikzpicture}[line cap=round,line join=round,>=triangle 45,x=1.1cm,y=1.1cm]
\draw [line width=1pt] (-2,0)-- (0,0);
\begin{scriptsize}
\draw [fill=black] (-2,0) circle (2.5pt);
\draw[color=black] (-2,-0.35) node {\large $1$};
\draw [fill=black] (-1,0) circle (2.5pt);
\draw[color=black] (-1,-0.35) node {\large $2$};
\draw[color=black] (-1.5,0.25) node {\large $a$};
\draw [fill=black] (0,0) circle (2.5pt);
\draw[color=black] (0,-0.35) node {\large $3$};
\draw[color=black] (-0.5,0.25) node {\large $b$};
\end{scriptsize}
\end{tikzpicture}
\caption{Graph in Example~\ref{ex3points}}\label{fig02_3points}
\end{figure}

We have $$ \nu(\{ 1 \}) = a, \quad \nu(\{ 2 \}) = a+b, \quad \nu(\{ 3 \}) = b,$$
and the random walk $m$ is given by
$$m_1 =  \delta_2 , \quad m_2 = \frac{a}{a+b}  \delta_1 +  \frac{b}{a+b}  \delta_3,\quad m_3 =\delta_2.$$
Then, this weighted graph is a   random  walk space as in Example~\ref{example.graphs}. Consider the evolution problem
$$u_t=\Delta_1^mu+\Delta_2^mu\quad\hbox{in } V$$
with the initial condition
$$u(0)=c\1_{\{1 \}}.$$
Let us call $x(t):=u(1,t)$, $y(t)=u(2,t)$ and $z(t)=u(3,t)$. Then, the above equation can be written as the following system of ODEs
\begin{equation}\label{masterE}\left\{\begin{array}{l}
x'(t)= \g_t(1,2) + y(t) - x(t); \\ \\
y'(t)= - \frac{a}{a+b} \g_t(1,2) + \frac{b}{a+b} \g_t(2,3)+\frac{a}{a+b}(x(t)-y(t)) + \frac{b}{a+b}(z(t) - y(t));\\ \\
z'(t)= -\g_t(2,3) + y(t) - z(t)
\end{array}\right.
\end{equation}
for antisymmetric functions $ \g_t(1,2), \g_t(2,3)$ satisfying
$$ \g_t(1,2) \in \hbox{sign}(y(t) - x(t)), \quad \g_t(2,3) \in\hbox{sign}(z(t) - y(t))$$
with the initial condition
$$x(0) = c, \quad y(0) = 0, \quad z(0) = 0.$$
We present the behaviour of this system in the following three special cases corresponding to different choices of weights and initial data.

{\flushleft Case A.}  We fix the weights $a=b=1$, and for the initial datum  we take $c=1$. We claim that, up to a time $t_1$, $x(t)>y(t)$ and  $y(t) = z(t)$; at $t=t_1$ we have that $x(t_1) = y(t_1)$, moment in which the value of $\g_t(1,2)$ may differ from $-1$. Then, problem \eqref{masterE} is reduced to the system of ODEs
$$\left\{\begin{array}{l}
x'(t)=-1+y(t)-x(t); \\ \\
y'(t)=\frac12+\frac12 \g_t(2,3) + \frac{1}{2}(x(t)-y(t)); \\ \\
z'(t)= - \g_t(2,3) \in\hbox{sign}(0)
\end{array}\right.
$$
with $y(t)=z(t)$, joint to the initial condition
$$x(0) = 1, \quad y(0) = 0, \quad z(0) = 0.$$
Now, since $y'(t) = z'(t) = - \g_t(2,3)$, the above system is reduced to
$$\left\{\begin{array}{l}
x'(t)=-1+y(t)-x(t);\\ \\
\frac32 y'(t)=\frac12+\frac{1}{2}(x(t)-y(t)) \\ \\
\end{array}\right.
$$
joint to the initial condition
$$x(0) = 1, \quad y(0) = 0, \quad z(0) = 0,$$
whose only solution is given by
$$  x(t)=\frac32 e^{-\frac43 t} -\frac12  ,\
y(t)=z(t)=-\frac12 e^{-\frac43 t} +\frac12.$$
Moreover,
$$ \g_t(2,3) = -z'(t) = -\frac23 e^{-\frac43 t} \in \hbox{sign}(0) = [-1,1].$$
 These formulas are valid until time $t_1 = \frac34 \, \log 2 \simeq 0.51986$, at which point
$$x(t_1) = y(t_1) = z(t_1) = \frac14.$$
Therefore, by uniqueness of the solution, the above is the solution of \eqref{masterE} for $0 \leq t \leq t_1$  (in particular, the assumption that $y(t) = z(t)$ was not restrictive). After that time, the solution is given by
$$x(t) = y(t) = z(t) = \frac14,$$
see Figure~\ref{ex3pointsplot}.   Note that $\frac14$ is the mean of the initial data with respect to the invariant measure $\nu$; by Lemma \ref{ConsMass}, this property holds also in the subsequent examples.

 \begin{figure}[ht]
\centering
\begin{tikzpicture}
    \begin{axis}[
        xlabel={}, ylabel={},
        domain=0:0.75, 
        samples=1000, 
        axis lines=middle, 
        axis line style={-},
        legend pos=north west, 
        width=8cm, 
        height=6cm, 
        ymin=0, ymax=1.1, 
        xtick={0,0.1, 0.2,0.3,0.4,0.5, 0.6, 0.7,0.8}, 
        ytick={0,0.2, 0.4,0.6, 0.8, 1,1.2}, 
        xticklabel style={font=\tiny},
        yticklabel style={font=\tiny},
        minor tick num=4,
    ]
    \addplot[black,thick, restrict x to domain=0:0.52] {(3/2)*exp(-4*x/3)-(1/2)};

    \addplot[black, thick,dashed,restrict x to domain=0:0.52] {-(1/2)*exp(-4*x/3)+1/2};
    \addplot[black,  thick] coordinates {(0.52, 1/4) (0.75, 1/4)};
    \end{axis}
\end{tikzpicture}
\caption{Example~\ref{ex3points}, case A. $a=b=1$, $c=1$. $x(t)$ continuous line; $y(t)=z(t)$ dashed line. After $t\approx 0.51986$,   $x(t)=y(t)=z(t)$.}
\label{ex3pointsplot}
\end{figure}
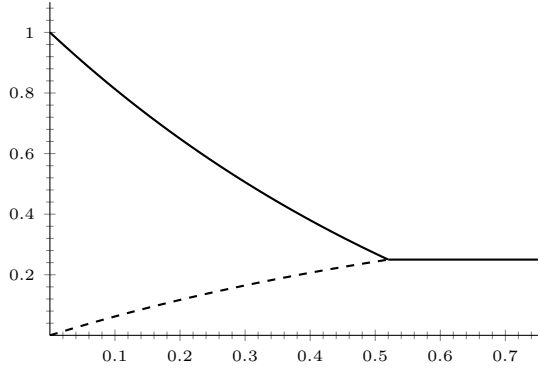

{\flushleft Case B.} The picture changes for the same weights  $a = b = 1$ if we take a larger initial datum with $c = 10$. In this case, we claim that until a time $t_1 > 0$ at which $y(t_1) = z(t_1)$, we have $x(t) > y(t) > z(t)$. This reduces problem \eqref{masterE}  to the system of ODEs
 $$\left\{\begin{array}{l}
  x'(t)=-1+y(t)-x(t);\\ \\
 y'(t)=\frac{1}{2}(x(t)-y(t)) + \frac{1}{2}(z(t)-y(t));\\ \\
 z'(t)= 1 + y(t) - z(t)
 \end{array}\right.
 $$
joint to the initial condition
$$x(0) = 10, \quad y(0) = 0, \quad z(0) = 0,$$
whose only solution is given by
$$  x(t)=\frac32 + \frac52 e^{-2t} + 6 e^{-t} ,\ \
y(t) = \frac52 -\frac52 e^{-2t}, \ \ z(t) = \frac72+ \frac52 e^{-2t} - 6 e^{-t}.$$
For $t_1 = \log 5 \approx 1.609438$,
$$y(t_1)  =  z(t_1) = \frac{12}{5}.$$
Again by uniqueness, the above is the solution of \eqref{masterE} for $0 \leq t \leq t_1$  (see Figure~\ref{ex3points04}).
 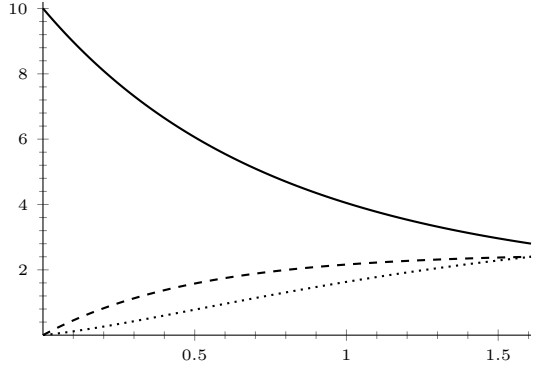
\begin{figure}[ht]
\centering
\begin{tikzpicture}
    \begin{axis}[
        xlabel={}, ylabel={},
        domain=0.:1.65, 
        samples=1000, 
        axis lines=middle, 
        axis line style={-},
        legend pos=north west, 
        width=8cm, 
        height=6cm, 
        ymin=0, ymax=10.2, 
        xtick={0,0.5, 1,1.5,2}, 
        ytick={0,2,4,6,8,10}, 
        xticklabel style={font=\tiny},
        yticklabel style={font=\tiny},
        minor tick num=4,
    ]
    \addplot[black,thick, restrict x to domain=0:1.61] {3/2 + (5/2)*exp(-2*x) + 6*exp(-x)};

    \addplot[black, thick,dashed,restrict x to domain=0:1.61] {5/2 -(5/2)*exp(-2*x)};
    \addplot[black, thick,dotted,restrict x to domain=0:1.61] {7/2+ (5/2)*exp(-2*x) - 6*exp(-x)};
    \end{axis}
\end{tikzpicture}
\caption{Example~\ref{ex3points} case B, $a=1$, $b=1$, $c=10$. $x(t)$ continuous line; $y(t)$ dashed line; $z(t)$ dotted line; $0\le t\lesssim 1.609438$.}
\label{ex3points04}
\end{figure}

Let us see how the solution continues. We have that
$$x(t_1) = \frac{14}{5}.$$
We now rewrite the problem \eqref{masterE} after the time $t_1$. We claim that, up to a time $t_2>t_1$, $x(t) > y(t)$ and  $y(t) = z(t)$. The time $t_2$ will be the moment at which $x=y$.
By comparing the expressions for $y'(t)$ and $z'(t)$,   we see that  $\g_t(2,3) = -y'(t).$
Therefore, the system~\eqref{masterE} takes the form
\begin{equation}
\left\{\begin{array}{l}
x'(t)= -1 + y(t) - x(t); \\ \\
y'(t)= \frac{1}{3}+ \frac{1}{3}(x(t) - y(t))
\end{array}\right.
\end{equation}
with the initial condition
$$x(t_1) = \frac{14}{5}, \quad y(t_1) = \frac{12}{5}.$$
The unique solution is given by
\begin{equation}
x(t) = \frac{7}{4} + \frac{21 \sqrt[3]{5}}{4} e^{-\frac{4}{3} t}, \quad y(t) = \frac{11}{4} - \frac{7 \sqrt[3]{5}}{4} e^{-\frac{4}{3} t}
\end{equation}
  until the time  $t_2 = \frac34\log(7\sqrt[3]{5}) \approx 1.861792$,   at which
$$x(t_2) = y(t_2)=z(t_2)=\frac52.$$
Moreover, it is easy to see that
\begin{equation}
\g_t(2,3) = -y'(t) \in \hbox{sign}(0) = [-1,1].
\end{equation}
Again by uniqueness, the assumption that $y(t) = z(t)$ was not restrictive, and the above is the solution of \eqref{masterE} for $t_1 \leq t \leq t_2$  (see Figure~\ref{ex3points05}). After this time, for $t \geq t_2$ the solution is constant and equal to
$$x(t)=y(t)=z(t)= \frac{5}{2}.$$
\begin{figure}[ht]
\centering
\begin{tikzpicture}
    \begin{axis}[
        xlabel={}, ylabel={},
        domain=1.61:2, 
        samples=1000, 
        axis lines=middle, 
        axis line style={-},
        legend pos=north west, 
        width=6cm, 
        height=4cm, 
        ymin=2, ymax=2.9, 
        xtick={1.861792}, 
        ytick={2.0,2.2,2.4,2.6,2.8,3}, 
        xticklabel style={font=\tiny},
        yticklabel style={font=\tiny},
        minor tick num=4,
        xticklabel={\pgfmathprintnumber[fixed, precision=6]{\tick}},
    ]
    \addplot[black,thick, restrict x to domain=1.61:1.86] {7/4 + ((21*(5^(1/3)))/4)*exp(-(4/3)*x)};

    \addplot[black, thick,dashed,restrict x to domain=1.61:1.86] {11/4 - ((7*(5^(1/3)))/4)*exp(-(4/3)*x)};
    \addplot[black,  thick] coordinates {(1.86, 2.5) (2, 2.5)};
    \end{axis}
\end{tikzpicture}
\caption{Example~\ref{ex3points}, case B. $a=1$, $b=1$, $c=10$. $x(t)$ continuous line; $y(t)=z(t)$ dashed line; $t\gtrsim 1.609438$. After $t\approx 1.861792$,  $x(t)=y(t)=z(t).$}
\label{ex3points05}
\end{figure}
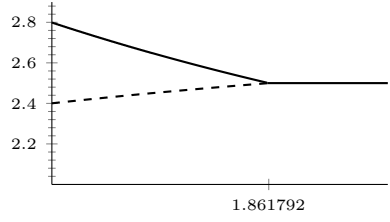

{\flushleft Case C.} If we change the weights to $a=10$, $b=1$, and we consider again $c=1$, then  problem \eqref{masterE} is reduced to the system of ODEs
\begin{equation}\label{CaseC}\left\{\begin{array}{l}
  x'(t)= \g_t(1,2) + (y(t)-x(t));  \\ \\
 y'(t)= -\frac{10}{11} \g_t(1,2) +\frac{1}{11} \g_t(2,3)+\frac{10}{11}(x(t)-y(t)) + \frac{1}{11} (z(t) - y(t));    \\ \\
 z'(t)= -\g_t(2,3) + y(t) - z(t),
 \end{array}\right.
\end{equation}
where
$$ \g_t(1,2) \in \hbox{sign}(y(t) - x(t)) \quad \mbox{and} \quad \g_t(2,3) \in\hbox{sign}(z(t) - y(t)),$$
with the initial condition
$$x(0) = 1, \quad y(0) = 0, \quad z(0) = 0.$$
In this case, we claim that, up to a time $t_1$ at which $x(t_1) = y(t_1)$, we have $x(t) > y(t) > z(t)$. This reduces problem \eqref{CaseC} to the system of ODEs
$$\left\{\begin{array}{l}
x'(t)=-1+y(t)-x(t); \\ \\
y'(t)= \frac{10}{11} -\frac{1}{11}   +\frac{10}{11}(x(t)-y(t)) +  \frac{1}{11} (z(t)-y(t));\\ \\
z'(t)= 1 + y(t) - z(t),
\end{array}\right.
$$
joint to the initial condition
$$x(0) = 1, \quad y(0) = 0, \quad z(0) = 0,$$
whose only solution is given by
$$  x(t)= -\frac{3}{22} + \frac{19}{22}e^{-2t}+\frac{6}{22} e^{-t} ,\ \
y(t) = \frac{19}{22}-\frac{19}{22}e^{-2t}, \ \ z(t) =  \frac{41}{22} + \frac{19}{22}  e^{-2t}- \frac{60}{22}  e^{-t}.$$
For $t_1=\log((3 + 13 \sqrt{5})/22) \approx 0.376844$ we have that
$$x(t_1) =    y(t_1)=\frac{295 + 39 \sqrt{5}}{836} \approx 0.457185.$$
Again, by uniqueness,   the assumption that $x(t) > y(t) > z(t)$ was not restrictive and the above  is the only solution of problem \eqref{CaseC} for $0 \leq t \leq t_1$ (see Figure~\ref{ex3points02}).
\begin{figure}[ht]
\centering
\begin{tikzpicture}
    \begin{axis}[
        xlabel={}, ylabel={},
        domain=0.:0.4, 
        samples=1000, 
        axis lines=middle, 
        axis line style={-},
        legend pos=north west, 
        width=8cm, 
        height=6cm, 
        ymin=0, ymax=1.05, 
        xtick={0,0.05,0.10,0.15,0.20,0.25,0.30,0.35,0.40}, 
        ytick={0,0.2,0.4,0.6,0.8,1.0,1.2}, 
        xticklabel style={font=\tiny},
        yticklabel style={font=\tiny},
        minor tick num=4,
        scaled x ticks=false, 
         xticklabel={\pgfmathprintnumber[fixed]{\tick}},
    ]
    \addplot[black,thick, restrict x to domain=0:0.377] {-3/22  + (19/22)*exp(-2*x)+(6/22)*exp(-x)};

    \addplot[black, thick,dashed,restrict x to domain=0:0.377] {(19/22)-(19/22)*exp(-2*x)};
    \addplot[black, thick,dotted,restrict x to domain=0:0.377] {41/22 + (19/22)*exp(-2*x)- (60/22)*exp(-x)};
    \end{axis}
\end{tikzpicture}
\caption{Example~\ref{ex3points}, case C. $a=10$, $b=1$, $c=1$. $x(t)$ continuous line; $y(t)$ dashed line; $z(t)$ dotted line; $0\le t\lesssim 0.376844$. }
\label{ex3points02}
\end{figure}
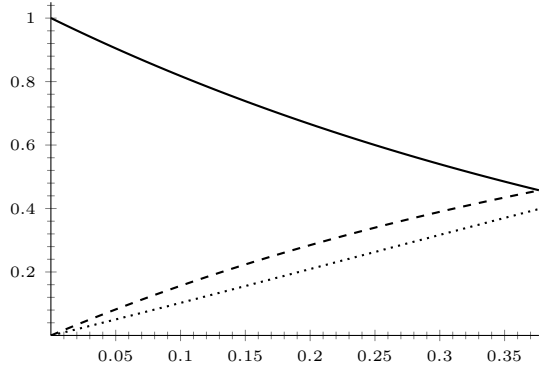

Let us see how the solution continues. We have that
$$ z(t_1)=\frac{2165 - 819\sqrt{5}}{836}  \approx 0.399115.$$
We now claim that, up to a time $t_2>t_1$,   it holds that  $x(t) = y(t)$ and  $y(t) > z(t)$;  the time $t_2$ will be the moment at which $x(t) = y(t)$. Therefore, for $t \geq t_1$, we consider the system of ODEs
\begin{equation}\label{CaseCNew}\left\{\begin{array}{l}
x'(t)= \g_t(1,2), \quad \g_t(1,2) \in \hbox{sign}(0); \\ \\
y'(t)= -\frac{10}{11} \g_t(1,2) -\frac{1}{11} + \frac{1}{11} (z(t) - y(t));\\ \\
z'(t)=1+ y(t) - z(t)
\end{array}\right.
\end{equation}
with the initial condition
$$x(t_1) =\frac{295 + 39 \sqrt{5}}{836}, \quad y(t_1) =\frac{295 + 39 \sqrt{5}}{836}, \quad z(t_1) =\frac{2165 - 819\sqrt{5}}{836}.$$
By using that $\g_t(1,2)$ should be equal to $y'(t)$, the system \eqref{CaseCNew} comes down to
$$  \left\{\begin{array}{l} y'(t) = - \frac{1}{21}+\frac{1}{21}(z(t) -y(t)); \\[10pt] z'(t) = 1 + (y(t) -  z(t)),\end{array}\right. $$
whose only solution is given by
\begin{equation}
y(t) = \frac{9}{22} + \frac{1}{484} (63 - 13 \sqrt{5}) \sqrt[21]{\frac{3 + 13 \sqrt{5}}{22}}  e^{- \frac{22}{21} t} \approx 0.409091 + 0.071375 e^{-\frac{22}{21}t}
\end{equation}
and
\begin{equation}
z(t) = \frac{31}{22} + \frac{21}{484} (13 \sqrt{5} - 63) \sqrt[21]{\frac{3 + 13 \sqrt{5}}{22}} e^{-\frac{22}{21} t} \approx 1.40909 - 1.49888 e^{-\frac{22}{21} t}.
\end{equation}
Now, if
\begin{equation}
t_2 = \frac{1}{22} \bigg( 20 \log(22) - 21 \log \bigg( \frac{484}{63 - 13 \sqrt{5}} \bigg) + \log(3 + 13 \sqrt{5}) \bigg) \approx 0.430724,
\end{equation}
we have
$$ x(t_2) = y(t_2) = z(t_2) = \frac{5}{11} \approx 0.454545,$$
and for $t_1\le t\le t_2$
$$\g_t(1,2)=y'(t)\in \hbox{sign}(0)=[-1,1].$$
Again, by uniqueness, the above is the unique solution to problem \eqref{CaseC} for $t_1 \leq t \leq t_2$.   Finally, for $t \geq t_2$,  the solution is constant and equal to $\frac{5}{11}$ (see Figure~\ref{ex3pointsfinal01}).

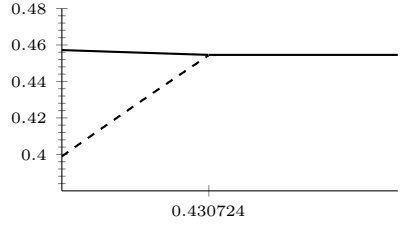
\begin{figure}[ht]
\centering
\begin{tikzpicture}
    \begin{axis}[
        xlabel={}, ylabel={},
        domain=0.376844:0.5, 
        samples=1000, 
        axis lines=middle, 
        axis line style={-},
        legend pos=north west, 
        width=6cm, 
        height=4cm, 
        ymin=0.38, ymax=0.48, 
        xtick={0.430724}, 
        ytick={ }, 
        xticklabel style={font=\tiny},
        yticklabel style={font=\tiny},
        minor tick num=4,
        xticklabel={\pgfmathprintnumber[fixed, precision=6]{\tick}},
    ]
    \addplot[black,thick, restrict x to domain=0.376844:0.430724] {0.409091 + 0.071375*exp(-(22/21)*x)};

    \addplot[black, thick,dashed,restrict x to domain=0.376844:0.430724] {1.40909 - 1.49888*exp(-(22/21)*x)};
    \addplot[black,  thick] coordinates {(0.430724, 5/11) (0.5, 5/11)};
    \end{axis}
\end{tikzpicture}
\caption{Example~\ref{ex3points}, case C. $a=10$, $b=1$, $c=1$. $x(t)=y(t)$ continuous line;  $z(t)$ dotted line; $t\gtrsim 0.376844$. After $t\approx 0.430724$,  $x(t)=y(t)=z(t).$  }
\label{ex3pointsfinal01}
\end{figure}
\hfill $\blacksquare$}
\end{example}

\begin{remark}\rm
Let us observe that in Case A of the above example, a signal  (the value of the function $u$)  equal to $1$ in the 1-vertex is diffused to the other  vertices (where the signal is null) in such a way that the signal increases in both   vertices  in the same quantity, even   though they are  not equally connected to the 1-vertex. This also happens if we make {\it stronger} the connection between the vertices $1$ and $2$ {\it for the $1$-Laplacian diffusion} (see the next Example~\ref{newex01}).

The above effect changes if the signal in the 1-vertex is much larger, as we see in Case B. Now this signal is   again  diffused to the other two vertices, but it is larger in the vertex nearest to the 1-vertex during a period of time (up to the moment in which signals at vertices $2$ and $3$ are equal).

In Case C, we see that the same initial signal as in the Case A, but with a {\it strong connection} between the vertex nearest to the 1-vertex, is diffused to the other vertices in such a way that it is larger in this nearest vertex (with a behaviour different to the Case B, now the signals at the vertices $1$ and $2$ will be equal first). This also happens even if we only make {\it stronger} the connection   between the vertices $1$ and $2$ {\it for the Laplacian diffusion} (see the next Example~\ref{newex01}).
\end{remark}

In the next example,   we consider the same graph as previously, but with two cases of different random walks for the two operators in problem~\eqref{eq:Cauchy1penex2rw}.

\begin{example}\label{newex01}\rm

\begin{figure}[ht]
\centering

\begin{tikzpicture}[line cap=round,line join=round,>=triangle 45,x=1.1cm,y=1.1cm]
\draw [line width=1pt] (-2,0)-- (0,0);
\begin{scriptsize}
\draw [fill=black] (-2,0) circle (2.5pt);
\draw[color=black] (-2,-0.35) node {\large $1$};
\draw [fill=black] (-1,0) circle (2.5pt);
\draw[color=black] (-1,-0.35) node {\large $2$};
\draw[color=black] (-1.5,0.25) node {$1$};
\draw[color=black] (-1.5,-0.25) node {$10$};
\draw [fill=black] (0,0) circle (2.5pt);
\draw[color=black] (0,-0.35) node {\large $3$};
\draw[color=black] (-0.5,0.25) node {$1$};
\draw[color=black] (-0.5,-0.25) node {$1$};
\end{scriptsize}
\end{tikzpicture}
\caption{Graph in Example~\ref{newex01}, Case A. Below the edges, the weights for $ m^1$ (corresponding to the $1$-Laplacian); above the edges, the weights for $ m^2$ (corresponding to the Laplacian).}
\label{fig02_3points2rw}
\end{figure}
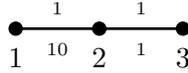

Case A. Consider the linear graph $G=(V,E)$ with weights shown in Figure~\ref{fig02_3points2rw}. Then, the random walks are given by
$$ m^1_{1} =  \delta_2 , \quad m^1_{2} = \frac{10}{11} \delta_1 +  \frac{1}{11} \delta_3,\quad m^1_{3} =\delta_2,$$
for which
$$ \nu_1(\{ 1 \}) = 10, \quad \nu_1(\{ 2 \}) = 11, \quad \nu_1(\{ 3 \}) = 1;$$
and
$$ m^2_{1} =  \delta_2 , \quad m^2_{2} =\frac12 \delta_1 +  \frac12  \delta_3,\quad m^2_{3} =\delta_2,$$
for which
$$ \nu_2(\{ 1 \}) = 1, \quad \nu_2(\{ 2 \}) = 2, \quad \nu_2(\{ 3 \}) = 1.$$
We now have that
$$ \mu(\{ 1 \}) = \frac{1}{10}, \quad \mu(\{ 2 \}) = \frac{2}{11}, \quad \mu(\{ 3 \}) = 1.$$
Let us call $x(t):=u(1,t)$, $y(t)=u(2,t)$ and $z(t)=u(3,t)$. Then,  the evolution problem
$$u_t = \Delta_1^{m_1} (u) + \mu\Delta_2^{m_2} (u)$$
can be written as the following system of ODEs
 \begin{equation}\label{masterECasoA}\left\{\begin{array}{l}
  x'(t)=\g_t(1,2)+ \frac{1}{10}(y(t)-x(t)); \\ \\
 y'(t)= - \frac{10}{11}  \g_t(1,2)+\frac{1}{11}\g_t(2,3)+\frac{2}{11}\left(\frac{1}{2}(x(t)-y(t)) + \frac{1}{2}(z(t) - y(t)) \right);\\ \\
 z'(t)=-  \g_t(2,3)+ y(t) - z(t)
 \end{array}\right.
\end{equation}
 for antisymmetric functions $\g_t(1,2), \g_t(2,3)$ satisfying
$$\g_t(1,2) \in \hbox{sign}(y(t) - x(t)), \quad \g_t(2,3)\in\hbox{sign}(z(t) - y(t)).$$
If we consider the initial condition $$u(0)=\1_{\{1 \}}, $$
as in the other examples, we have that there exists a time $t_1$ such that $x(t) > y(t) = z(t)$ for $0\leq t \leq t_1$. Then, \eqref{masterECasoA} can be written as
\begin{equation}
\left\{\begin{array}{l}
x'(t)= -1 + \frac{1}{10}(y(t) - x(t)); \\ \\
\frac{12}{11} y'(t)= \frac{10}{11}+ \frac{1}{11}(x(t) - y(t));\\ \\
z(t)=y(t),
\end{array}\right.
\end{equation}
with $$x(0)=1,\quad y(0)=z(0)=0,$$
and the solution is given by
$$x(t) = -5 + 6 e^{-\frac{11}{60}t}$$
and
$$y(t)=z(t)  = 5 - 5 e^{-\frac{11}{60}t}.$$
It is valid until the time $t_1=\frac{60}{11}\log{\frac{11}{10}}\approx 0.519874$. For $t \geq t_1$,  the solution is constant:
$$x(t)=y(t)=z(t)=\frac{5}{11}\approx 0.454545.$$

\begin{figure}[ht]
\centering\begin{tikzpicture}[line cap=round,line join=round,>=triangle 45,x=1.1cm,y=1.1cm]
\draw [line width=1pt] (-2,0)-- (0,0);
\begin{scriptsize}
\draw [fill=black] (-2,0) circle (2.5pt);
\draw[color=black] (-2,-0.35) node {\large $1$};
\draw [fill=black] (-1,0) circle (2.5pt);
\draw[color=black] (-1,-0.35) node {\large $2$};
\draw[color=black] (-1.5,0.25) node {$10$};
\draw[color=black] (-1.5,-0.25) node {$1$};
\draw [fill=black] (0,0) circle (2.5pt);
\draw[color=black] (0,-0.35) node {\large $3$};
\draw[color=black] (-0.5,0.25) node {$1$};
\draw[color=black] (-0.5,-0.25) node {$1$};
\end{scriptsize}
\end{tikzpicture}
\caption{Graph in Example~\ref{newex01}, Case B. Below the edges, the weights for $m^1$ (corresponding to the $1$-Laplacian); above the edges, the weights for $m^2$ (corresponding to the Laplacian).}
\label{fig02_3points2rwchanged}
\end{figure}
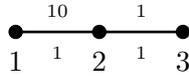

{\flushleft Case B.} Consider now the linear graph $G=(V,E)$ with weights shown in Figure~\ref{fig02_3points2rwchanged}. Then, the random walks are given by
$$ m^1_{1} =  \delta_2 , \quad m^1_{2} = \frac12 \delta_1 +  \frac12  \delta_3,\quad m^1_{3} =\delta_2,$$
for which
$$ \nu_1(\{ 1 \}) = 1, \quad \nu_1(\{ 2 \}) = 2, \quad \nu_1(\{ 3 \}) = 1;$$
and
$$m^2_{1} =  \delta_2 , \quad m^2_{2} =\frac{10}{11} \delta_1 +  \frac{1}{11} \delta_3,\quad m^2_{3} =\delta_2,$$
for which
$$ \nu_2(\{ 1 \}) = 10, \quad \nu_2(\{ 2 \}) = 11, \quad \nu_2(\{ 3 \}) = 1.$$
We have that
$$ \mu(\{ 1 \}) = 10, \quad \mu(\{ 2 \}) = \frac{11}{2}, \quad \mu(\{ 3 \}) = 1.$$
Now, at a first stage, between $0$ and $t_1\approx 0.187904$, the evolution problem
$$u_t = \Delta_1^{m_1} (u) + \mu\Delta_2^{m_2} (u)$$
with initial condition $$u(0)=\1_{\{1 \}},$$
is governed by
\begin{equation}
\left\{\begin{array}{l}
x'(t)= -1 + 10(y(t) - x(t)); \\ \\
 y'(t)= 5(x(t) - y(t))+\frac{1}{2}(z(t)-y(t);\\ \\
z'(t)=1+y(t)-z(t),
\end{array}\right.
\end{equation}
with $$x(0)=1,\quad y(0)=z(0)=0,$$
and the solution is given by
$$x(t)=-\frac{3}{40} + \left(\frac{43}{80}+\frac{341}{80\sqrt{769}}\right) e^{-\frac14 (33 + \sqrt{769}) t} +
 \left(\frac{43}{80}-\frac{341}{80\sqrt{769}}\right) e^{-\frac14 (33 - \sqrt{769}) t},$$
$$y(t)=\frac{1}{40} - \left(\frac{1}{80}+\frac{767}{80\sqrt{769}}\right) e^{-\frac14 (33 + \sqrt{769}) t} +
 \left(-\frac{1}{80}+\frac{767}{80\sqrt{769}}\right) e^{-\frac14 (33 - \sqrt{769}) t},$$
 $$z(t)=\frac{41}{40} + \left(-\frac{41}{80}+\frac{1193}{80\sqrt{769}}\right) e^{-\frac14 (33 + \sqrt{769}) t} -
 \left(\frac{41}{80}+\frac{1193}{80\sqrt{769}}\right) e^{-\frac14 (33 - \sqrt{769}) t}.$$
 At the time $t_1\approx 0.187904$, we have that $x(t_1)=y(t_1)\approx 0.264504$ and $z(t_1)\approx 0.206487$,  and the evolution continues attending a new system of ODEs.
\hfill $\blacksquare$
\end{example}

The next example concerns the behaviour of the solution to the problem associated to~\eqref{ProblemABenex} on a four-point graph.

\begin{example}\label{Ex1}{\rm Consider the graph $G=(V,E)$  with vertices $V = \{1,2,3,4\}$
and edges $E = \{(1,4), (1,2), (2,3), (3,4) \}$. To these edges, we assign the positive weights
$$w_{1,2} = a, \ w_{2,3} = b, \ w_{3,4} = c, \ w_{4,1} = d. $$
The graph is shown in Figure~\ref{fig:graphfourpoints}.   The invariant measure  $\nu$ is
$$ \nu(\{ 1 \}) = a + d, \quad \nu(\{ 2 \}) = a+ b, \quad \nu(\{ 3 \}) =b + c, \quad \nu(\{ 4 \}) = c + d,$$
and the random walk $m$ is given by
$$m_1 = \frac{a}{a+d}  \delta_2 + \frac{d}{a+d} \delta_4, \quad m_2 = \frac{a}{a+b}  \delta_1 + \frac{b}{a+b} \delta_3,$$ $$ m_3 = \frac{b}{b+c}  \delta_2 + \frac{c}{b+c} \delta_4, \quad  m_4 = \frac{c}{c+d}  \delta_3 + \frac{d}{c+d} \delta_1.$$
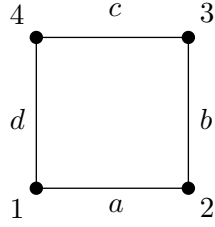
\begin{figure}[ht]
\begin{tikzpicture}[line cap=round,line join=round,,x=1.0cm,y=1.0cm]
\clip(-0.5,-0.5) rectangle (3,2.5);
\begin{scriptsize}
\draw[line width=1pt,fill=black] (0.,0.) circle (2pt);
\draw[line width=1pt,fill=black] (2.,0.) circle (2pt);
\draw[line width=1pt,fill=black] (0.,2.) circle (2pt);
\draw[line width=1pt,fill=black] (2.,2.) circle (2pt);
\draw[line width=0.5pt] (0.,0.)-- (2.,0.);
\draw[line width=0.5pt] (0.,0.)-- (0.,2.);
\draw[line width=0.5pt] (2.,0.)-- (2.,2.);
\draw[line width=0.5pt] (0.,2.)-- (2.,2.);
\draw (-0.45,-0.05) node[anchor=north west] {\large $1$};
\draw (2.05,-0.03) node[anchor=north west] {\large $2$};
\draw (2.05,2.55) node[anchor=north west] {\large $3$};
\draw (-0.45,2.53) node[anchor=north west] {\large $4$};
\draw (0.85,-0.05) node[anchor=north west] {\large $a$};
\draw (2.05,1.15) node[anchor=north west] {\large $b$};
\draw (0.85,2.55) node[anchor=north west] {\large $c$};
\draw (-0.45,1.15) node[anchor=north west] {\large $d$};
\end{scriptsize}
\end{tikzpicture}
\caption{ The weighted graph  of Example~\ref{Ex1}.}
\label{fig:graphfourpoints}
\end{figure}
For this random walk space we are going to consider an evolution problem that involves the $1$-Laplacian in the edges $(1,4)$ and $(2, 3)$, and the Laplacian in the edges $(1,2)$ and $(3,4)$. So we make the following partition on the random walk (attending to the nomenclature used previously),
$$A_1 = \{ 4 \}, \quad A_2 = \{ 3 \}, \quad A_3 = \{ 2 \}, \quad A_4 = \{ 1 \}$$
and
$$B_1 = \{ 2 \}, \quad B_2 = \{ 1 \}, \quad B_3 = \{ 4 \}, \quad B_4 = \{ 3 \}.$$
 Hence,
if we denote
$$x(t):= u(t,1), \quad y(t):= u(t,2), \quad z(t):= u(t,3), \quad w(t):= u(t,4),$$
the equation
$$u_t- \Delta_{1,A}^m(u) - \Delta_{2,B}^m(u) \ni 0$$
corresponds to the following system
\begin{equation}\label{SystODE1} \left\{ \begin{array}{llll} x'(t) = \frac{d}{a+d} \g_t(1,4) + \frac{a}{a+d} (y(t) - x(t)); \\[10pt] y'(t) = \frac{b}{a+b} \g_t(2,3) + \frac{b}{a+b} (x(t) - y(t)); \\[10pt] z'(t) =  -\frac{b}{b+c} \g_t(2,3) + \frac{c}{b+c} (w(t) - z(t))); \\[10pt] w'(t) = -\frac{d}{c+d} \g_t(1,4) + \frac{c}{c+d} (z(t) - w(t)) \end{array} \right.
\end{equation}
for antisymmetric functions $\g_t$ satisfying
$$\g_t(1,4)  \in {\rm sign}(w(t)-x(t)), \ \g_t(2,3)  \in {\rm sign}(z(t)-y(t)). $$
In general, observe that this system of ODEs breaks into two systems of two linear ODEs (the first two and the last two equations respectively), for which the associated 2x2 matrix has one eigenvalue negative and one equal to zero. Therefore, the solution will be a linear combination of exponentials and linear functions until a finite time at which we have $w(t) = x(t)$ or $z(t) = y(t)$. To illustrate this more explicitly, we will now solve this system of equations for a particular choice of weights and initial datum $u_0$.

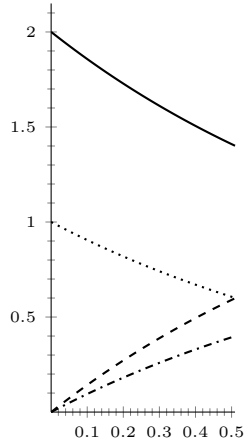
\begin{figure}[ht]
\centering
\begin{tikzpicture}
    \begin{axis}[
        xlabel={}, ylabel={},
        domain=0:0.52, 
        samples=1000, 
        axis lines=middle, 
        axis line style={-},
        legend pos=north west, 
        width=4cm, 
        height=7cm, 
        ymin=0, ymax=2.15, 
        xtick={0,0.1, 0.2,0.3,0.4,0.5,0.6}, 
        ytick={0,0.5, 1,1.5,2,2.5}, 
        xticklabel style={font=\tiny},
        yticklabel style={font=\tiny},
        minor tick num=4,
    ]
    \addplot[black,thick,restrict x to domain=0:0.51] {(3/2)*exp(-x) +1/2};
    \addplot[black, thick,dashed,restrict x to domain=0:0.51] {-(3/2)*exp(-x)+3/2};
    \addplot[black,thick,dotted,restrict x to domain=0:0.51] {exp(-x)};
    \addplot[black, thick,dashdotted ,restrict x to domain=0:0.51] {-exp(-x)+1};
    \end{axis}
\end{tikzpicture}
\caption{Example~\ref{Ex1}. $x(t)$ continuous line; $y(t)$ dashed line; $z(t)$ dotted line; $w(t)$ dashed-dotted line; $0\le t\lesssim 0.510826$.}
\label{ex4pointsplot1}
\end{figure}

Therefore, we take $a=b=c=d=1$ as weights and, for the initial datum, we take
$$u_0(1) = 2, \quad u_0(2) = 0, \quad u_0(3) = 1, \quad u_0(4) = 0,$$
which we rewrite as
\begin{equation}\label{tdheincon01} x(0) = 2, \quad y(0) = 0, \quad z(0) = 1, \quad w(0) = 0.
\end{equation}
Then, the system~\eqref{SystODE1} joint to the initial condition~\eqref{tdheincon01} is reduced, until the time
$$t_1 := \log(5/3)\approx 0.510826$$
at which $y$ and $z$ are equal, to the following system of ODEs
\begin{equation}\label{SystODE2} \left\{ \begin{array}{llll} x'(t) =  -\frac12 + \frac12 (y(t) - x(t)); \\[10pt] y'(t) =   \frac12  + \frac12 (x(t) - y(t)); \\[10pt] z'(t) =  -\frac12  + \frac12 (w(t) - z(t))); \\[10pt] w'(t) = \frac12 + \frac12 (z(t) - w(t)), \end{array} \right.
\end{equation}
joint to the initial condition~\eqref{tdheincon01}. In fact, it has a unique solution up to such time given by (see Figure~\ref{ex4pointsplot1})
$$x(t)=\frac32 e^{-t} +\frac12, \quad y(t)=-\frac32 e^{-t}+\frac32, \quad z(t)=e^{-t}, \quad w(t)=-e^{-t}+1.$$
Let us see how the evolution continues for $t > t_1$.  First, observe that for $t \geq t_1$ we have that $y(t) = z(t)$. We prove this by contradiction. Assume otherwise, and for simplicity, assume that $y$ and $z$ are not equal already in a neighbourhood of $\log(\frac53)$ from the right; the argument is similar for larger $t$. Observe that by the equation we have that $x(t)$ is decreasing and $w(t)$ is increasing. Thus, we have that the absolute value of the difference of any two given functions among $x,y,z,w$ is bounded from above by~$1$ (strictly for $t \geq t_1$). Moreover, by the equation the functions $x,y,z,w$ are Lipschitz, and consequently if $y(t) < z(t)$ this means that this property holds on an open interval (similarly if $y(t) > z(t)$). However, whenever $y(t) < z(t)$, it holds that
$$y'(t) = \frac{1}{2} + \frac{1}{2} (x(t) - y(t)) > 0$$
and
$$z'(t) = -\frac{1}{2} (w(t) - z(t)) < 0,$$
so $y$ is increasing and $z$ is decreasing; a similar argument works for $y(t) > z(t)$. Hence, it is not possible that we have $y(t) \neq z(t)$ on any open interval, and consequently $y(t) = z(t)$ for $t \geq t_1$.

With this property in mind, let us find the system of ODEs corresponding to
$$u_t- \Delta_{1,A}^m(u) - \Delta_{2,B}^m(u) \ni 0$$
for $t > t_1$. Recall that formally it is the system \eqref{SystODE1}; for $t < t_1$ it simplified to the system \eqref{SystODE2}, and our current goal is to find and solve a similar system of ODEs with explicitly given coefficients. Since $y$ and $z$ are Lipschitz with $y(t) = z(t)$, it follows that $y'(t) = z'(t)$ a.e., and using the equation we get that
\begin{equation}
\frac12 \g_t(2,3) + \frac12 (x(t) - y(t)) = -\frac12 \g_t(2,3) + \frac12 (w(t) - z(t)),
\end{equation}
which implies that
\begin{equation}
  \g_t(2,3) = \frac12(w(t) - x(t))
\end{equation}
(observe that it belongs to $[-1,1]=\hbox{sign(0)}={\rm sign}(z(t)-y(t))$).
Plugging this to the system \eqref{SystODE1}, we see that
$$ y'(t) = z'(t)  =  \frac14(w(t) - x(t))  + \frac12 (x(t) - y(t)) = \frac14(x(t) + w(t) - 2y(t)) $$
and keeping in mind that $\g_t(1,4)  \in {\rm sign}(w(t)-x(t))$, we arrive to the following system of equations with three variables $x,y,w$
\begin{equation}\left\{ \begin{array}{llll} x'(t) =  -\frac12 + \frac12 (y(t) - x(t)); \\[10pt] y'(t) = \frac14(x(t) + w(t) - 2y(t)); \\[10pt] w'(t) = \frac12 + \frac12 (y(t) - w(t)) \end{array} \right.
\end{equation}
with initial data
$$x(\log(5/3)) = \frac75; \quad y(\log(5/3)) = \frac35; \quad w(\log(5/3)) = \frac25.$$
Notice that by appropriately summing up the equations we get that $(x+w+2y)' = 0$, therefore $x + w + 2y \equiv 3$. From this, we may compute $y(t)$, and the system splits into separate equations for $x$ and $w$; the unique solution (after the time $t_1$) is given by
$$x(t)= \frac14 e^{-t} +\frac{\sqrt{15}}{2} e^{-\frac12t} - \frac14, \quad y(t)= -\frac14 e^{-t} + \frac34, \quad w(t)= \frac14 e^{-t} - \frac{\sqrt{15}}{2} e^{-\frac12t} + \frac{7}{4}.$$
These solutions are valid until the time
$$t_2 := \log(15/4) \approx 1.32176,$$
at which point we have $x(t) = w(t)$ (see Figure~\ref{ex4pointsplot1G01});
\begin{figure}[ht]
\centering
\begin{tikzpicture}
    \begin{axis}[
        xlabel={}, ylabel={},
        domain=0.510826:1.32176, 
        samples=1000, 
        axis lines=middle, 
        axis line style={-},
        legend pos=north west, 
        width=4cm, 
        height=7cm, 
        ymin=0.2, ymax=1.5, 
        xtick={0.4,0.6,0.8,1.0,1.2,1.4 }, 
        ytick={0.2,0.4,0.6,0.8,1.0,1.2,1.4,1.6}, 
        xticklabel style={font=\tiny},
        yticklabel style={font=\tiny},
        minor tick num=4,
    ]
    \addplot[black,thick,restrict x to domain=0.510826:1.32176] {(1/4)*exp(-x) +((15^(1/2))/2)*exp(-(1/2)*x) - 1/4};
    \addplot[black,thick,dotted,restrict x to domain=0.510826:1.32176] {-(1/4)*exp(-x) +  3/4};
    \addplot[black, thick,dashdotted ,restrict x to domain=0.510826:1.32176] {(1/4)*exp(-x)  - ((15^(1/2))/2)*exp(-(1/2)*x)  + 7/4};
    \end{axis}
\end{tikzpicture}
\caption{Example~\ref{Ex1}. $x(t)$ continuous line; $y(t)=z(t)$ dotted line; $w(t)$ dashed-dotted line; $0.510826\lesssim t\lesssim 1.32176$.}
\label{ex4pointsplot1G01}
\end{figure}
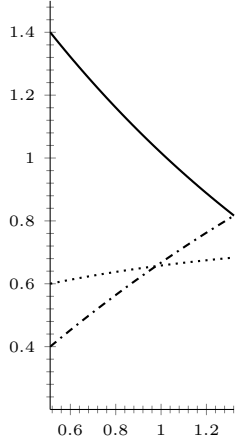
 throughout the interval $(t_1,t_2)$, we always have $w(t) < x(t)$. Arguing as above, we see that $x(t) = w(t)$ for $t > t_2$; therefore, since $x$ and $w$ are Lipschitz, it follows that $x'(t) = w'(t)$ a.e., and using the equation we get that
\begin{equation}
\frac12 \g_t(1,4) + \frac12 (x(t) - y(t)) = -\frac12 \g_t(1,4) + \frac12 (w(t) - z(t)),
\end{equation}
which, taking into account that $x(t) = w(t)$ and $y(t) = z(t)$, implies that
\begin{equation}
\g_t(1,4) = 0.
\end{equation}
Notice that for $t > t_2$ we also have
\begin{equation}
\g_t(2,3) = \frac12(w(t) - x(t)) = 0.
\end{equation}
Thus, for $t > t_2$, inserting this information to the system \eqref{SystODE1}, we arrive to the following system of equations with two variables $x,y$
\begin{equation}\left\{ \begin{array}{llll} x'(t) = \frac12 (y(t) - x(t)); \\[10pt] y'(t) = \frac12(x(t) - y(t)) \end{array} \right.
\end{equation}
with initial data
$$x(\log(15/4)) = \frac{49}{60}, \quad y(\log(15/4)) = \frac{41}{60}.$$
The unique solution (after the time $t_2$) is given by (see Figure~\ref{ex4pointsplot1G02}):
$$x(t)= \frac14 e^{-t} + \frac34, \quad y(t)= -\frac14 e^{-t} + \frac34.$$
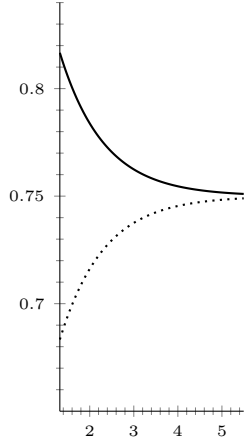
\begin{figure}[ht]
\centering
\begin{tikzpicture}
    \begin{axis}[
        xlabel={}, ylabel={},
        domain=1.32176:5.5, 
        samples=1000, 
        axis lines=middle, 
        axis line style={-},
        legend pos=north west, 
        width=4cm, 
        height=7cm, 
        ymin=0.65, ymax=0.84, 
        xtick={1,2,3,4,5,6}, 
        ytick={ }, 
        xticklabel style={font=\tiny},
        yticklabel style={font=\tiny},
        minor tick num=4,
        xticklabel={\pgfmathprintnumber[fixed, precision=6]{\tick}},
    ]
    \addplot[black,thick, restrict x to domain=1.32176:5.5] {(1/4)*exp(-x) + 3/4};

    \addplot[black, thick,dotted,restrict x to domain=1.32176:5.5] {-(1/4)*exp(-x) + 3/4};
    \end{axis}
\end{tikzpicture}
\caption{Example~\ref{Ex1}.   $x(t)=w(t)$  continuous line; $y(t)=z(t)$ dotted line; $t\gtrsim 1.32176$.}
\label{ex4pointsplot1G02}
\end{figure}
Note that the solution of this system converges to the mean of the initial data and has an infinite extinction time.   Formally speaking,  due to the fact that the equation selects only some directions in the random walk, the graph effectively splits into two pieces; the sets $\{ 1,4 \}$ and $\{ 2,3 \}$. Within these sets, the evolution is primarily governed by the $1$-Laplacian (with some diffusion effect between the two sets), until the point where the values within the two sets are identical, and then the two pieces move together at an exponential rate towards the mean of the initial data.
\hfill $\blacksquare$
}\end{example}

The last example concerns the behaviour of the solution to problem \eqref{ProblemAB} on a different four-point graph. Note that depending on the choice of the initial data,  we may end up with a finite or infinite   time   to get the mean of the initial data.

\begin{example}\label{Ex12}\rm Consider now the graph $G=(V,E)$  with vertices $V = \{1,2,3,4\}$
and edges $\quad E = \{(1,4), (1,2), (2,3) \}.$
To these edges, we assign the positive weights
$$w_{1,2} = a, \ w_{2,3} = b,  \ w_{4,1} = d. $$
Now we are considering the linear graph shown in Figure~\ref{fig:graphfourpointsbis}.   The invariant measure  $\nu$ is
$$ \nu(\{ 1 \}) = a + d, \quad \nu(\{ 2 \}) = a+ b, \quad \nu(\{ 3 \}) =b, \quad \nu(\{ 4 \}) = d,$$
and the random walk $m$ is given by
$$m_1 = \frac{a}{a+d}  \delta_2 + \frac{d}{a+d} \delta_4, \quad m_2 = \frac{a}{a+b}  \delta_1 + \frac{b}{a+b} \delta_3,\quad m_3 =   \delta_2  , \quad  m_4 =   \delta_1.$$
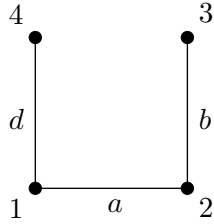
\begin{figure}[ht]
\begin{tikzpicture}[line cap=round,line join=round,,x=1.0cm,y=1.0cm]
\clip(-0.5,-0.5) rectangle (3,2.5);
\begin{scriptsize}
\draw[line width=1pt,fill=black] (0.,0.) circle (2pt);
\draw[line width=1pt,fill=black] (2.,0.) circle (2pt);
\draw[line width=1pt,fill=black] (0.,2.) circle (2pt);
\draw[line width=1pt,fill=black] (2.,2.) circle (2pt);
\draw[line width=0.5pt] (0.,0.)-- (2.,0.);
\draw[line width=0.5pt] (0.,0.)-- (0.,2.);
\draw[line width=0.5pt] (2.,0.)-- (2.,2.);
\draw (-0.45,-0.05) node[anchor=north west] {\large $1$};
\draw (2.05,-0.03) node[anchor=north west] {\large $2$};
\draw (2.05,2.55) node[anchor=north west] {\large $3$};
\draw (-0.45,2.53) node[anchor=north west] {\large $4$};
\draw (0.85,-0.05) node[anchor=north west] {\large $a$};
\draw (2.05,1.15) node[anchor=north west] {\large $b$};
\draw (-0.45,1.15) node[anchor=north west] {\large $d$};
\end{scriptsize}
\end{tikzpicture}
\caption{ The linear weighted graph  of Example~\ref{Ex12}.}
\label{fig:graphfourpointsbis}
\end{figure}

For this random walk space we are going to consider an evolution problem that involves the $1$-Laplacian in the edges $(1,4)$ and $(2, 3)$, and the Laplacian in the edge  $(1,2)$.  Hence,
if we denote
$$x(t)= u(t,1), \quad y(t)= u(t,2), \quad z(t)= u(t,3), \quad w(t)= u(t,4),$$
the equation
$$u_t- \Delta_{1,A}^m(u) - \Delta_{2,B}^m(u) \ni 0$$
corresponds to the following system of ODEs
\begin{equation}\label{SystODE1d} \left\{ \begin{array}{llll} x'(t) = \frac{d}{a+d} \g_t(1,4) + \frac{a}{a+d} (y(t) - x(t)); \\[10pt] y'(t) = \frac{b}{a+b} \g_t(2,3) + \frac{b}{a+b} (x(t) - y(t)); \\[10pt] z'(t) =  - \g_t(2,3); \\[10pt] w'(t) = -  \g_t(1,4) , \end{array} \right.
\end{equation}
for antisymmetric functions $\g_t$ satisfying
$$\g_t(1,4)  \in {\rm sign}(w(t)-x(t)) \quad \mbox{and} \quad \g_t(2,3)  \in {\rm sign}(z(t)-y(t)). $$
{\flushleft Case A.}   Consider the weights $a=b=c=1$ and  the initial datum
$$x(0) = 1, \quad y(0) = 1, \quad z(0) = 0, \quad w(0) = 0.$$
We claim that, up to a time $t_1$, we have $x(t) > w(t)$ and $y(t) > z(t)$. Then,
we arrive to the following system of ODEs
\begin{equation}\label{SystODE3} \left\{ \begin{array}{llll} x'(t) = -\frac12 + \frac12 (y(t) - x(t)); \\[10pt] y'(t) = - \frac12  + \frac12(x(t) - y(t)); \\[10pt] z'(t) =  1; \\[10pt] w'(t) = 1, \end{array} \right.
\end{equation}
 whose solution is (see Figure~\ref{ex4points05})
$$x(t)=  1- \frac{t}{2}, \quad  y(t)=  1- \frac{t}{2}, \quad z(t) = t, \quad w(t)=t.$$
  This solution is valid until the time  $t_1 = \frac23$. Afterwards,
$$x(t)=  y(t) = z(t) = w(t) =  \frac{2}{3}  \quad \hbox{for } \  t \geq  \frac23.$$
\begin{figure}[ht]
\centering
\begin{tikzpicture}
    \begin{axis}[
        xlabel={}, ylabel={},
        domain=0:1.6, 
        samples=1000, 
        axis lines=middle, 
        axis line style={-},
        legend pos=north west, 
        width=4cm, 
        height=7cm, 
        ymin=0, ymax=1.1, 
        xtick={}, 
        ytick={ }, 
        xticklabel style={font=\tiny},
        yticklabel style={font=\tiny},
        minor tick num=4,
        xticklabel={\pgfmathprintnumber[fixed, precision=6]{\tick}},
    ]
    \addplot[black,thick, restrict x to domain=0:0.66667] {1- (1/2)*x};

    \addplot[black, thick,dashdotted,restrict x to domain=0:0.66667] {x};
    \addplot[black,  thick] coordinates {(2/3, 2/3) (1.6, 2/3)};
    \end{axis}
\end{tikzpicture}
\caption{Example~\ref{Ex12}. $x(t)=y(t)$ continuous line; $z(t)=w(t)$ dashed-dotted line.
After $t=2/3$, $x(t)=y(t)=z(t)=w(t)$. }
\label{ex4points05}
\end{figure}
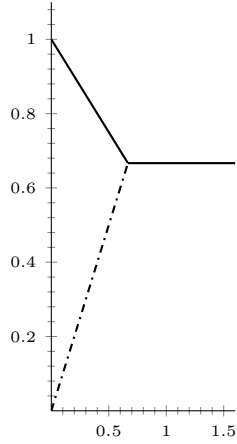

{\flushleft Case B.}   Consider the same weights $a=b=d=1$, but with  the initial datum
$$x(0) = 2, \quad y(0) = 0, \quad   z(0)  = 1, \quad   w(0) = 0.$$
We claim that, up to a time $t_1$, we have $x(t) > w(t)$ and $z(t) > y(t)$. Then, for $0\leq t \leq t_1$
we arrive to the following system of ODEs
\begin{equation}\label{SystODE2d} \left\{ \begin{array}{llll} x'(t) =  -\frac12 + \frac12 (y(t) - x(t)); \\[10pt] y'(t) =   \frac12  + \frac12 (x(t) - y(t)); \\[10pt] z'(t) = - 1; \\[10pt] w'(t) = 1, \end{array} \right.
\end{equation}
whose solution is  given by (see~Figure~\ref{ex4pointsplotbis}):
$$x(t)=\frac32 e^{-t} +\frac12, \quad y(t)=-\frac32 e^{-t}+\frac32, \quad z(t)=1-t, \quad w(t)=t.$$
Thus, $y(t_1) = z(t_1)$ if $t_1  = W_0\left(\frac32\sqrt{e}\right)-\frac12 \approx 0.453295$, where $W_0$ is the principal branch of the Lambert $W$ function.

\begin{figure}[ht]
\centering
\begin{tikzpicture}
    \begin{axis}[
        xlabel={}, ylabel={},
        domain=0:0.453295, 
        samples=1000, 
        axis lines=middle, 
        axis line style={-},
        legend pos=north west, 
        width=4cm, 
        height=7cm, 
        ymin=0, ymax=2.1, 
        xtick={ }, 
        ytick={ }, 
        xticklabel style={font=\tiny},
        yticklabel style={font=\tiny},
        minor tick num=4,
    ]
    \addplot[black,thick,restrict x to domain=0:0.453295] {(3/2)*exp(-x) + 1/2};
    \addplot[black, thick,dashed,restrict x to domain=0:0.453295] {-(3/2)*exp(-x)+3/2};
    \addplot[black,thick,dotted,restrict x to domain=0:0.453295] {1-x};
    \addplot[black, thick,dashdotted ,restrict x to domain=0:0.453295] {x};
    \end{axis}
\end{tikzpicture}
\caption{Example~\ref{Ex12}. $x(t)$ continuous line; $y(t)$ dashed line; $z(t)$ dotted line; $w(t)$ dashed-dotted line; $0\le t\lesssim 0.453295$.}
\label{ex4pointsplotbis}
\end{figure}
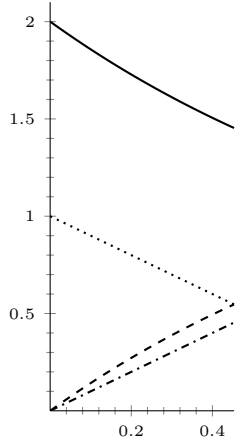

Let us see how the solution continues. For $t \geq t_1$, arguing as in the previous example, we see that $y(t) = z(t)$. Then, by comparing the expressions for $y'$ and $z'$, we see that
\begin{equation}
\g_t(2,3) = \frac13 (y(t) - x(t)),
\end{equation}
and consequently the system \eqref{SystODE1d} becomes
\begin{equation} \left\{ \begin{array}{llll} x'(t) =  -\frac12 + \frac12 (y(t) - x(t)); \\[10pt] y'(t) =  \frac13 (x(t) - y(t)); \\[10pt] w'(t) = 1 \end{array} \right.
\end{equation}
with initial data
\begin{equation}\begin{array}{l}
\displaystyle x(t_1) =W_0\left(\frac32\sqrt{e}\right)+\frac12 \approx 1.4533,\\
\\
\displaystyle  y(t_1) =\frac32-W_0\left(\frac32\sqrt{e}\right) \approx 0.546704,\\
\\
\displaystyle z(t_1)= W_0\left(\frac32\sqrt{e}\right)-\frac12 \approx 0.453295,\end{array}
\end{equation}
 that is,
$$ x(t_1) =t_1+1,\quad  y(t_1) =1-t_1  ,\quad z(t_1)=t_1.$$
Its unique solution is given by
 $$\begin{array}{l}
\displaystyle x(t)= \frac35\left(2W_0\left(\frac32\sqrt{e}\right)-\frac25\right)
e^{ -\frac56\left(t-t_1\right)}-\frac15t+\frac{16}{25},
\\ \\
\displaystyle
y(t) =-\frac25\left(2W_0\left(\frac32\sqrt{e}\right)-\frac25\right)e^{ -\frac56\left(t-t_1\right)}-\frac15t+\frac{31}{25},\\
\\ w(t)=t,
\end{array}$$
 until the time $$t_2 = \frac{18}{15}W_0\left(
 \frac{-e^{\frac56 W_0\left(\frac32\sqrt{e}\right)} + 5 e^{\frac56 W_0\left(\frac32\sqrt{e}\right)} W_0\left(\frac32\sqrt{e}\right)}{6 e^{\frac{31}{36}}}\right) + \frac{8}{15} \approx 1.00786,$$ at which $x(t_2) = w(t_2)$.
 Observe that, as we wanted, $\frac13(y(t)-x(t))\in[-1,1]$.

 For $t \geq t_2$, again arguing as in the previous example, we have that $$x(t) = w(t)\ \hbox{ and }\   y(t) = z(t).$$ The formula for $\g_t(2,3)$ remains unchanged, and by comparing the expressions for $x'$ and $w'$ we see that
\begin{equation}
\g_t(1,4) = \frac{1}{3} (x(t) - y(t)),
\end{equation}
and consequently the system \eqref{SystODE1d} becomes
\begin{equation}
\left\{ \begin{array}{llll} x'(t) = \frac13 (y(t) - x(t)); \\[10pt] y'(t) =  \frac13 (x(t) - y(t)) \end{array} \right.
\end{equation}
with initial data
\begin{equation}
x(t_1)=t_2 \approx 1.00786, \quad y(t_1) \approx 0.658806.
\end{equation}
The unique solution is given by
\begin{equation}
x(t) \approx 0.833333 + 0.341719 e^{- \frac23 t}, \quad y(t) \approx 0.833333 - 0.341719 e^{- \frac23 t}
\end{equation}
which has infinite extinction time.  It is easy to see that $\frac{1}{3} (x(t) - y(t))\in[-1,1]$. Observe that the mean value of the solution is equal to $\frac56\approx 0.833333$, the same as for the initial data. \hfill $\blacksquare$
\end{example}

\subsection{Nonlocal problems  in $\mathbb{R}^N$}

In the case of nonlocal problems  in the Euclidean space, that is, for the random walk space given in Example \ref{example.nonlocalJ},  we obtain the following consequence of our general results. Suppose we have two radially symmetric kernels $J,G: \R^N \rightarrow [0, +\infty)$ with
$$\int_{\R^N} J(\xi) \, d \xi = \int_{\R^N} G(\xi) \, d \xi =1.$$
In this case, the corresponding random walk spaces $[\R^N, d, m^J, \mathcal{L}^N]$ and $[\R^N, d, m^G, \mathcal{L}^N]$ have the same invariant and reversible measure. Hence, the Radon-Nikodym derivative $\mu$ is the identity, and it satisfies the assumption \eqref{ASSUMP1}. Then,  as a consequence of Theorem \ref{0942cor} and Theorem \ref{0942corI} respectively, we have the following results.

\begin{theorem}\label{0942corNonLocal}
Let $T>0$ and assume that $p \geq 2$. For any $u_0\in L^2(\R^N,\mathcal{L}^N)$ and $f\in L^2(0,T;L^2(\R^N, \mathcal{L}^N))$ the following problem has a unique strong solution $u=u_{u_0,f}$:
$$\left\{\begin{array}{lll}
\displaystyle u_t(t,x) -\int_{\R^N} \frac{u(t,y) -u(t,x)}{\vert u(t,y) -u(t,x) \vert} \, J(x-y) \, dy \\[10pt]  \displaystyle \qquad\quad\, -\int_{\R^n} \vert u(t,y) - u(t,x) \vert^{p-2} (u(t,y) - u(t,x)) \, G(y-x) \, dy \ni f(t,x)&\hbox{on } [0,T]
\\[10pt]
u(0,x)=u_0(x).
\end{array}
\right.
$$
\end{theorem}

\begin{theorem}\label{0942corNewNonLocal}  Let $T>0$. For any $u_0\in L^2(\R^N,  \mathcal{L}^N)$ and $f\in L^2(0,T;L^2(\R^N,  \mathcal{L}^N))$ the following problem has a unique strong solution $u=u_{u_0,f}$:
$$\left\{\begin{array}{lll}
\displaystyle u_t(t,x) -\int_{\R^N} \frac{u(t,y) -u(t,x)}{\vert u(t,y) -u(t,x) \vert} \, J(x-y) \, dy \\[10pt]  \displaystyle \qquad\quad\, -\int_{\R^n}\frac{u(t,y) -u(t,x)}{\vert u(t,y) -u(t,x) \vert} \, G(y-x) \, dy \ni f(t,x)&\hbox{on } [0,T]
\\[10pt]
u(0,x)=u_0(x).
\end{array}
\right.
$$
\end{theorem}

 In a similar manner, we recover results corresponding to the $(q,p)$-Laplace equation with $q \leq \frac{p}{p-1} \leq 2 \leq p$ and to the case of a partition of the random walk, as a consequence of Theorems \ref{thm:qplaplacegradientflow} and \ref{0942corNew} respectively.

{\flushleft \bf Acknowledgements.} The first author acknowledges support of the Austrian Science Fund (FWF) grant 10.55776/ESP88. The second and third authors  have been partially supported  by  Grant PID2022-136589NB-I00 funded by MCIN/AEI/10.13039/501100011 033 and FEDER and by Grant RED2022-134784-T funded by MCIN/AEI/10.13039/501 100011033. For the purpose of open access, the authors have applied a CC BY public copyright licence to any Author Accepted Manuscript version arising from this submission.

\

{\bf Data Availability.} Data sharing is not applicable to this article as no datasets were generated or
analyzed during the current study.

\

{\bf Conflict of interest.} The authors have no Conflict of interest to declare that are relevant to the content of this
article.

\end{document}